\documentclass[a4paper,12pt]{amsart}
\newcommand{\emailadd}{2} 
\usepackage{amssymb}
\usepackage{amscd}
\usepackage{ifthen}
\newcommand{\newc}{\newcommand}
\newc{\PoinEinFlatOp}[1]{
  \makebox[0mm][l]
   {
    \hspace{-1.5ex}
    \raisebox{1.9ex}{\upshape .}
   }
  \delta{\rule{0mm}{1.7ex}}^{\mbox{\upshape\tiny 0}}_{#1}}
\newc{\sumjl}{\sum_{\ell=1}^r{j_{\ell}}}
\newc{\jrg}{\underline{j}{}_{\rule{0.20ex}{0mm}\rule{0mm}{1.6ex}r}}
\newc{\jr}{\underline{j}{}_{\rule{0.05ex}{0mm}\rule{0mm}{1.20ex}r}}
\newc{\djr}{\delta_{\jr}}
\newc{\LDD}{\ell_2}
\newc{\stack}[2]{{\stackrel{\scriptscriptstyle{#1}}{#2}}\phantom{}}
\newc{\sffo}{
\rule{0mm}{2.45ex}%
\begin{picture}(0,0)
\put(2.6,6.9){$\stack{o}{\rule{0mm}{0mm}}$}
\end{picture}
L{}
}
\newc{\tfsff}[1]{\sffo_{#1}}     
\newc{\tfsffup}[1]{\sffo{}^{#1}} 
\newc{\tfsffgen}{\sffo{}}         
\newc{\indind}{s}                 
\newc{\tdKsOne}
  {\tilde{\delta}_{\mbox{\tiny $K_s$}}%
                 ^{\hspace{0.2ex}\raisebox{.3ex}{\tiny $1$}}
  }
\newc{\tdKsTwo}
  {\tilde{\delta}_{\mbox{\tiny $K_s$}}%
                 ^{\hspace{0.2ex}\raisebox{.3ex}{\tiny $2$}}
  }
\newc{\tpsiKsOne}
  {\tilde{\psi}_{K_s}%
               ^{\hspace{0.2ex}\raisebox{.3ex}{\tiny $1$}}
  }
\newc{\tpsiKsTwo}
  {\tilde{\psi}_{K_s}%
               ^{\hspace{0.2ex}\raisebox{.3ex}{\tiny $2$}}
  }
\def\frak{\mathfrak}
\def\Bbb{\mathbb}
\def\Cal{\mathcal}
\newc{\nn}[1]{(\ref{#1})}
\newc{\al}{\alpha}
\newc{\be}{\beta}
\newc{\ga}{\gamma}
\newc{\de}{\delta}
\newc{\ep}{\varepsilon}
\newc{\ze}{\zeta}
\newc{\ka}{\kappa}
\newc{\la}{\lambda}
\newc{\imu}{\bar\mu}
\newc{\irho}{\bar\rho}
\newc{\si}{\sigma}
\newc{\isigma}{\bar\si}
\renewcommand{\phi}{\varphi}
\newc{\ph}{\varphi}
\newc{\ps}{\psi}
\newc{\om}{\omega}
\newc{\De}{\Delta}
\newc{\La}{\Lambda}
\newc{\Ga}{\Gamma}
\newc{\Ph}{\Phi}
\newc{\Om}{\Omega}
\newc{\Si}{\Sigma}
\newc{\trho}{\tilde{\rho}}
\newc{\drho}{d^{\tilde{\rho}}}

\newc{\aR}{\mbox{\boldmath{$R$}}}
\newc{\aS}{\mbox{\boldmath{$S$}}}
\newc{\aDeR}{\mbox{\boldmath{$U$}}_B{}^P{}_C{}^Q}
\newc{\aDe}{\mbox{\boldmath$\Delta$}}
\newc{\aNd}{\mbox{\boldmath$\nabla$}}
\newc{\aK}{\mbox{\boldmath{$K$}}}
\newc{\aL}{\mbox{\boldmath{$L$}}}
\newc{\BoldUpm}{\mbox{\upshape\textbf{m}}}
\newc{\BoldUpmSub}{\mbox{\upshape\tiny\textbf{m}}}
\newc{\BoldUpmD}{\BoldUpm_{\mbox{\upshape\tiny D}}}
\newc{\BoldUpmN}{\BoldUpm_{\mbox{\upshape\tiny N}}}
\newc{\BoldUpmZ}{\BoldUpm_0}
\newtheorem{theorem}{Theorem}[section]
\newtheorem{lemma}[theorem]{Lemma}
\newtheorem{proposition}[theorem]{Proposition}
\newtheorem{corollary}[theorem]{Corollary}
\newtheorem{hypotheses}[theorem]{Hypotheses}
\newtheorem{condition}[theorem]{Condition}

\theoremstyle{definition}
\newtheorem{definition}[theorem]{Definition}

\newtheorem{remark}[theorem]{Remark}
\newtheorem{problem}[theorem]{Problem}
\newc{\llhd}{{\triangleleft}}

\newc{\gf}{{\frak g}}
\newc{\q}{{\frak q}}
\newc{\p}{{\frak p}}
\newc{\tg}{{\tilde {\frak g}}}
\newc{\tq}{{\tilde {\frak q}}}
\newc{\tp}{{\tilde {\frak p}}}

\newc{\tw}{\mbox{$\tilde{w}$}}

\newc{\ca}{{\Cal A}}
\newc{\ce}{{\Cal E}}
\newc{\ice}{\bar\ce}
\newc{\cf}{{\Cal F}}
\newc{\cg}{{\Cal G}}
\newc{\cN}{{\Cal N}}
\newc{\cv}{{\Cal V}}
\newc{\cp}{{\Cal P}}
\newc{\cw}{{\Cal W}}

\newc{\ct}{{\Cal T}}
\newc{\cT}{\mathcal{T}}
\newc{\tca}{{\tilde{\Cal A}}}
\newc{\tce}{{\tilde{\Cal E}}}
\newc{\tcg}{{\tilde{\Cal G}}}
\newc{\tct}{{\tilde{\Cal T}}}

\newc{\bW}{{\Bbb W}}
\newc{\bV}{{\Bbb V}}

\newc{\nd}{\nabla}
\newc{\Ps}{\Psi}
\newc{\Up}{\Upsilon}

\newc{\End}{\operatorname{End}}
\newc{\im}{\operatorname{im}}
\newc{\FSAOp}[1]{\delta_{#1}^{T}}
\newc{\OpOne}{O_1}
\newc{\OpTwo}{O_2}
\newc{\BasicOp}{\mbox{\textit{Op}}}
\newc{\OurOp}{\psi}             
\newc{\LowTrac}[2]{\cp^{k}{}_{A_{#1}\cdots A_{#2}}}
\newc{\LowTracPsi}[3]{\Psi^{k}{}_{A_{#1}\cdots A_{#2}}{}^{#3}}
\newc{\ConFlatOp}[1]{\delta^0_{#1}}
\newc{\critop}{\delta_{\lfloor(n-2)/2\rfloor,\,\lfloor(n+1)/2\rfloor}}
\newc{\POneAB}{P^1{}_{A_1\ldots A_k}{}^E}       
\newc{\psiOp}{\psi}                        
\newc{\psigw}[1]{\psiOp_{\g,#1}}            
\newc{\psighw}[1]{\psiOp_{\hat{\g},#1}}      
\newc{\PTh}{{\mbox{\sf P}_3}}              
\newc{\IPan}{\bar{P}}                      
\newc{\tildedeltaOneTwo}{\tilde{\delta}_{1,2}}
\newc{\tildedeltaOneThree}{\tilde{\delta}_{1,3}}
\newc{\tildedeltaTwoThree}{\tilde{\delta}_{2,3}}
\newc{\tildedeltaJ}{\tilde{\delta}_J}
\newc{\tildedeltaJOne}
  {\tilde{\delta}{}_{\mbox{\tiny $J$}}%
                   ^{\raisebox{0.4ex}{\tiny\hspace{0.45ex}$1$}}
  }
\newc{\tildedeltaJTwo}
  {\tilde{\delta}{}_{\mbox{\tiny $J$}}%
                   ^{\raisebox{0.4ex}{\tiny\hspace{0.2ex}$2$}}
  }
\newc{\tildedeltaJPK}{\tilde{\delta}_{J+k}}
\newc{\tildedeltaJPKOne}
  {\tilde{\delta}
   {}_{\mbox{\tiny $J\hspace{-0.5ex}+\hspace{-0.5ex}k$}}%
     ^{\raisebox{0.4ex}{\tiny\hspace{0.45ex}$1$}}
  }
\newc{\tildedeltaJPKTwo}
  {\tilde{\delta}
   {}_{\mbox{\tiny $J\hspace{-0.5ex}+\hspace{-0.5ex}k$}}%
     ^{\raisebox{0.4ex}{\tiny\hspace{0.45ex}$2$}}
  }
\newc{\tildedeltaJk}{\tilde{\delta}_{J,k}}
\newc{\tildedeltaJkg}[1]{\tilde{\delta}_{J,k,\g,#1}}
\newc{\tildedeltaJkgh}[1]{\tilde{\delta}_{J,k,\hat{\g},#1}}

\newc{\deltaOneTwo}{\delta_{1,2}}
\newc{\deltaOneThree}{\delta_{1,3}}
\newc{\deltaTwoThree}{\delta_{2,3}}
\newc{\deltaBGK}{\delta^{BG}_{K}}
\newc{\deltaBG}{\delta^{BG}}
\newc{\CQOp}{P_3^{\mbox{\tiny\it CQ}}}
\newc{\CQOpg}{P_{3,\g}^{CQ}}
\newc{\deltaJ}{\delta_J}
\newc{\deltaJk}{\delta_{J,k}}
\newc{\djonejtwo}{\deltaJk}
\newc{\deltaJkg}[1]{\delta_{J,k,\g,#1}}
\newc{\deltaJkgh}[1]{\delta_{J,k,\hat{\g},#1}}
\newc{\OJk}{O_{J,k}}
\newc{\deltajr}{\delta_{j_1,\cdots,\,j_r}}
\newc{\GDN}{P'_{k,\BoldUpmSub,m_j,\ell}}
\newc{\ckbeta}{c_{k,\beta}}

\newcommand{\lots}{\mbox{\it lots}}        
\newcommand{\ltots}{\mbox{\it ltots}}      

\newc{\fkbeta}{f_{\beta}}
\newc{\hkbeta}{h_{\beta}}
\newc{\qkbetan}{q_{\beta,n}}
\newc{\rkbetan}{r_{\beta,n}}
\newc{\skbetan}{s_{\beta,n}}
\newc{\tkbetan}{t_{\beta,n}}
\newc{\RiemInv}{I}                         
\newc{\URiemInv}{i}                        
\newc{\ad}{\operatorname{ad}}
\newc{\Ad}{\operatorname{Ad}}

\newc{\id}{\operatorname{id}}
\newc{\Ker}{\operatorname{Ker}}
\renewcommand{\Im}{\operatorname{Im}}

\newc{\ddt}{\tfrac{d}{dt}|_{t=0}}
\newc{\into}{\hookrightarrow}
\newc{\ndy}{\aNd^{(Y)}}
\newc{\ndyy}{\aNd_{(Y)}}
\newc{\tD}{\tilde{D}}
\newc{\bD}{\mbox{\boldmath{$D$}}}
\newc{\D}{\mbox{\boldmath{$D$}}}
\newc{\X}{\mbox{\boldmath{$X$}}}
\newc{\sX}{\mbox{\scriptsize\boldmath{$X$}}}        

\newc{\cce}{\tilde{\ce}}
\newc{\tM}{\tilde{M}}
\newc{\tS}{\tilde{S}}
\newc{\tth}{\tilde{h}}
\newc{\tf}{\tilde{f}}
\newc{\tW}{R}
\newc{\tV}{\tilde{V}}
\newc{\tU}{\tilde{U}}

\newc{\JuhlPi}{\mathbf{P}}
\newc{\JuhlPe}{P^e}
\newc{\ilc}{
\nabla
\begin{picture}(0,0)
\put(-9.45,9.5){\line(1,0){8.8}}
\end{picture}}
\newc{\ptc}{\widetilde{\nabla}}
\newc{\ila}{\Delta
\protect\begin{picture}(0,0)
\protect\put(-7.7,9.5){\protect\line(1,0){5.5}}
\end{picture}}
\newc{\StrongGJMS}{P^{\Phi}}                
\newc{\iGJMSden}{P
\protect\begin{picture}(0,0)
\protect\put(-6.0,9.7){\protect\line(1,0){6.0}}
\end{picture}}
\newc{\bop}{\mbox{$\Box$}}                  
\newc{\ibop}{\mbox{$\overline{\Box}$}}      
\newc{\LThree}{\Lambda_{3,w}}               
\newc{\PThree}{\mbox{\sf P}_3}              
\newc{\cRobin}{\delta}                      
\newc{\iD}{D
\protect\begin{picture}(0,0)
\protect\put(-7.0,9.7){\protect\line(1,0){6.5}}
\end{picture}}
\newc{\lowerw}{c}   
\newc{\FieldOfFuns}{F} 
\newc{\RealCoeff}{\mathcal{Q}} \newc{\RationalCoeff}{\mathcal{R}}
\newc{\Classes}{\mathcal{S}}
\newc{\RealBasis}{\{\RiemInv_{\beta,\g}\}_{\beta\in A}}
\newc{\ClassesBasis}{\{[\RiemInv_{\beta,\g}]\}_{\beta\in A}}
\newc{\PreBasis}{\{\RiemInv_{\beta,\g}\}_{\beta\in A}}
\newc{\PreBasisBC}{\{\RiemInv_{\beta,\g}\}_{\beta\in B\cup C}}
\newc{\SJd}{\mathcal{S}_{K,\lowerw,w,n}}
\newc{\VJd}{\mathcal{V}_{K,\lowerw,n}}
\newc{\NJd}{\mathcal{N}_{K,\lowerw}} \newc{\Dimensions}{\mbox{\sf D}}
\renewcommand{\o}{\circ}
\let\ccdot\cdot
\def\cdot{\hbox to 2.5pt{\hss$\ccdot$\hss}}
\newc{\x}{\times}
\newc{\rpl}                         
{\mbox{$
\begin{picture}(12.7,8)(-.5,-1)
\put(0,0.2){$+$}
\put(4.2,2.8){\oval(8,8)[r]}
\end{picture}$}}
\newc{\lpl}                         
{\mbox{$
\begin{picture}(12.7,8)(-.5,-1)
\put(2,0.2){$+$}
\put(6.2,2.8){\oval(8,8)[l]}
\end{picture}$}}
\newc{\IQ}{I}                       
\newc{\g}{{\textsl{g}\hspace{0.12ex}}} 
\newc{\amr}{r}                        
\newc{\gplus}{\g_{+}}                
\newc{\hatg}{\widehat{\g}}          
\newc{\ig}{{\g
\begin{picture}(0,0)
\put(-4.3,6.9){\line(1,0){3.9}}
\end{picture}
\hspace{0.13ex}}
}
\newc{\igsub}{g
\begin{picture}(0,0)
\put(-4.0,4.9){\line(1,0){3.6}}
\end{picture}}

\newc{\hatig}{
\rule{0mm}{1.2ex}\ig
\begin{picture}(0,0)
\put(-4.8,0.9){$\widehat{\ }$}
\end{picture}
}

\newc{\ic}{\mathbf{c}
\begin{picture}(0,0)
\put(-5.1,6.8){\line(1,0){4.3}}
\end{picture}}
\newc{\cc}{\mathbf{c}}                  

\newc{\xigw}[1]{(\xi^{\g})^{#1}}          
\newc{\xig}{\xi^{\g}}
\newc{\xir}{\xi^{\amr}}
\newc{\bg}{\textbf{\textsl{g}}}
\newc{\br}{\textbf{\textsl{r}}}
\newc{\ibg}{{\bg
\begin{picture}(0,0)
\put(-5.9,6.9){\line(1,0){4.6}}
\end{picture}
\hspace{0.0ex}}
}
\newc{\h}{\mbox{\boldmath{$h$}}}
\newc{\Rho}{{\mbox{\sf P}}}         
\newc{\iRh}{\Rho
\begin{picture}(0,0)
\put(-6.1,9.9){\line(1,0){4.9}}
\end{picture}}
\newc{\J}{{\mbox{\sf J}}}           
\newc{\iJ}
{\J\begin{picture}(0,0)\put(-4,9.4){\line(1,0){3}}\end{picture}}
\newc{\DenOne}{V}             
\newc{\namb}{\nabla^{\text{amb}}}
\newc{\Ramb}{R^{\text{amb}}}
\newc{\vol}{\mbox{\boldmath $\epsilon$}}
\newc{\iR}{R
\begin{picture}(0,0)
\put(-6.1,9.6){\line(1,0){6.0}}
\end{picture}}
\newc{\QJkg}{Q_{\g}(\djr)}   
\newc{\TgK}{T^{\g}_K}                  
\newc{\QJkgh}{Q_{\jr,\hat{\g}}}        
\newc{\R}{R}                        
\newc{\C}{C}                        
\newc{\tcurve}{\Omega}              
\newc{\tracW}{W}                    
\newc{\Rc}{\mbox{Rc}}               
\newc{\Sc}{\mbox{Sc}}               
\newc{\tm}{h}                       
\newc{\tmco}{h^{\#}}                
\newc{\tnt}{{\textstyle\frac{n}{2}}}
\newc{\tmi}{h
\begin{picture}(0,0)
\put(-4.7,9.6){\line(1,0){4.0}}
\end{picture}}
\newcommand{\ih}{\tmi}
\newc{\ki}{k
\begin{picture}(0,0)
\put(-4.5,8.8){\line(1,0){4.0}}
\end{picture}}
\newc{\dfn}{t}                      
\newc{\tfrm}{E}                     
\newc{\cfrm}{F}                     
\newc{\cfrmi}{H}                    
\newc{\di}{\Theta}                  
\newc{\da}{\Gamma}                  
\newc{\Nv}{n}                       
\newc{\Nt}{N}                       
\newc{\PP}[2]{\Pi_{#1}{}^{#2}}      
\newc{\sff}{L}                      
\newc{\sffh}{\widehat{L}}           
\newc{\tsff}{L}                     
\newc{\mc}{H}                       
\newc{\mmc}{G}                      
\newc{\mch}{\widehat{H}}            
\newc{\tcon}{C}                     
\newc{\RRR}{\mbox{R}}               

\newc{\iRR}                         
{\mbox{R}
\protect\begin{picture}(0,0)
\protect\put(-7.6,9.7){\protect\line(1,0){5.9}}
\end{picture}}

\newc{\JuhlQi}{\mathbf{Q}}
\newc{\JuhlQe}{Q^e}
\newc{\iBranQ}                         
{B
\protect\begin{picture}(0,0)
\protect\put(-6.6,9.7){\protect\line(1,0){5.8}}
\end{picture}}

\newc{\YS}{
Y
\begin{picture}(0,0)
\put(-8.7,10.1){\line(1,0){8.0}}
\end{picture}}
\newc{\YSO}{
Y
\begin{picture}(0,0)
\put(-6.0,2.1){$\tilde{\ }$}
\end{picture}}
\newc{\ZS}{
Z
\begin{picture}(0,0)
\put(-5.9,10.1){\line(1,0){5.5}}
\end{picture}}
\newc{\ZSO}{
Z
\begin{picture}(0,0)
\put(-4.5,2.1){$\tilde{\ }$}
\end{picture}}
\newc{\XS}{
X
\begin{picture}(0,0)
\put(-7.9,10.1){\line(1,0){7.2}}
\end{picture}}
\newc{\XSO}{
X
\begin{picture}(0,0)
\put(-5.9,2.1){$\tilde{\ }$}
\end{picture}}
\newc{\amM}{\tilde{M}}             
\newc{\act}{\mbox{\boldmath{$\ct$}}}
\newc{\amQ}{\Cal Q}                 
\newc{\Mplus}{M_{+}}               
\newc{\Mg}{(M,\g)}                  
\newc{\Mcc}{(M,\cc)}                
\newc{\ComplexNumbers}{\Bbb C}
\newc{\Reals}{\Bbb R}
\newc{\Integers}{\Bbb Z}
\newc{\QDimensions}{S}             
\newc{\tb}{\ct^A}                  
\newc{\tbn}{\ct}                   
\newc{\tbd}{\ct_A}                 
\newc{\tbPhi}{\ct^{\Phi}}           
\newc{\tbnPhiOne}{\tbn^{\Phi_1}}     
\newc{\tbnPhiTwo}{\tbn^{\Phi_2}}     
\newc{\tbS}{\tbn
\begin{picture}(0,0)
\put(-7.8,9.8){\line(1,0){7.5}}
\end{picture}}
\newc{\tbSsup}{\tbn
\begin{picture}(0,0)
\put(-5.45,7.0){\line(1,0){5.5}}
\end{picture}}
\newc{\tbSA}{\tbS^A}                 
\newc{\tbpA}[1]{\tbn^{#1}_{\hspace*{.3mm}
\rule{0.1mm}{1.5mm}\hspace{.4mm}\rule{0.1mm}{1.5mm}}}
\newc{\tbp}{\tbn_{\hspace*{.3mm}
\rule{0.1mm}{1.5mm}\hspace{.4mm}\rule{0.1mm}{1.5mm}}}

\newc{\hS}{\overline{h}}

\newc{\sm}{\Sigma}                  
\newc{\zerolocus}{\mathcal Z}
\newcommand{\cF}{\mathcal{F}}
\newc{\gequal}{\stackrel{\g}{=}}    
\newc{\ghatequal}{\stackrel{\hatg}{=}}
\newc{\igequal}{\stackrel{\g
\begin{picture}(0,0)
\put(-4.0,4.8){\line(1,0){3.6}}
\end{picture}}{=}}                  
\newc{\trem}{i}                     
\newc{\demb}{J}                     
\newc{\musical}{\Phi}               
\newc{\musicali}{\Phi
\begin{picture}(0,0)
\put(-7.0,9.2){\line(1,0){6.0}}
\end{picture}}
\newc{\ceS}{\ce
\begin{picture}(0,0)
\put(-4.5,9.4){\line(1,0){4.0}}
\end{picture}}
\newc{\sdim}{n-1}                   
\newc{\strutaa}{\rule{0mm}{4mm}}
\newc{\strutdd}{\rule{0mm}{5mm}}
\newc{\pf}{\noindent\textit{Proof:}\ \ }
\newc{\QED}{\quad $\Box$\vspace{2mm}}
\newc{\RHS}{right-hand side}
\newc{\mass}{mass\ }
\newc{\masses}{masses\ }
\newc{\masspoint}{mass}
\newc{\tensor}[1]{#1}
\newc{\Mvariable}[1]{\mbox{#1}}
\newc{\down}[1]{{}_{
\ifthenelse{\equal{#1}{;}}{|}{#1}}}
\newc{\up}[1]{{}^{#1}}
\newc{\midtenPan}{\mbox{\sf S}}
\newc{\midten}{\mbox{\sf T}}
\newc{\midtenEi}{\mbox{\sf U}}
\newc{\ATen}{\mbox{\sf E}}
\newc{\BTen}{\mbox{\sf F}}
\newc{\CTen}{\mbox{\sf G}}

\def\sideremark#1{\ifvmode\leavevmode\fi\vadjust{\vbox to0pt{\vss
 \hbox to 0pt{\hskip\hsize\hskip1em
 \vbox{\hsize2.5cm\tiny\raggedright\pretolerance10000
 \noindent #1\hfill}\hss}\vbox to8pt{\vfil}\vss}}}
\newlength{\pluslength}
\newc{\leadingplus}{\hspace{\pluslength}}
\begin{document}
%
\begin{abstract}
We construct continuously parametrised families of conformally
invariant boundary operators on densities. These may also be viewed as
conformally covariant boundary operators on functions and generalise
to higher orders the first-order conformal Robin operator and an
analogous third-order operator of Chang-Qing. Our families include
operators of critical order on odd-dimensional boundaries. Combined
with the (conformal Laplacian power) GJMS operators, a suitable
selection of the boundary operators yields formally self-adjoint
elliptic conformal boundary problems.  Working on a conformal manifold
with boundary, we show that the operators yield odd-order conformally
invariant fractional Laplacian pseudo-differential operators.  To do
this, we use higher-order conformally invariant Dirichlet-to-Neumann
constructions.  We also find and construct new curvature quantities
associated to our new operator families.
These have links to the Branson $Q$-curvature and include higher-order
generalisations of the mean curvature and the $T$-curvature of
Chang-Qing.  In the case of the standard conformal hemisphere, the
boundary operator construction is particularly simple; the resulting
operators provide an elementary construction of families of symmetry
breaking intertwinors between the spherical principal series
representations of the conformal group of the equator, as studied by
Juhl and others.  We use our constructions to shed
light on some conjectures of Juhl.
\vspace{0mm} \\
%
%
\ifthenelse{\emailadd=1}
{ 
\vspace{3mm}
\\
\newlength{\eaddwid}
\settowidth{\eaddwid}{Authors' e-mail addresses:\ \ }
\parbox[t]{\eaddwid}{Authors' e-mail addresses:}
\parbox[t]{55mm}{
\texttt{gover@math.auckland.ac.nz}\\
\texttt{lawrence.peterson@email.und.edu}\\
}
}  
{} 
%
%
\end{abstract}

\title{Conformal boundary operators, T-curvatures, and conformal
  fractional Laplacians of odd order}
\thanks{ARG gratefully acknowledges
    support from the Royal Society of New Zealand via Marsden Grants
  13-UOA-018 and 16-UOA-051}
\thanks{LJP would like to thank the University of Auckland for its
  hospitality on three separate occasions.}
\author{A. Rod Gover and Lawrence J. Peterson} \date{}
%
%
\address{Department of Mathematics\\
  The University of Auckland\\
  Private Bag 92019\\
  Auckland 1142\\
  New Zealand}\email{r.gover@auckland.ac.nz}
%
%
\address{Department of Mathematics\\
  The University of North Dakota\\
  101 Cornell Street Stop 8376\\
  Grand Forks, ND 58202-8376\\
  USA}\email{lawrence.peterson@email.und.edu}

\maketitle

\pagestyle{myheadings}
\markboth{A.R. GOVER AND L.J. PETERSON}{Conformal boundary Operators}

%
%

\section{Introduction}\label{IntroSect}

Conformally invariant differential operators play a central role in
the global geometric analysis of Riemannian, pseudo-Riemannian, and
conformal structures
\cite{Changbull,MalchiodiQ,Schoen-Y,Witten-Mass}. The nature of
conformal geometry means that local issues can also be subtle and
difficult. For example, given a proposed leading symbol, a
corresponding conformally invariant differential operator may or may
not exist, depending on the conformal class. While important open
problems remain, nevertheless over the past decades there have been
significant advances in our understanding of these issues
\cite{ESlo,GoH,Grnon,GJMS}.

By comparison, relatively little is known about natural conformally
invariant differential operators along submanifolds.  Here we treat an
aspect of this gap. Our work is primarily motivated by potential
applications to boundary value problems, conformally invariant
Dirichlet-to-Neumann constructions (cf.\ \cite{TomSrni,BrGonon,
  CaffSil,Case,CaseCh,chang-Gon,escoA,esco}), and related issues
linked to geometric PDE, scattering, and the AdS/CFT correspondence of
physics \cite{GrZ,Mald,Nd,Nd2}.  Although we do not treat continuous
families of non-local operators here, we believe that our results are
complementary to e.g.\ \cite{CaffSil,CaseCh,chang-Gon,GrZ} and that
the tools developed here should shed light on such families.  The
results also give new constructions of families of representation
intertwining operators in the sense of Juhl's families \cite{JB} and
the differential symmetry breaking operators of Kobayashi et
al.\ \cite{KobayashiSpeh}.  See also
\cite{Clerc,FJuhlSom,KobayashiPevzner}.

In this paper, we construct continuously parametrised families of
conformally invariant differential boundary operators on densities.
We also construct curvatures associated to these operator families.
The new operators may be viewed as conformally covariant boundary
operators on functions and for most parameter values generalise to
higher orders the well-known first-order conformal Robin operator.  We
show that some of the operators from the new families combine with the
GJMS operators of Graham et al.\ \cite{GJMS} to yield formally
self-adjoint elliptic conformally invariant higher-order Laplacian
boundary problems.  In conjunction with these problems, we construct
higher-order conformally invariant Dirichlet-to-Neumann operators that
provide realisations of odd-order conformally invariant fractional
Laplacian pseudo-differential operators. In the later sections, we
work in the setting of general conformally curved manifolds with
boundary, but by virtue of the conformal invariance, the results may
be applied to conformally compact manifolds and in particular to
asymptotically hyperbolic and Poincar\'e-Einstein manifolds.  Earlier
work of Branson and the first author in \cite{BrGonon} used tractors
to construct new families of boundary operators.  (We will review
tractors in Section~\ref{confsec}, below.)  In the present article,
the approach is again via tractors, but it is simpler, quite
different, and most importantly, more effective; we close gaps in
\cite{BrGonon}.

We now explain in more detail the problems treated in this paper, the
motivations for these problems, and the results obtained.  Throughout
our work, $M$ is a manifold of dimension $n$, $\cc$ is a conformal
equivalence class of metrics on $M$, and $\g$ is a metric in $\cc$.
The term {\em hypersurface} will always refer to a smooth conformally
embedded $(n-1)$-dimensional conformal manifold $(\sm,\ic)$ in $\Mcc$.
In this context, $\ic$ denotes the conformal equivalence class of the
metric on $\sm$ induced by any $\g\in\cc$, and $\ig$ will denote the
metric on $\sm$ induced by such a $\g$.  We will often assume that $M$
is initially given as a manifold with boundary, and $\sm=\partial M$
is the boundary of $M$.  But in this case, for the convenience of
treating the local issues, we will assume that we have extended $M$ to
a collar neighbourhood of $\sm$ (so that $\Sigma$ consists of interior
points of $M$).  Unless we indicate otherwise, we will always assume
that $n\geq 3$.  For simplicity of discussion, all structures we
consider below will be taken to be smooth to infinite order, and we
shall restrict to the case of Riemannian signature. These restrictions
can be relaxed to a large extent with little change.

The operators in our new families are natural.  In
Section~\ref{basics}, below, we will discuss the idea of natural
differential operators along a hypersurface, as well as the idea of
scalar Riemannian hypersurface invariants.  We note here, however,
that a natural differential operator along a hypersurface is a rule
that defines a differential operator on a neighbourhood of any
hypersurface $(\sm,\ic)$ in any ambient conformal $n$-manifold $\Mcc$.
In a similar way, a hypersurface invariant is actually a rule which
defines an invariant in a neigbourhood of any hypersurface $(\sm,\ic)$
of any $\Mcc$.  (In both cases, however, we may require $(\sm,\ic)$
and $\Mcc$ to meet certain conditions.)  The operators in our new
families are parametrised by $n$ (in each case in some infinite set
$\Dimensions\subseteq\Integers_{\geq 3}$) and a real number $w$.  We
say that $w$ is a \textit{weight}.  The operators in a given family
will always be determined by a single universal symbolic formula, so
we will
sometimes use the word ``operator'' to refer to any of our new
families of operators or to any other family of operators determined
by a single universal symbolic formula.  Similar remarks apply to
families of hypersurface invariants.  Although our main results give
operators that, along the hypersurface, are determined by the
conformal embedding, the development of these results involves
operators that can depend on a choice of metric $\g$ within $\cc$.

Some of the operators and curvature quantities we construct in this
paper are motivated by certain well-known natural differential
operators and their associated curvature quantities.  One such
operator and curvature pair is the \textit{conformal Robin operator}
(of Cherrier \cite{cherrier}) and the mean curvature.  To define the
conformal Robin operator, we begin by identifying $\sm$ with its image
submanifold under an embedding $\iota: \Sigma \to M$.  Fixing
$\g\in\cc$ determines a metric $\ig\in\ic$, and we write $n_a$ for the
associated unit conormal field to $\Sigma$.  This conormal field
together with the metric determines a unit normal vector field
$n^a=\g^{ab}n_b$.  We let $H$ (or sometimes $H^{\g}$) denote the mean
curvature of $\Sigma$.  Let $\nd$ denote the Levi-Civita connection
determined by $\g$, and let $C^{\infty}(M)$ and $C^{\infty}(\sm)$
denote the sets of all smooth real-valued functions on $M$ and $\sm$,
respectively.  The conformal Robin operator is the operator $\delta :
C^\infty(M) \to C^\infty(\Sigma)$ defined by the composition of
$f\mapsto n^a\nabla_a f - w H f$ with restriction to
$C^\infty(\Sigma)$.  Here $w$ is a real parameter.  As before, we say
that $w$ is a \textit{weight}.  We define conformal covariance and
bidegree in Definition~\ref{Bid} below.  For all $w\in\Reals$,
$\delta$ is conformally covariant of bidegree $(-w,1-w)$.

The operator $\delta$ is a Robin operator because it mixes Dirichlet
and Neumann data.  Its conformal covariance is important for forming
well-posed conformal boundary problems involving the conformal
Laplacian (or ``Yamabe operator'') \cite{BrGonon,cherrier,escoA,GG} on
the $n$-dimensional interior.  Let $w=1-n/2$, which is the weight
selected by the covariance properties of the Yamabe operator, and
consider the limiting case in which $n=2$.  Then $w=0$, and $\delta$
is just the Neumann operator $n^a\nabla_a$.  The mean curvature $H$
drops out of the formula for $\delta$, but it retains an interesting
limiting link with $\delta$: In this specific case, $H$ is the
geodesic curvature and transforms conformally by the rule
$H^{\widehat{\g}}=e^{-\Upsilon}(H^{\g}+\delta^{\g} \Upsilon)$.  Here
and throughout this paper, $\Upsilon\in C^{\infty}(M)$, and
$\widehat{\g}=e^{2\Upsilon}\g$.

In connection with the problem of understanding Polyakov-type formulae
for conformally covariant elliptic differential operators on compact
4-manifolds with boundary, Chang and Qing discovered a third-order
analogue, $\CQOp$, of the conformal Robin operator.  (They denoted
this operator by $P_3$.)  See \cite{CQ}.  This operator acts along the
boundary of the compact 4-manifold.  Chang and Qing showed that
$\CQOp$ is naturally associated to a scalar curvature quantity $T=T^{CQ}$
along the boundary and that the transformation rule for $T$ under
conformal change of metric is $T^{\widehat{\g}}=e^{-3\Upsilon}(T^{\g}
+ \CQOpg\Upsilon)$. In this and other ways (as explained in
Section~\ref{Qsec}, below) the relation between $\CQOp$ and $T^{CQ}$
is analogous to the relation between the Neumann operator and the geodesic
curvature on a surface with boundary which we described above.  By
\cite{CQ}, $\CQOp$ is conformally covariant of bidegree $(0,3)$.

In \cite{GJMS}, Graham et al.\ defined a conformally invariant $k$th
power of the Laplacian on a conformal $n$-manifold $\Mcc$.  This
operator is commonly known as the \textit{GJMS operator} of order
$2k$, and we will denote it by $P_{2k}$.  It acts on conformal
densities, and it is well-defined for all positive integers $k$ within
an appropriate range.  (We will discuss conformal densities and their
weights in Section~\ref{confsec}, below.)  In \cite{Brsharp}, Branson
used $P_n$ to define a scalar curvature quantity $Q_n$ on a compact
Riemannian manifold $\Mg$ without boundary of even dimension $n\geq
4$.  This curvature is called the $Q$-curvature of $\Mg$.  The
operator $P_n$ may be viewed as a conformally covariant operator
$P_n^{\g}:C^{\infty}(M)\rightarrow C^{\infty}(M)$ of bidegree $(0,n)$.
(See Proposition~\ref{ReplaceCCov}.)  Under conformal change of
metric, $Q_n$ transforms according to the rule
$Q_n^{\hatg}=e^{-n\Upsilon}(Q_n^{\g}+P_n^{\g}\Upsilon)$.
The pair $(\delta,H)$ in dimension $n=2$ and the Chang-Qing pair
$(\CQOp,T^{CQ})$ in dimension $n=4$ give, in a certain sense, odd-order
boundary analogues of the pair $(P_n,Q_n)$.  This has generated
considerable interest (\cite{Case,CatinoNd,JB,Nd,Nd2,SunXiong})
and motivates Problem~\ref{GenProblem}, below.

The statement of Problem~\ref{GenProblem} involves the notion of the
\textit{transverse order} of a boundary operator along $\sm$, as
defined in Definition~\ref{TOrderDefinition}, below.
Definition~\ref{TOrderDefinition} is adapted from the definition of
the \textit{normal order} of a boundary operator given in
\cite{BrGonon}.  Our definition uses a defining function for $\sm$ on
a neighbourhood $U$ of a point $p\in\sm$.  Here and below, a such a
function is a function $t\in\C^{\infty}(U)$ such that $\sm\cap U$ is
the zero locus of $t$ and $dt$ is nonvanishing on $U$.
\begin{definition}\label{TOrderDefinition}
Let a Riemannian manifold $(M,\g)$, a
hypersurface $(\sm,\ig)$ in $(M,\g)$, a vector bundle $\cf$ over $M$,
and a differential operator $B:\cf\mapsto \cf|_{\sm}$ be given.  Also
let $p\in\sm$ be given, and suppose that $\dfn$ is any defining function
for $\sm$ on some neighbourhood $U$ of $p$.  For any
$m\in\Integers_{>0}$, we say that $B$ has \textit{transverse order}
$m$ at $p$ if there is a smooth section $V$ of $\cf$ such that
$B(\dfn^mV)|_{p}\neq 0$ but $B(\dfn^{m+1}V')|_{p}=0$ for all smooth
sections $V'$ of $\cf$.  If $B(\dfn V)|_{p}=0$ for all smooth sections
$V$ of $\cf$, then we say that $B$ has transverse order $0$ at $p$.
Now let $m\in\Integers_{\geq 0}$ be given.  If $B$ has transverse
order $m$ at every $p\in\sm$, then we say that $B$ has transverse
order $m$.  Finally, suppose that $m>0$.  If $B$ has transverse order
less than $m$ at every $p\in\sm$, then we say that $B$ has transverse
order \textit{less than} $m$ or transverse order \textit{at most}
$m-1$.
\end{definition}
In the definition here, and also throughout the article, we interpret
notation to mean vector bundles or their smooth section spaces
according to context.  The properties described in this definition are
independent of the choice of the defining function $\dfn$.  The
transverse order at a point $p\in\sm$ measures the number of
derivatives in directions transverse to $\sm$ at $p$.  If we say that
a natural hypersurface differential operator $B$ has order $m$, order
less than $m$, transverse order $m$, or transverse order less than
$m$, we mean that the property holds for all hypersurfaces in all
possible Riemannian manifolds; for the operator families that we deal
with in this paper, the presence of these properties may depend on the
values of the parameters $w$ and $n$ mentioned above.  The conformal
Robin operator has transverse order 1 in all dimensions, and $\CQOp$
has transverse order 3.

In the statement of Problem~\ref{GenProblem} that follows, and throughout this
paper, differential operators will act on everything to their right,
unless parentheses indicate otherwise.  In addition, we will always
indicate multiplication operators by juxtaposition.  For example, let
differential operators $Op_1,\,Op_2:C^{\infty}(M)\to C^{\infty}(M)$
and functions $f_1,\,f_2\in C^{\infty}(M)$ be given.  Then
$Op_1f_1Op_2f_2$ denotes $Op_1$ acting on the product of $f_1$ and
$Op_2f_2$.  All differential operators in this paper will be linear.

We may now state one of the main problems of this paper.
\begin{problem}\label{GenProblem}
  {\it For some $K\in\Integers_{>0}$, construct a family of natural
    hypersurface operators $P^{\g}_{w,K}:C^{\infty}(M)\rightarrow
    C^{\infty}(\sm)$ parametrised by the dimension $n$ and a real
    parameter $w$ such that the following properties hold: (1)~For all
    $w\in\Reals$, $P^{\g}_{w,K}$ has order at most $K$.  There is a
    small finite set $E$ such that for all $w\in\Reals\backslash E$,
    $P^{\g}_{w,K}$ has order and transverse order $K$.  The set $E$
    may depend on $n$.
(2)~For all $f\in C^{\infty}(M)$ and all
$w\in\Reals$,
$P^{\hatg}_{w,K}f=e^{-(K-w)\Upsilon}P^{\g}_{w,K}e^{-w\Upsilon}f$.
(3)~The family $P^{\g}_{w,K}$ determines a scalar curvature
quantity $Q$ along $\sm$ with the property that}
$
Q^{\hatg}=e^{-K\Upsilon}(Q^{\g}+P^{\g}_{0,K}\Upsilon)
$.  This scalar curvature quantity is a local scalar Riemannian
hypersurface invariant.
\end{problem}
One of our main objectives will be to find solutions to
Problem~\ref{GenProblem} in which the set $E$ of exceptional weights
is as small as possible or solutions in which some specific
undesirable weight (such as $w=0$) is absent from $E$.

We will be especially interested in higher-order generalisations of
the conformal (Cherrier-)Robin operator and the Chang-Qing operator.
The Chang-Qing operator has order and transverse order $n-1=3$, the
dimension of $\sm$.  Similarly, in the limiting case in which $n=2$,
the conformal Robin operator likewise has order and transverse order
$n-1$.  In general even dimensions $n$, establishing the existence of
conformal boundary operators of order and transverse order $n-1$ is
rather delicate.  Such operators sit in a special position that is in
part analogous to the place of the dimension-order GJMS operators of
\cite{GJMS}.  Thus, in the setting of Problem~\ref{GenProblem}, if $n$
is even, we will say that a boundary operator of order and transverse
order $n-1$ is a \textit{critical operator} and that $n-1$ is the
\textit{critical order}. Paramount in our considerations here is
finding a solution to Problem~\ref{GenProblem} that includes such
operators as part of a continuous family. For emphasis we state the
following special case of Problem~\ref{GenProblem}:
%
%
\begin{problem}\label{CQprob}
{\it In Problem~\ref{GenProblem}, suppose that $n$ is even and
  $K=n-1$.  Find a solution such that $P^{\g}_{0,K}$ is a critical
  operator.}
\end{problem}
In this paper, we construct families of operators which solve
Problem~\ref{CQprob} in general even dimensions $n\geq 4$.  We do this
by first constructing three families of operators which solve
Problem~\ref{GenProblem}.  Let $K$ be as in Problem~\ref{GenProblem}.
Then for all $K\in\Integers_{>0}$,
the operator family $\ConFlatOp{K}$ of Theorem~\ref{cflatkey} solves
Problem~\ref{GenProblem} in the simplest setting, namely the case of
conformally flat metrics.  (Throughout this paper, conformally flat
means what is sometimes referred to as ``locally conformally flat'',
that is, the Weyl and
Cotton tensors both vanish.  Also, note that we suppress the $\g$ and
the $w$ from the operator notation.)  The $\ConFlatOp{K}$ family is
defined in all dimensions $n\geq 3$.  Next, for all
$K\in\Integers_{>0}$, the operators $\deltaJk$ of Theorem~\ref{key},
with $J+k=K$, are defined for general Riemannian conformal structures
in suitable dimensions and for conformally flat structures in all
dimensions $n\geq 3$.  Similarly, for all $K\in\Integers_{>0}$, the
operators $\delta_K$ of Lemma~\ref{basicversion} are defined on
general Riemannian conformal structures in all dimensions $n\geq 3$.
The families $\deltaJk$ and $\delta_K$ solve Problem~\ref{GenProblem}.

The $\ConFlatOp{K}$ operator family solves Problem~\ref{GenProblem} on
the standard conformal hemisphere, and in this case, the operators
$\ConFlatOp{K}$ provide families of symmetry breaking intertwinors
between the spherical principal series representations of the
conformal group of the equator, as studied by Juhl and others
\cite{Clerc,JB,KobayashiSpeh}. There is some appeal in this picture,
as here the symmetry breaking (the reduction of the conformal group of
the sphere to a subgroup preserving the closed hemisphere) is
manifestly governed
by the normal tractor.  By using the tools in, for example,
\cite{GLW-forms}, one may develop generalisations of the
$\ConFlatOp{K}$ and $\deltaJk$ families which treat differential forms
and other tensors.  Such generalisations should be of some interest to
the more general intertwinor programme as in
\cite{FJuhlSom,Kobayashi-book}, but we do not investigate this idea in
this paper.

As part of a construction of a family of conformally invariant
higher-order Dirichlet-to-Neumann operators, Branson and the first
author discover and construct a family $\delta^{BG}_K$ of conformally
invariant hypersurface operators in Theorem~5.1 of
\cite{BrGonon}. Although used at a discrete set of weights in
\cite{BrGonon}, by construction the operators are available on any
density bundle, and in the interesting work \cite{JB}, Juhl studies
the resulting continuously parametrised families of boundary operators
(as well as so-called ``residue families'') and considers some
problems linked to those studied here; see for example \cite[Section
  1.10 and Section 6.21]{JB}.  Unfortunately, as evident in
\cite[Theorem 5.1]{BrGonon} and its proof, the $\delta^{BG}_{K}$
family does not provide an answer to Problem~\ref{CQprob}.  In
\cite{Grant}, Grant constructs a modification $\delta^G_3$ of
$\deltaBG_3$ which solves Problem~\ref{CQprob} in the dimension $n=4$
case. (See also Stafford \cite{Stafford} and Case \cite{Case}.)
However, Grant's work was specific to third-order operators, and the
problem of finding higher-dimensional analogues of the Chang-Qing pair
$(\CQOp,T^{CQ})$ has, to our knowledge, remained open up to now.

 For every even integer $n\geq 4$, the $\ConFlatOp{K}$ family includes
 an operator which solves Problem~\ref{CQprob} for conformally flat
 Riemannian conformal manifolds of dimension $n$, and the $\deltaJk$
 family contains a solution to this same problem for general Riemannian
 conformal manifolds of dimension $n$.  See Theorem~\ref{SolnTheorem},
 below.

In Section~\ref{FSAsec}, the operators $\deltaJk$ and $\delta_K$ are
used to set up conformally invariant elliptic boundary problems
for the GJMS operators and then to construct non-local operators. The
first main result is Theorem~\ref{FSAthm}, which, for any GJMS
operator $P_{2k}$, describes a corresponding conformally invariant
boundary system $B$ such that the boundary value problem $(P_{2k}, B)$
is formally self-adjoint.  The boundary system $B$ is given by
\nn{MultiBOp}.  Lemma~\ref{LSLemma2} then shows that the boundary
problem $(P_{2k}, B)$ satisfies the Lopatinski-Shapiro
condition.  The pair $(P_{2k}, B)$ is properly elliptic, and in
Theorem~\ref{nonloc}, we use this fact, together with the other
properties of the pair, to construct Dirichet-to-Neumann operators.
The statement of Theorem~\ref{nonloc} is too technical to be given in
this introduction, but we may easily state two corollaries of the
theorem here.  In these corollaries, $\ice[w]$ denotes the bundle of
conformal densities of weight $w$ associated to $(\sm,\ic)$; here $w$
is any given real number.  The first corollary follows from the case
$k=n/2$ ($n$ even)
of Theorem~\ref{nonloc} and Lemma~\ref{deltaTK}.
\begin{corollary}\label{oddbound}
Let a compact Riemannian conformal manifold with boundary
$(M,\Sigma,\cc)$ of even dimension $n\geq 4$ be given.  Let $B$ be as
in \nn{MultiBOp}, and suppose that $(P_n,B)$ has trivial kernel.  Then
for $2m=1$, $3$, $5$, \ldots, $\sdim$, Theorem \ref{nonloc} yields
conformally invariant non-local operators
$$P^{T,\,n/2}_{2m}:
\ice\left[m-\frac{\sdim}{2}\right]
\to
\ice\left[-m-\frac{\sdim}{2}\right]
$$
which have leading term $(-\ila)^m$.
\end{corollary}
Corollary~\ref{oddbound}
produces
a critical-order (i.e.\ order $\sdim$) fraction Laplacian as the
case $2m=\sdim$.
Such an operator was missing from the results of \cite{BrGonon}.

The next corollary also follows from Theorem~\ref{nonloc} and
Lemma~\ref{deltaTK}.
\begin{corollary}\label{evenbound}

Let a compact Riemannian conformal manifold with boundary
$(M,\Sigma,\cc)$ of dimension $n\geq 3$ be given.  Also let
$k\in\Integers_{>0}$ be given, and suppose that (1)~$n$ is odd, or
(2)~$n$ is even and $k\leq n/2$, or (3)~$\Mcc$ is conformally flat.
Finally, let $B$ be as in \nn{MultiBOp}, and suppose that $(P_{2k},B)$
has trivial kernel.  Then for $2m=1$, $3$, $5$, \ldots, $2k-1$,
Theorem \ref{nonloc} yields conformally invariant non-local operators
$$
 P^{T,k}_{2m}:
\ice\left[m-\frac{\sdim}{2}\right]\to
\ice\left[-m-\frac{\sdim}{2}\right]
$$
which have leading term $(-\ila)^m$.
\end{corollary}
For each solution $P^{\g}_{w,K}$ to Problem~\ref{GenProblem}, we will
say that the associated scalar curvature quantity $Q$ is a {\em
  $Q$-type curvature}.  This $Q$-type curvature has certain properties
analogous to those of Branson's $Q$-curvature.
However, for any dimension $n$ and any $K\in\Integers_{>0}$, if a
solution $P^{\g}_{0,K}$ has transverse order $K$, then we will say
that the $Q$-type curvature associated to $P^{\g}_{w,K}$ is a
\textit{$T$-curvature} of order $K$ and that $(P^{\g}_{0,K},Q^{\g})$
is a $T$-curvature \textit{pair} of this order.  We often let $T^{\g}$
denote $Q^{\g}$ in this case.  In Section~\ref{Qsec}, we will see that
for any given even dimension $n\geq 4$, there are $T$-curvature pairs
of all orders.  Specifically, we will establish the following theorem:
\begin{theorem}\label{T-thm-1}
Let $n_0\in\Integers_{\geq 4}$ be given, and suppose that $n_0$ is
even.  Then in dimension $n=n_0$, there are canonical $T$-curvature
pairs
\[
\begin{array}{ll}
\lefteqn{
(\delta_1,T_{1}^{\g}),\,(\delta_2,T_2^{\g}),\,\ldots\,,
  \,(\delta_{n_0/2},\,T_{n_0/2}^{\g}),\,
  (\delta_{1,\,n_0/2},\,T_{1+n_0/2}^{\g}),\,\ldots\,,}
\\
&
(\delta_{(n_0-2)/2,\,n_0/2},\,T_{n_0-1}^{\g}),
\,(\delta_{n_0},T_{n_0}^{\g}),\,\,(\delta_{n_0+1},T_{n_0+1}^{\g}),
\,\,\ldots
\hspace{9ex}
\end{array}
\]
of orders $1,\,2,\,3,\,\,\ldots\,,$ respectively.
\end{theorem}
Let $\g\in\cc$ be given.  Proposition~\ref{H0}, below, gives
conditions which ensure that for all $m\in\Integers_{\geq 2}$, there
is a metric $\hatg\in\cc$ which induces $\ig$ and which satisfies
$T_1^{\hatg}=T_2^{\hatg}=\cdots=T_m^{\hatg}=0$ along $\sm$.  Under
these same conditions, $T_1^{\hatg}=H^{\hatg}$; this will become
clear later.  So under these conditions, the hypersurface $\sm$ is
minimal for the metric $\hatg$ in the sense that $H^{\hatg}=0$ along
$\sm$, and $\sm$ also satisfies the related higher-order condition
$T^{\hatg}_2=\cdots=T^{\hatg}_{m}=0$ along $\sm$.

We have computed explicit symbolic symbolic formulae for our
new operator families and their $Q$-type curvatures in certain cases.
One such formula is the following formula for the $Q$-type curvature
of the $\delta_{1,2}$ family:
\begin{equation}\label{OneTwoQ}
3\Nv^a\nd_a\J-(n-2)\Nv^a\Nv^b\Nv^c\nd_a\Rho_{bc}+6\mc\J
-6(n-2)\mc\Nv^a\Nv^b\Rho_{ab}
+2(n-2)\mc^3~.
\end{equation}
Here $\Rho_{ab}$ is the Schouten tensor, and $\J=\Rho_{a}{}^a$.  (We
define the Schouten tensor in Section~\ref{back}, below.)  We give
explicit formulae for $\delta_{1,2}$ and a few of our other operator
families and their curvatures in Section~\ref{exs}, below.  Some of
these formulae are valid only in certain dimensions, as explained in
that section.

Our operator and curvature constructions use the Fefferman-Graham
ambient metric of \cite{FGambnew,FGast}, its link to tractors\cite{CapGoamb}, and the tractor construction of the GJMS operators
developed in \cite{GP-CMP}.  We will often work with symbolic formulae
which are polynomial in the parameter $w$ of Problem~\ref{GenProblem}
and rational in the dimension $n$.  As a consequence, our proofs will
often use polynomial continuation in $w$ and rational continuation in
$n$.
To treat these notions, we develop some tools and results in
Section~\ref{OpBases}, below.

%
%
%
%
%

\section{Conformally covariant operators along a
  hypersurface} \label{basics}

In this section, we work with arbitrary
Riemannian metrics.  We will generally employ Penrose's abstract index
notation.  (If no ambiguity will occur, however, we will sometimes
omit indices from tensors.)  We shall write $\ce^a$ to denote the
space of sections of the tangent bundle $TM$ over $M$, and $\ce_a$ for
the space of sections of the cotangent bundle $T^*M$.  (In fact, we
will often use the same symbols for the bundles themselves.)  We write
$\ce$ for the space of real-valued functions on $M$ (or for the
trivial bundle $M\times\Reals$).
All functions, vector bundles, and sections of vector bundles will be
assumed to be smooth, meaning $C^\infty$.  An index which appears
twice, once raised and once lowered, indicates a contraction.  The
metric $\g_{ab}$ and its inverse $\g^{ab}$ enable the identification
of $\ce^a$ and $\ce_a$, and we indicate this by raising and lowering
indices in the usual way.  In Section~\ref{back}, below, we will
discuss the Riemannian and Weyl curvature tensors and the Ricci and
Schouten tensors; throughout much of
this paper, the term ``Riemannian curvature tensor'' will include
these tensors and the traces of the Ricci and Schouten tensors.

Let an embedded hypersurface $\sm$ of $M$ be given, and suppose that
$M$ and $\sm$ are both orientable.  We will usually work locally and
assume that $\sm$ is the zero locus $\zerolocus(\dfn)$ of a defining
function $\dfn$, so the orientability assumptions will usually not
impose any restriction.  We will need to consider geometric quantities
determined on $\sm$. Rather than deal with the awkwardness of fields
and quantities which are defined only along $\Sigma$, we will define
extensions of these into a neighbourhood of $\sm$; we emphasise that
our final results will not depend on the choice of these extensions.
We will calculate in a neighbourhood on which $d\dfn$ is nowhere zero,
and we define
\begin{equation}\label{pndef}
\Nv_a:=\frac{d\dfn}{|d\dfn|_{\g}}
\end{equation} 
in this neighbourhood.  The level sets of $\dfn$ define a foliation
$\sm_{\dfn}$, with $\sm=\sm_0$, such that $n_a$ gives the unit
conormal field along each leaf $\dfn=\mbox{constant}$.

For operators on functions, the notion of naturality along a
hypersurface is an obvious adaptation of the usual notion from
Riemannian geometry (see, e.g., \cite{ABP,Stredder}).  A
\textit{natural differential operator along a hypersurface} is a
differential operator which, in a neighbourhood of the hypersurface,
may be expressed by a universal symbolic formula which is polynomial
in the Levi-Civita connection $\nd$ of $\Mg$ and has tensor-valued
\textit{pre-invariants} as coefficients.  We refer the reader to
Section~2.4 of \cite{GW-willmore} for the definitions of scalar- and
tensor-valued Riemannian hypersurface pre-invariants and
\textit{invariants}.  In this paper, any family of natural operators
will always be given by the same universal symbolic formula for all
possible conformal manifolds $\Mcc$; a similar remark applies to
families of invariants.  The symbolic formula for a family of such
operators or invariants will be a polynomial in the conormal field
$\Nv_a$, the Riemannian metric $\g_{ab}$, its inverse $\g^{ab}$, the
Riemannian curvature tensor $R_{ab}{}^{c}{}_{d}$, the mean curvature
$\mc$, and the Levi-Civita connection $\nd_a$.  (We give examples of
symbolic formulae for operators and invariants at various points in
this paper.  Tractors, the trace-free part of the second fundamental
form, and various operators will appear in some of these formulae, but
it will always be possible to expand these formulae and write them as
polynomials of the above type.)  The coefficients in the polynomial
formula for a family of operators or invariants will be real functions
of the dimension $n$ and the weight parameter $w$ that we discussed in
Section~\ref{IntroSect}, above.  This real function will be polynomial
in $w$ and rational in $n$.  In the case of an invariant, $w$ will be
absent from the formula.  For most of the universal formulae that we
work with in this paper, we may assume that $\nd_a$ never explicitly
hits $\Nv_b$ or $\mc$.

In Section~\ref{confsec}, below, we will incorporate conformal
densities into our operators and invariants.  The above definitions of
natural operators and hypersurface invariants will extend to this
situation in the obvious way.  Our work will also require the notion
of natural differential operators between sections of tractor bundles.
As we will see in Section~\ref{confsec}, below, a tractor bundle is a
finite-dimensional vector bundle over $\Mcc$.  We
will see that for each $g\in \cc$, a tractor bundle decomposes into a
direct sum of tensor bundles.  Naturality of a differential operator
between sections of tractor bundles will mean naturality with respect
to any such decomposition.  Our earlier discussion concerning
\textit{families} of operators and invariants extends to our work with
densities and tractors in an obvious way.

Now suppose that we choose an orientation for $\sm$.  Then on $\sm$,
hypersurface invariants and natural differential operators along $\sm$
are independent of the choice of the defining function $\dfn$, because
we insist that $d\dfn/|d\dfn|_{\g}$ be consistent with the
orientation.  Such invariants and operators need not be uniquely
determined \textit{off} of $\sm$.  Let $\cf$ denote $\ce$ or the set
of all smooth sections of any power of the tractor bundle over $M$,
and let $V\in\cf$ be given.  Also let a natural hypersurface
differential operator $Op:\cf\rightarrow\cf|_{\sm}$ be given.  Then
$Op\,V$ will usually denote the restriction of $Op\,V$ to $\sm$, but
we will sometimes indicate the restriction explicitly by writing
$Op\,V|_{\sm}$.  We will sometimes write $Op:\cf\rightarrow\cf$
instead of $Op:\cf\rightarrow\cf|_{\sm}$.

One example of a scalar hypersurface invariant is the mean curvature
$H =(1/(n-1))\ig^{ab}L_{ab}$.  Here $L_{ab}$ is the second fundamental
form, and $\ig^{ab}$ is the hypersurface metric determined by $\g$ as
in Section~\ref{RH}, below.  Similarly, $L_{ab}$ is a tensor-valued
hypersurface invariant, and $|L|^2= L_{ab}L^{ab}:=\ig^{ac}\ig^{bd}
L_{ab}L_{cd} $ is another scalar invariant.%

For any $w\in\Reals$, a scalar Riemannian hypersurface invariant $K$
determines a \textit{conformal invariant} of weight $w$ if it
satisfies the conformal covariance relation
$K(e^{2\Upsilon}\g,\dfn)= e^{w\Upsilon_\Sigma} K(\g,\dfn)$
for all $\Upsilon\in\ce$.  Here $\Upsilon_\Sigma$ is the pullback of
$\Upsilon$ to the hypersurface $\sm$.  Thus $K$ determines a
homogeneous function on $\amQ$, the bundle of conformal metrics.  This
function represents an invariant conformal density of weight $w$.
\begin{definition}\label{Bid}
We say that a natural differential operator $P^{\g}:\ce\rightarrow\ce$
is \textit{conformally covariant} of
\textit{bidegree} $(w_1,w_2)$ in $\Reals^2$ if
$
P^{\hatg}V= e^{-w_2\Upsilon}
P^{\g}e^{w_1\Upsilon}V
$
for all $\g$ and for all $\Upsilon$ and $V$ in $\ce$.
This definition extends to hypersurface operators and operators on
tractors in the obvious way.
\end{definition}
In Section~\ref{confhy}, we will replace conformal covariance with the
equivalent notion of conformal invariance.  This will simplify the
discussion and calculations.  Our aim will be to construct special
natural conformally invariant operators; there will be no attempt to
classify operators. The strategy is to build these in such a way that,
by construction, they satisfy the naturality conditions.  Achieving
conformal invariance, although more subtle, will eventually be seen to
yield to the same approach.

Before continuing, we consider two examples of conformally covariant
natural operators.  For the conformal Robin operator
$\delta:\ce\rightarrow\ce|_{\sm}$, naturality is evident from the
formula
\begin{equation}\label{cr}
\delta f:=n^a\nabla_a f -w H f~.
\end{equation}
We will also let $\delta_1$ or $\delta_{1,\g,w}$ denote the conformal
Robin operator.
In any dimension $n\geq 2$, the mean
curvature $H$ satisfies the conformal transformation rule
$H_{\widehat{\g}}=e^{-\Upsilon}(H_{\g}+\delta_{1,\g,0}\Upsilon) $.
From this it follows that for any $w\in\Reals$, the conformal Robin
operator is conformally covariant of bidegree $(-w,1-w)$.

A second-order analogue of $\delta_1$ is the operator $\delta_2$ (also
denoted by $\delta_{2,\g,w}$) given by the formula
\begin{equation}\label{d2}
\begin{array}{ll}
\lefteqn{\delta_2 f=}\\
& 
-(\Delta+w\J) f+(n+2w-2)\Nv^a\Nv^b\nd_a\nd_b f
-2(w-1)(n+2w-2)\mc\Nv^a\nd_af
\\
&
+(w-1)w(n+2w-2)\mc^2 f
+w(n+2w-2)\Nv^a\Nv^b\Rho_{ab}f~.
\end{array}
\end{equation}
Here $\Delta=\nd_i\nd^i$ and $f\in\ce$.  In Section~\ref{back}, below,
we will see that $\Rho_{ab}$ and $\J$ are Riemannian invariants.  From
this it will follow that $\delta_2$ is manifestly natural, since each
object in the formula for $\delta_2$ is determined by the data of the
ambient Riemannian structure on $M$ and the embedding. No other
information is involved. The normal vector field $\Nv^a$ appears in
the formula, but along $\sm$, the value of $\delta_2f$ is independent
of the extension of $\Nv^a$ off of $\sm$.  Consideration of the
tractor formula for $\delta_2$ in Section~\ref{main} will show that if
$w$ is any real number, then $\delta_{2,\g,w}$ is conformally
covariant of bidegree $(-w,-w+2)$.  If $w=1-n/2$, then
$\delta_2=-\Box$, where $\Box$ is as defined in \nn{YamOp}, below.  In
this case, $\Box$ is the Yamabe operator for $(M,\cc)$.  In
Section~\ref{ExamplesOrder2}, we will relate $\delta_2$ to the
intrinsic Yamabe operator on $\sm$ and to a second-order hypersurface
operator of \cite{Grant}.  Finally, note that $\delta_2$ is
considerably more complex than $\delta_1$. It is easily seen that
there is exponential growth in complexity as order increases.  A
na\"{\i}ve approach to conformal submanifolds will therefore not
suffice.

%
%

%
%
%

\section{Conformal geometry and hypersurfaces} \label{back}

Let $\nabla_a$ denote the Levi-Civita connection on a Riemannian
manifold $\Mg$ of dimension $n\geq 2$. 
The Riemannian curvature
tensor $R$ on $\Mg$ is defined by
\[
\RRR(X,Y)Z=\nd_{X}\nd_{Y}Z-\nd_{Y}\nd_{X}Z-\nd_{[X,Y]}Z~.
\]
Here $X$, $Y$, and $Z$ are arbitrary vector fields. In abstract index
notation, $R$ is denoted by $R_{ab}{}^{c}{}_d$, and $\RRR(X,Y)Z$ is
$X^aY^bZ^d R_{ab}{}^{c}{}_d$.  In dimensions $n\geq 3$ this can be
decomposed into the totally trace-free {\em Weyl curvature} $C_{abcd}$
and the symmetric {\em Schouten tensor} $\Rho_{ab}$ according to
\begin{equation}\label{Rsplit}
R_{abcd}=C_{abcd}+2\g_{c[a}\Rho_{b]d}+2\g_{d[b}\Rho_{a]c}~,
\end{equation}
where $[\cdots]$ indicates antisymmetrisation over the enclosed
indices.  In \nn{Rsplit}, $\Rho_{ab}$ is a trace modification of the
Ricci tensor given by
\begin{equation}\label{SchoutenDef}
  \Rc_{ab}=(n-2)\Rho_{ab}+\J\g_{ab}~.
\end{equation}
Here
${\Rc}_{ab}=R_{ca}{}^c{}_b$, and $\J=\Rho_a{}^a$. In dimensions 2 and
3 the Weyl tensor $C_{abcd}$ vanishes identically by dint of its
symmetries.  In dimension~2, \nn{SchoutenDef} does not define a
Schouten tensor in the sense we require in this work.  By adding
additional structure (a M\"{o}bius structure \cite{Calderbank}),
however, one may define a Schouten tensor in dimension~2 in such a way
that that it has conformal properties similar to the Schouten tensor
in higher dimensions.

\subsection{Riemannian hypersurfaces}\label{RH}

In this subsection, we discuss covariant de\-riv\-atives and various
structures on $\sm$ and relate them to covariant derivatives and
structures on $M$.  These ideas will facilitate the understanding of
some of the example formulae in Section~\ref{exs}, below. This
material is standard and appears in such sources as \cite{Spivak},
Chapter~1, but we develop it here to set notation fit with our current
approach via defining functions, cf.\ \cite{CG,GW-willmore,Grant,
  Stafford}.

To begin, let $\sm$ denote a hypersurface in $\Mg$ as in
Section~\ref{IntroSect}, and let $\dfn$ denote any local defining
function for $\sm$, as in Section~\ref{basics}.  We are concerned with
local theory here, so without loss of generality, we may assume that
$d\dfn$ is nowhere zero.  Let $\Nv_a$ be as in \nn{pndef}.  Recall
that the level sets of $\dfn$ define a foliation $\sm_{\dfn}$, with
$\sm=\sm_{0}$.  For each leaf $\dfn=\dfn_0$, the embedding $\iota:
\sm_{\dfn_0} \to M$ induces an injective bundle map $\iota_*:
T\sm_{\dfn_0} \to TM$, and we shall simply identify $T\sm_{\dfn_0}$
with its image. Dually, $T^*\sm_{\dfn_0}$ is naturally a quotient of
$T^*M $.  The bundle epimorphism $T^*M\to T^*\sm_{\dfn_0}$ is split by
the metric, and so $T^*\sm_{\dfn_0}$ is naturally identified with the
subbundle of $T^*M$ consisting of 1-forms orthogonal to the conormal
bundle to $\sm_{\dfn_0}$.  We shall thus use the same abstract indices
for $T\sm_{\dfn_0}$ as we do for $TM$.  We will, however, use a bar to
denote objects intrinsic to $\sm$ or to any leaf of the foliation.
Thus $\ice^a$ (resp.\ $\ice_a$) will denote the subbundle of $\ce^a$
defined by the foliation (resp.\ the subbundle of 1-forms annihilating
$n^a$), the spaces of smooth sections of these, or the restrictions of
these objects to a leaf $\sm_{\dfn_0}$ (for any $\dfn_0\in \mathbb{R}$
where the foliation is defined).  For example, $u_a\in \ce_a$ is a
section of $\ice_a$ if and only if $u_a n^a=0$. In particular, this
applies along $\sm=\sm_0$, where our interest really lies, and any
section of $\ice^a$ (resp.\ $\ice_a$) along $\sm $ will be assumed to
be the restriction to $\sm$ of a section of $\ice^a$ (resp.\ $\ice_a$)
defined in a neighbourhood of $\sm$. This idea will be extended to
tensor powers in an obvious way with little further mention.

Next, let $\PP{a}{b}: =\delta_a{}^b-\Nv_a\Nv^b$. As a section of $\End
(TM)$, this defines projections $TM\to T\sm_{\dfn_0}$ and $T^*M\to
T^*\sm_{\dfn_0}$ in the obvious way. Thus, for example, the formula
$\ig_{ab}:=\g_{ab}-n_an_b$ defines a symmetric 2-tensor that restricts
to give an induced intrinsic metric along any
$\sm_{\dfn_0}$. Similarly $\ig^{ab}=\g^{ab}-n^a n^b$, which is
consistent with raising indices using the ambient metric $\g^{ab}$. We
write $\ilc$ for the corresponding intrinsic Levi-Civita connection
along the leaves, and $\iRR$ denotes the corresponding Riemannian
curvature. Although we shall be finally interested in these quantities
only along $\sm=\sm_0$, it is convenient to have fixed an extension
off $\sm$ in this way; this is consistent with our treatment of vector
fields and tensors in general, as discussed above. They depend
smoothly on points in the foliated neighbourhood of $\sm$.

Along $\sm$, and indeed along each leaf $\sm_{\dfn_0}$, the second
fundamental form of the embedding, $L_{ab}$, is given by the formula
$\sff_{ab}=\PP{a}{c}\nd_c\Nv_b$.  (Note that the sign of $\sff$ here
differs from that of many sources, including \cite{Spivak}.)  By the
Weingarten equations  $\sff_{ab}$ is
symmetric.  An easy exercise shows that $n^aL_{ab}=0$, so $L$ defines
a smooth section of $\ice_{(ab)}$, the second symmetric power of
$\ice_a$.  Let $V^a,T^b\in\ice^a$ be given.  Then
$\sff_{ab}V^aT^b=-\g(\Nv^a,\nd_{V}T)$ along $\sm$, again by the
Weingarten equations.  Thus if $d\dfn$ is compatible with a given
orientation on $\sm$, then $\sff_{ab}$, and hence also the mean
curvature $\mc$, are independent of the choice of the defining
function $\dfn$.
The metric {\em trace-free} part of $L$, denoted by $\sffo$, is given
by $\sffo_{ab}=\sff_{ab}-H \ig_{ab}$.

An easy computation shows that $\mc=(1/(n-1))\nd_a\Nv^a$.
Hypotheses~\ref{NaturalHyp}, below, will refer to a modified mean
curvature tensor $\mmc$ given by $\mmc=\nd_a\Nv^a$.  The purpose of
defining $\mmc$ in this way is to obtain a curvature given by a
symbolic formula which involves $\nd_a$ and $\Nv^a$ but not the
dimension $n$.

It is easily verified that $\ilc$ is
related to the ambient Levi-Civita connection $\nabla$ by the {\em
  Gauss formula}
\begin{equation}\label{gf}
\ilc_a{}V^b=\PP{a}{c}\nd_{c}V^b+\Nv^bV^{c}\sff_{ac}~,
\end{equation}
where $V^a\in\ice^a$; in particular
$\ilc_aV^b=\PP{a}{c}\PP{d}{b}\nd_cV^d $. From this follows the
classical Gauss equation, which we give in our current notation:
%
%
\begin{proposition}\label{Gauss}
Let $\iR_{ab}{}^{c}{}_{d}$ denote the intrinsic Riemannian curvature
tensor.  Then
\[
\iR_{abcd}=
\PP{a}{i}\PP{b}{j}\PP{c}{k}\PP{d}{l}
R_{ijkl}+ L_{ac}L_{bd}-L_{ad}L_{bc}~.
\]
\end{proposition}
%

\subsection{Conformal structures}\label{confsec}

A conformal geometry is a manifold of dimension at least 2 equipped
with a conformal structure, i.e.\ an equivalence class $\cc$ of
Riemannian metrics such that if $\g,\,\hatg\in\cc$, then
$\hatg=e^{2\Up}\g$ for some $\Up\in\ce$.  Our operator constructions
will require several results and techniques from conformal geometry.
For further details see \cite{BEGo,CapGoamb,GP-CMP}, or \cite{CG} for
a recent overview.

We begin by interpreting a conformal structure $\cc$ as a ray
subbundle $\amQ\subseteq S^2T^*M$ whose fibre over $x\in M$ consists
of the set of all metrics at $x$ which are conformally related to some
given metric $\g$ at the point $x$.  The principal bundle $\pi:\amQ\to
M$ has structure group $\Bbb R_{>0}$, so for any $w\in\Reals$, the
representation ${\Bbb R}_{>0}\ni x \mapsto x^{-w/2}\in {\rm End}(\Bbb
R)$ induces a natural (oriented) line bundle on $M$ that we term the
bundle of \textit{conformal densities} of weight $w$.  We let $\ce[w]$
denote the space of sections of this bundle or the bundle itself.
There is a tautological section $\bg_{ab}$ of $S^2T^*M\otimes \ce[2]$
that is termed the {\em conformal metric}.  Similarly, $\bg^{-1}$ is a
section $\bg^{ab}$ of $S^2TM\otimes\ce[-2]$.  Henceforth, $\bg_{ab}$
and $\bg^{ab}$ will be the default objects that will be used to
identify $TM$ with $T^*M\otimes\ce[2]$ and to raise and lower indices
associated to these bundles (even when a metric $\g\in \cc $ has been
chosen).  As a consequence, tensors, hypersurface invariants, and
natural differential operators will now carry weights.  For example,
if $R_{ab}{}^c{}_{d}$ is the Riemannian curvature tensor associated to
a metric $\g\in\cc$, then $R_{ab}{}^c{}_{d}$ will have weight $0$.
Similarly, $R_{abcd}$ will denote $\bg_{ce}R_{ab}{}^e{}_d$, and this
tensor will have weight $2$.  The weights of other such curvature
tensors will be evident from their definitions.  We will assign
weights to $\Nv_{a}$, $\sff_{ab}$, $\mc$, and $\tfsff{ab}$ in
Section~\ref{confhy}.  If $P$ is a natural differential operator
acting between densities, then for some $\lowerw\in\Reals$ and all
conformal manifolds $\Mcc$ (or all conformally flat conformal
manifolds $\Mcc$) of appropriate dimensions, $P$ will map $\ce[w]$ to
$\ce[w-\lowerw]$.  For example, for any $f\in\ce[w]$, we have $\Delta
f=\bg^{ab}\nd_a\nd_bf\in\ce[w-2]$.

A metric $\g\in\cc$ is equivalent to a positive section $\xig$ of
$\ce[1]$ via the relation $\g_{ab}=\xigw{-2}\bg_{ab}$.  We say that
$\xig$ is the \textit{scale density} associated to $\g$.  Let
$w\in\Reals$ and a section $\sigma$ of $\ce[w]$ be given.  We may
write $\sigma=(\xig)^wf$ for some $f\in\ce$.  It is easily verified
that the Levi-Civita connection $\nd$, of $\g$, acting on $\ce[w]$, is
the connection mapping $\sigma$ to $(\xig)^wdf$, where $d$ is the
exterior derivative.  Let $\ce_b[w]=\ce_b\otimes\ce[w]$, and let
$\widehat{\nd}$ denote the Levi-Civita connection associated to the
metric $\hatg=e^{2\Upsilon}\g$.  Then for all $\mu_b\in\ce_b[w]$,
\begin{equation}\label{NabTran}
\widehat{\nd}_a\mu_b
=
\nd_a\mu_b+(w-1)\Upsilon_a\mu_b-\Upsilon_b\mu_a+\bg_{ab}\Upsilon^c\mu_c~,
\end{equation}%
where $\Upsilon_a=\nd_a\Upsilon$.

On a general conformal manifold there is no distinguished connection
on $TM$.  There is, however, a canonical conformally invariant
connection on a slightly larger bundle, and this is called the
(conformal) {\em tractor connection} \cite{BEGo}.  It is linked, and
equivalent to, the normal conformal Cartan connection of Elie Cartan
\cite{CGtams}.  On a conformal manifold $\Mcc$ of dimension $n\geq 3$,
let $\ct$ (or $\ct^A$ as the abstract index notation) denote the
(standard) tractor bundle.  This bundle is a canonical rank $n+2$
vector bundle equipped with the canonical (normal) tractor connection
$\nabla^\ct$.  This connection is conformally invariant.  We usually
write $\nd$ instead of $\nabla^{\ct}$.  We let $\tbPhi$ denote any
tensor power of $\ct$, including $\ce$.  To distinguish different (or
potentially different) powers of $\ct$, we write $\tbnPhiOne$ and
$\tbnPhiTwo$.  Let $\tbPhi[w]=\tbPhi\otimes\ce[w]$.  Also let $\ct$,
$\tbPhi$ and $\tbPhi[w]$ denote the spaces of sections of these
bundles.  A choice of metric $\g\in\cc$ determines a splitting of
$\ct$, i.e.\ an isomorphism
\begin{equation}\label{split}
\mathcal{T} \stackrel{\g}{\cong} \ce[1]\oplus\ce_b[1]\oplus
\ce[-1]~.
\end{equation}
We may write $T\gequal(\si,~\mu_b,~\rho)$ to indicate that $T$ is an
invariant section of $\ct$, and $(\si,~\mu_b,~\rho) $ is its image
under the splitting given by \nn{split}.  In general, a conformally
related metric $\hatg$
determines a different splitting of
$\ct$.  If $T\gequal(\si,~\mu_b,~\rho)$, then
\begin{equation}\label{transf}
T\ghatequal
(\si,\,\mu_b+\si\Up_b,\,\rho-\bg^{cd}\Up_c\mu_d-
\tfrac{1}{2}\si\bg^{cd}\Up_c\Up_d)~.
\end{equation}

To facilitate our computations, we introduce three algebraic splitting
operators,
\[
Y^A\in \ce^A[-1]~, \quad 
Z^{Ab}\in\tbn^{Ab}[-1]:=\tbn^A\otimes \ce^b\otimes \ce [-1]~, \quad 
X^A\in \ce^A[1]~,
\]
which administer the isomorphism \nn{split} determined by the metric
$\g\in\cc$.  If $T^A\gequal(\sigma,\mu_b,\rho)$, then $T^A= \sigma
Y^A+\mu_{b}Z^{Ab}+\rho X^A$.  Note that \nn{transf} determines the
transformations of $Y^A$, $Z^{Ab}$, and $X^A$ under conformal change
of metric. These are easily computed and are given explicitly in
\cite{GP-CMP}.

There is a conformally invariant {\em tractor metric} $\tm$ on $\tbn$
which is preserved by $\nd^{\ct}$.  We let $\tmco$ denote the
co-metric associated to $\tm$ on the dual bundle to $\tbn$.  In terms
of the splitting operators, the tractor metric is given by
Figure~\ref{TrIProd}.
\begin{figure}
$$
\begin{array}{l|ccc}
& Y^A & Z^{Ac} & X^{A}
\\
\hline
Y_{A} & 0 & 0 & 1
\\
Z_{A}{}^b & 0 & \bg^{bc} & 0
\\
X_{A} & 1 & 0 & 0
\end{array}
$$
\caption{The tractor metric}
\label{TrIProd}
\end{figure}
In a symbolic tractor expression with tractor indices, one may
eliminate all references to $\tm$ and $\tmco$ as follows.  First, if
one index of $\tm$ or $\tmco$ is contracted with an index of some
other tractor, one may eliminate the reference to $\tm$ or $\tmco$ by
raising or lowering the index of this other tractor.  On the other
hand, if $\tm$ or $\tmco$ has only free indices, then one may express
$\tm$ and $\tmco$ in terms of the splitting operators $X$, $Y$, and
$Z$.  For example, $\tm_{AB}=Z_A{}^cZ_{Bc}+X_AY_B+Y_AX_B$, as noted in
\cite{GP-CMP}.

The tractor connection is usefully encoded in the formulae for the
tractor Levi-Civita coupled derivatives of the splitting operators:
\begin{equation}\label{connids}
\begin{array}{rcl}
\nd_aY^A=\Rho_{ab}Z^{Ab},&
\nd_aZ^{A}{}_{b}=-\Rho_{ab}X^A-\bg_{ab}Y^A,&
\nd_aX^A=Z^A{}_{a}~.
\end{array}
\end{equation}
The curvature $\Omega$ of the tractor connection is defined by
\begin{equation}\label{TracCurv}
(\nabla_i\nabla_j-\nabla_j\nabla_i)V^A=\Omega_{ij}{}^A{}_BV^B
\end{equation}
for $V^A\in\tbn^A$.   A basic
computation using \nn{connids}  shows that
\begin{equation}\label{TracCurveOne}
\tcurve_{ij}{}^A{}_{B}=
\C_{ij}{}^{k}{}_mZ^{A}{}_kZ_{B}{}^m
+2(\nd_{[i}\Rho_{j]}{}^k)Z^A{}_{k}X_B
-2(\nd_{[i}\Rho_{j]m})X^AZ_B{}^{m}~.
\end{equation}
We will also use the conformally invariant tractor $W$-curvature
$W_{ABCE}$ as described in e.g.\ \cite{GP-CMP}:
\begin{equation}\label{TracCurveTwo}
W_{ABCE}=(n-4)Z_{A}{}^{a}Z_{B}{}^b\tcurve_{abCE}
-2X_{[A}Z_{B]}{}^b\nd^p\tcurve_{pbCE}~.
\end{equation}
From this formula and well known results it follows easily that $\cc$
is conformally flat if and only if $W_{ABCE}=0$.

The notation $\psi:\tbnPhiOne[w]\rightarrow\tbnPhiTwo[w-c]|_{\sm}$ or
$\psi:\tbnPhiOne[w]\rightarrow\tbnPhiTwo[w-c]$ will indicate that
$\psi$ is a family of natural differential operators parametrised by
$w$ and $n$.  Here $\lowerw$ is a real constant and $w$ is a real
number.
The family $\psi$ is a rule given by a universal symbolic formula.
This rule defines an operator for all $\Mcc$ of dimension $n$, all
bundles $\tbnPhiOne$, and all $w\in\Reals$.  (We may require $\Mcc$ to
meet various conditions, and $\tbnPhiTwo$ will depend on $\tbnPhiOne$
and $\psi$.)  If $w$ appears in the symbolic formula for $\psi$, we
set $w$ equal to the weight of the bundle on which $\psi$ acts, unless
we explicitly indicate otherwise.  Since $\psi$ is given by a
universal symbolic formula, we will sometimes refer to $\psi$ as an
``operator'' rather than a family of operators.

One important family of conformally invariant natural operators on
weighted tractors is the family
$D:\tbPhi[w]\rightarrow\ct^A\otimes\tbn^{\Phi}[w-1]$ defined as
follows:
\begin{equation}\label{tractorD}
D^AV=
w(n+2w-2)Y^AV+
(n+2w-2)Z^{Ab}\nd_bV
-X^A(\Delta+w\J)V~.
\end{equation}
Cf.\ \cite{BEGo}.  For developments of $D$ and proofs of its conformal
invariance, see \cite{Esrni,G-srni1}.  Another important family of
natural operators is the family
$\bop:\tbPhi[w]\rightarrow\tbPhi[w-2]$ given by
\begin{equation}\label{YamOp}
\bop V=(\Delta+w\J)V~.
\end{equation}
If $w=1-n/2$, then $\bop$ is the Yamabe operator, which is conformally
invariant.  We will use $D$ and $\Box$ in our main operator
constructions.

\subsection{Conformal hypersurfaces}\label{confhy}
Let a hypersurface $\iota:\sm\to M$ be given.  In this subsection, we
present the necessary elements of basic conformal hypersurface
geometry.

Let $\g,\,\hatg\in\cc$ be given.  These metrics induce conformally
related metrics $\ig$ and $\hatig$ on $\sm$, and so $\cc$ induces a
conformal structure $\ic$ on $\sm$. We term $\ic$ the {\em intrinsic
  conformal structure} of $\sm$.  If $n\geq 4$, this conformal
structure determines an intrinsic version of each of the constructions
and results from Section~\ref{confsec} in the usual way.  (We treat
the $n=3$ case in Section~\ref{cgt}, below.)  Let $\ibg$ denote the
intrinsic conformal metric, and let $(\tbS, \ilc^{\tbSsup})$ denote
the intrinsic tractor bundle and its connection on $\sm$. (In fact, we
shall usually write $\ilc$ rather than $\ilc^{\tbSsup}$.)  The
conformally invariant (and $\ilc$-parallel) metric on $\tbS$, the
intrinsic tractor metric, shall be denoted $\ih$ and has signature
$(n,1)$.  We identify $\ce[w]|_\Sigma$ with $\ice[w]$ in the obvious
way.

Now let a local defining function, $\dfn$, for $\sm$, as in
Section~\ref{basics}, be given.  Henceforth, we let
\begin{equation}\label{ndef}
\Nv_a:= d\dfn/|d\dfn|_{\bg}~.
\end{equation}
Thus $\Nv_a$ is now a weight $1$ conformally invariant conormal field
for $\sm$ and, more generally, for the $\sm_{\dfn}$ foliation.  Let
$\ibg_{ab}:=\bg_{ab}-n_an_b$.  This tensor extends the intrinsic
conformal metric of $(\sm,\ic)$ to a neighbourhood of $\sm$; its
restriction gives the conformal metric on each leaf $\sm_{\dfn_0}$.
Given $\g\in \cc$, we again define $\sff_{ab}$, $\mc$, and
$\tfsff{ab}$ on a neighbourhood of $\sm$ as in Section~\ref{RH},
above.  This time, however, we replace $\ig$ with $\ibg$, and we use
\nn{ndef} to define $\Nv_a$.  As a consequence, $\sff_{ab}$, $\mc$,
and $\tfsff{ab}$ will have weights $1$, $-1$, and $1$, respectively.
This convention will hold for the remainder of this paper.  In fact,
\textit{all} formulae will now carry conformal weights, unless we note
otherwise.  This use of weights simplifies the transformation of the
formulae under conformal rescaling.  The formulae and results from
Section~\ref{RH} carry over to the present setting in the obvious way.

These remarks apply to natural differential operators, of course, and
this leads to the following proposition.%
\begin{proposition}\label{ReplaceCCov}
Let a pair $(w_1,w_2)\in\Reals^{2}$ and natural differential operators
$P:\tbPhi[-w_1]\rightarrow\tbPhi[-w_2]$ and
$P':\tbPhi\rightarrow\tbPhi$ be given.  Finally, suppose that
$P=(\xig)^{-w_2}P'(\xig)^{w_1}$ for all $\g\in\cc$.  Then $P'$ is
conformally covariant of bidegree $(w_1,w_2)$ if and only if $P$ is
conformally invariant.  A similar statement holds for operators
mapping $\tbPhi[-w_1]$ to $\left.\tbPhi[-w_2]\right|_{\sm}$ and
$\tbPhi$ to $\left.\tbPhi\right|_{\sm}$.
\end{proposition}
\begin{proof}
Note that $(\xi^{\hatg})^w=e^{-w\Upsilon}(\xig)^w$ for any
$w\in\Reals$.  The result thus follows from an elementary argument.
\end{proof}
\begin{remark}\label{IdentificationRem}
Proposition~\ref{ReplaceCCov} allows us to identify conformally
covariant and conformally invariant operators.  In the main new
operator constructions of this paper, we will work with conformally
invariant operators, so for the rest of this paper, we will usually
replace property~(2) of Problem~\ref{GenProblem} with the following
equivalent property: $P^{\g}_{w,K}:\ce[w]\rightarrow\ice[w-K]$ is
conformally invariant.  Here $K$ is as in Problem~\ref{GenProblem}.
\end{remark}
\subsection{Tractors and conformal Gauss theory} \label{cgt} We will
subsequently exploit a conformally invariant replacement for the Gauss
formula \nn{gf}. Here we develop this machinery. In the following, we
work along $\sm$, but the discussion applies to any leaf of the
foliation without adjustment.  Thus all quantities defined are
extended into a neighbourhood of $\sm$.  We assume $n\geq 3$, except
as noted.

An elementary computation shows that
\begin{equation}\label{full}
\sffh_{ab}=\sff_{ab}+\ibg_{ab}\Upsilon^{c}\Nv_{c}~,
\end{equation} 
whence 
\begin{equation}\label{meancurhat}
\mch=\mc+\Up_a\Nv^a~.
\end{equation}
(Since $\mc$ and $\Nv^a$ now carry weights, \nn{meancurhat} differs
from the conformal transformation law for $\mc$ that we discussed in
Section~\ref{IntroSect}, above.)  It follows from \nn{full} that
$\tfsff{ab}$ is conformally invariant.  From \nn{meancurhat} and
\nn{transf} it follows that
\begin{equation}\label{NormalTractor}
\Nt^A=\Nv_bZ^{A}{}^b-HX^A
\end{equation}
is conformally invariant along $\sm$ and, more generally, along each
leaf of the foliation $\sm_\dfn$. This is the \textit{normal tractor}
as defined in \cite{BEGo}.  It has conformal weight 0.  From
$\bg^{ab}n_an_b=1$, it follows that $h^{AB}N_AN_B=1$.  Let $\tbp$ be
the (conformally invariant) subbundle of $\tb|_{\sm}$ whose fibre is
the orthogonal complement (with respect to $h$) of $\Nt^A$.  If $n\geq
4$, then $\tbS$ exists and has the same rank as $\tbp$.  This suggests
the following proposition.  This proposition follows
\cite{Grant,Stafford}, which in turn follow an equivalent argument in
\cite{BrGonon}.
\begin{proposition}\label{Tisom}
  Suppose $n\geq 4$.  Then along any leaf of the foliation, we may
  canonically and isometrically identify the bundles $\tbS$ and
  $\tbp$. \end{proposition}
\begin{proof}
Let $\g\in\cc$ be given, and along any leaf of the foliation $\sm_{\dfn}$,
define an embedding $\trem:\tbS\hookrightarrow\ct$ as follows:
\begin{equation}\label{Tembed}
\tbS \ni T\igequal(\isigma,\,\imu_b,\,\irho) \mapsto
(\isigma,\,\imu_b+\isigma
H\Nv_b,\,\irho-\textstyle{\frac{1}{2}}\mc^2\isigma)\gequal iT~.
\end{equation}
The range of $\trem$ is clearly orthogonal to $N^A$.  By \nn{transf},
\nn{meancurhat}, and a basic calculation, it follows that
$\trem:\tbS\rightarrow\tbp$ is a conformally invariant bundle
isomorphism compatible with  the tractor metrics.
\end{proof}

In the case $n=3$, we can still define the intrinsic tractor bundle
$\tbS$ in the usual way (i.e.\ as in higher dimensions), or we can
equivalently identify $\tbS=\tbp$ along $\sm$ as in \cite{BrGonon}. In
any case, we then define the intrinsic tractor connection $\ilc$ and
tractor $D$-operator, $\iD$, on $\tbS$ as in that reference.

For all $n\geq 3$, we shall henceforth identify $\tbS$ and $\tbp$.  We
can thus use the same abstract indices for the intrinsic and ambient
tractor bundles; sections of $\tbS$ are characterised by orthogonality
to $N^A$. In an obvious way, these conventions are extended to the
dual tractor bundle, tensor products, and so forth.

The normal tractor gives a conformally invariant tractor analogue of a
Riemannian hypersurface conormal.  Along any hypersurface $\sm$, the
ambient tractor bundle $\ct$ decomposes directly and orthogonally into
$\tbS \oplus \cN$, where $\cN$ is the tractor subbundle generated by
the normal tractor field $N^A$. Note that
$\Pi_A{}^B:=\delta_A{}^B-N_AN^B$, as a section of $\End (\ct)$, gives
the projection $\ct\to \tbS$, and using the abstract index notation,
we write $\Pi_A{}^B: \ct^A \to \tbS^B$.%

Although Proposition~\ref{Tisom} is not surprising, there is a slight
subtlety involved.  Specifically, in the case $n\geq 4$, the
proposition shows that the tractor splitting \nn{split} determined by
$\ig\in \ic$ is {\em not} related to the splitting determined by
$\g\in\cc$ in the most na\"{\i}ve way.  (But it \textit{is}, if, on a
particular hypersurface, we work in a scale with $H\equiv 0$.  Compare
to \cite{BrGonon,Go-al}.)

By Proposition~\ref{Tisom} and the above discussion, we have the
following.
%
%
\begin{proposition}\label{embedtmetric}
$\hS_{AB}=\tm_{AB}-\Nt_{A}\Nt_{B}$ and
  $\hS^{AB}=\tm^{AB}-\Nt^{A}\Nt^{B}$ for all $n\geq 3$.
\end{proposition}
%
%

Now suppose $n\geq 4$.  Let $\YS^A$, $\ZS^{Ab}$, and $\XS^A$, in
$\tb[-1]|_{\sm}$, $\tb{}^{b}[-1]|_{\sm}$, and $\tb[1]|_{\sm}$,
respectively, denote the algebraic splitting operators associated to
$\tbS$ as determined by the induced metric $\ig$.  From \nn{Tembed},
we see that
\[
\YS^{A}=Y^A+\mc\Nv_bZ^{Ab}-\textstyle{\frac{1}{2}}\mc^2X^A,\hspace{3.0ex}
\ZS^{Ab}=\PP{c}{b}Z^{Ac},\hspace{3.0ex}
\XS^{A}=X^A~.
\]
Let a section $V^A$ of $\tbS$ be given, and suppose that
$V^A\igequal(\isigma,\,\imu_b,\,\irho)$.  Then
\[
V^A=
\isigma\YS^A+\imu_b\ZS^{Ab}+\irho\XS^A~.
\]
On the other hand, if $V^A=\sigma Y^A+\mu_b Z^{Ab}+\rho
X^A\in\tbp$, then
\begin{equation}\label{L23Nov8a}
V^A
=
  \sigma\YS^A
  +\mu_a\PP{b}{a}\ZS^{Ab}
  +(\rho+\textstyle{\frac{1}{2}}\mc^2\sigma)\XS^{A}~.
\end{equation}
Equation \nn{L23Nov8a} describes the inverse of the isomorphism
$\tbS\rightarrow\tbp$ (as in \cite{Stafford}).

Suppose again that $n\geq 4$, and define a projected ambient
connection $\ptc$ on $\tbS$ as in \cite{Stafford} by letting
\[
\ptc_cT^A=\PP{B}{A}\PP{c}{e}\nd_eT^B
\]
for all sections $T^A$ of $\tbS$.  In \cite{Stafford}, Stafford
defines the tractor contorsion, which has the property that
\begin{equation}\label{tracgf}
\ilc_cT^B-\ptc_cT^B=\tcon_{cA}{}^BT^A
\end{equation}
for any section $T^B$ of $\tbS$.  By \cite{Stafford},
\begin{equation}\label{t-cont}
\tcon_{cA}{}^{B}=
X_AZ^{Bb}\cF_{bc} -
X^BZ_{A}{}^a\cF_{ac},
\end{equation}
where $\cF_{bc}: = \iRh_{bc}-\PP{b}{i}\PP{c}{j}\Rho_{ij}-\mc\sffo_{bc}
-\frac{1}{2}\ibg_{bc}\mc^2$
is the conformally invariant {\em Fialkow tensor}. 
Formulae~\nn{tracgf} and \nn{t-cont} will be useful in
Section~\ref{OpBases} below.
%
%

Some of the examples in Section~\ref{exs}, below, will refer to the
tractor second fundamental form.  The tractor second fundamental form
is a conformally invariant tractor prolongation
$\tsff_{AB}\in\tbS_{AB}[-1]$ of the second fundamental form
$\sff_{ab}$.  We use the definition given in \cite{Grant}.
\begin{definition}
In dimensions $n=\dim(M)\geq 4$, the \textit{tractor second
  fundamental form} is given by
\[
\tsff_{AB}=(n-3)\ZS_A{}^a\PP{a}{d}\nd_d\Nt_B
-\XS_{A}\ilc^a\PP{a}{d}\nd_d\Nt_B~.
\]
along $\sm$.
\end{definition}
The tractor second fundamental form is conformally invariant.  To see
this, begin by letting $w\in\Reals$ and $n\geq 4$ be given.  Define an
operator
$E_{AB}{}^{aC}:\ice_a\otimes\tbS_{C}[w]\rightarrow\tbS_{AB}[w-1]$ as
follows.  For any $T_{aC}\in\ice_a\otimes\tbS_{C}[w]$, let
\[
E_{AB}{}^{aC}T_{aC}=
(n+w-3)\ZS_{A}{}^aT_{aB}-\XS_A\ilc^aT_{aB}~.
\]
In this definition, $\ilc$ acts as the intrinsic Levi-Civita
connection on $\ceS_a[w]$ and as the intrinsic tractor connection on
$\tbS_C$.  An easy adaptation of \nn{NabTran} shows that
$E_{AB}{}^{aC}$ is conformally invariant.  A substitution then shows
that $\tsff_{AB}=E_{AB}{}^{aC}\PP{a}{d}\nd_{d}\Nt_C$.  Thus
$\tsff_{AB}$ is conformally invariant.

Various versions of the following proposition appear in
\cite{GW-willmore,Grant, Stafford}.
\begin{proposition}\label{tsfftwo}
  In dimensions $n\geq 4$,
\[
\begin{array}{ll}
\lefteqn{
\textstyle\tsff_{AB}=
(n-3)\ZS_A{}^a\ZS_{B}{}^{b}\sffo_{ab}
-2\,\frac{n-3}{n-2}\,\XS_{(A}\ZS_{B)}{}^{a}\ilc^b\sffo_{ab}}
&\\
&

+\XS_A\XS_B(
\sffo_{ab}\iRh^{ab}+\frac{1}{n-2}\ilc^a\ilc^b\sffo_{ab})~.
\hspace{35mm}
\end{array}
\]
Here $(A\ B)$ indicates symmetrisation over $A$ and $B$.
\end{proposition}
The following corollary summarises some of the important properties of
the tractor second fundamental form:
\begin{corollary}\label{LABProps}
In dimensions $n\geq 4$, $\tsff_{AB}$ is conformally invariant,
symmetric, and trace free.
\end{corollary}
%

%
%

%
%
%
\section{Bases of Operators}\label{OpBases}

One of the keys to our operator constructions will be to work with
bases of finite-dimensional vector spaces of natural hypersurface
operators.  In this section, we will construct such bases and derive
some results relating to our work with these bases.  The development
of these bases will involve families of
hypersurface operators $\psi$ which satisfy the following hypotheses:
\begin{hypotheses}\label{NaturalHyp}
The family $\psi$ is a family of natural hypersurface differential
operators given by a single universal symbolic formula.  This formula
is a polynomial in $\Nv_a$, $R_{ab}{}^c{}_d$, $\bg_{ab}$, $\bg^{ab}$,
the modified mean curvature $\mmc$, and Levi-Civita connection $\nd$.
In this formula, $\nd$ never explicitly hits $\Nv_a$ or $\mmc$.  If
$\bg_{ab}$ or $\bg^{ab}$ appears in the formula, then $a$ and $b$ are
contracted with indices on $\Nv_a$, $R_{ab}{}^c{}_{d}$, or $\nd_a$.
For some $\lowerw\in\Reals$, some infinite set $\Dimensions$ of
dimensions, all Riemannian manifolds $\Mg$ of all dimensions
$n\in\Dimensions$, all hypersurfaces $\sm$ in such manifolds, and all
$w\in\Reals$, the symbolic formula for $\psi$ determines an operator
$\psi:\ce[w]\rightarrow\ice[w-\lowerw]$.  In the polynomial formula
for $\psi$, the coefficients are real functions of $w$ and $n$ which
are polynomial in $w$ and rational in $n$.
\end{hypotheses}
The set $\Dimensions$ in Hypotheses~\ref{NaturalHyp} will be important
in our work with the spaces $\RealCoeff_{K,\lowerw,w}$,
$\RationalCoeff_{K,\lowerw,0}$, and $\Classes_{K,\lowerw,0}$, below.
The trace-free part, $\tfsff{ab}$, of the second fundamental form
appears in many of the example formulae in Section~\ref{exs}, below,
but these appearances of $\tfsff{ab}$ result from manipulation of
other computed formulae.

In Hypotheses~\ref{NaturalHyp} and throughout this section, we may
work with general Riemannian metrics.  All of the results of this
section continue to hold, however, if we assume that all metrics are
conformally flat.  In some situations, we will implicitly modify
Hypotheses~\ref{NaturalHyp} by adding the assumption that all metrics
are conformally flat.  We do this, for example, in the proof of
Theorem~\ref{cflatkey}, below.  Theorem~\ref{key}, below, implicitly
treats two cases, namely (1)~conformally flat metrics and (2)~general
metrics, including conformally flat metrics.  The proof of this
theorem will implicitly require two separate applications of the
theory of this section.  Some of our other proofs will also implicitly
treat the above two cases as separate cases.

Our development of operator bases and related results will require
some groundwork.  One key idea will be the idea of the
\textit{\masspoint} of an operator, tensor, tractor, or combination of
these.  The concept of \mass will allow us to construct
finite-dimensional vector spaces of operators and classes of
operators.  We will also use the concept of \mass in
Section~\ref{refine}, below, to estimate the order of a differential
operator.  We define \mass as follows.  Let $\mathcal{P}$ be a
symbolic formula which is polynomial in $\Nv_a$, $R_{ab}{}^c{}_d$,
$\bg_{ab}$, $\bg^{ab}$, the coupled Levi-Civita connection $\nd$, the
splitting operators $Y^A$, $Z^{Ab}$, and $X^A$, the tractor metric
$h$, its co-metric $h^{\#}$, and real powers of $\xig$.  Suppose that
the coefficients in this polynomial are real rational functions of $n$
which may depend polynomially on a weight $w$.  We will say that
$R_{ab}{}^c{}_d$ has \mass 2, $\nd$ and $Y^A$ have \mass 1, $\Nv^a$,
$\bg_{ab}$, $\bg^{ab}$, $Z^{Ab}$, $h$, $h^{\#}$, and powers of $\xig$
have \mass 0, and $X^A$ has \mass $-1$.  All coefficients in
$\mathcal{P}$ will have \mass $0$.  We define the \mass of any given
term of $\mathcal{P}$ to be the sum of the \masses of the expressions
that appear in the term.  In this paper, all of the terms in a given
symbolic formula $\mathcal{P}$ will always have the same \masspoint,
and we will say that this common \mass is the \mass of $\mathcal{P}$
(or the \mass of the operator, tensor, or tractor corresponding to
$\mathcal{P}$).

The use of \nn{connids}, \nn{TracCurv}, and Figure~\ref{TrIProd}
preserves \masspoint.  If one commutes covariant derivatives, one
typically generates terms containing curvatures, but these terms have
the same \mass as the original expression.  Recall that
$\sff_{ab}=\Pi_a{}^{c}\nd_c\Nv_b$.  Thus $\sff_{ab}$ has \mass 1, and
hence $\mc$ also has \mass 1.  From \nn{NormalTractor} it then follows
$\Nt^A$ has \mass 0.  Similarly, $\XS^A$ has \mass $-1$, $\ZS^{Ab}$
has \mass 0, $\delta$, $D^A$, $\ilc$, $\sffo_{ab}$, $\YS^A$,
$\tcon_{cA}{}^{B}$, and $\tsff_{AB}$ have \mass $1$, and $\delta_2$,
$\Delta$, $\Box$, $\iR_{abcd}$, $C_{abcd}$, $\Rho_{ab}$, $\J$,
$\Omega_{ij}{}^A{}_B$, and $W_{ABCE}$ have \mass $2$.
In Section~\ref{exs}, below, the operators in
figures~\ref{ThirdOrderOperator} and \ref{GeneralWeight} have \masses
3 and 4, respectively.

In the proofs of propositions~\ref{RationalBasis} and \ref{ZeroAlln},
below, we will use a certain technical procedure which we now
describe.  Let a family of natural hypersurface operators $\psi$ be
given, and suppose that $\psi$ satisfies Hypotheses~\ref{NaturalHyp}.
Let $w$, $\lowerw$, $n$, and $\Dimensions$ be as in
Hypotheses~\ref{NaturalHyp}, and let $w=0$.  Suppose there is an
infinite set $S\subseteq\Dimensions$ such that in all dimensions $n\in
S$, $\psi:\ce[0]\rightarrow\ice[-\lowerw]$ is the zero operator.  Let
$m\in\Dimensions$ be given, and also let the following be given and
fixed: a Riemannian manifold $\Mg$ of dimension $m$, a hypersurface
$\sm$ of $M$, a point $p\in\sm$, and a smooth function $V$ on $M$.
Then $(\xig)^{\lowerw}(\psi V)(p)$ is given by a symbolic formula.  In
this symbolic formula, we have set $n$ equal to $m$, but suppose we
now replace $m$ with an arbitrary integer parameter $n$.  (In the
symbolic formula for $(\xig)^{\lowerw}(\psi V)(p)$, we still keep
$\Mg$, $\sm$, $p$, and $V$ fixed.)  The result is a real rational
function $f$ of $n$.  Of course $f(m)$ is equal to the original
numerical value of $(\xig)^{\lowerw}(\psi V)(p)$.  In the proofs of
propositions~\ref{RationalBasis} and \ref{ZeroAlln}, our plan will be
to show that $f(n)=0$ for infinitely many $n\in\Integers$.  The key to
doing this will be Proposition~\ref{Embedding}, below.  The idea is to
consider the action of a natural operator in a given dimension and
relate this action to the action of the operator in other dimensions.

Proposition~\ref{Embedding} will refer to the scale density $\xig$
associated to a Riemannian metric $\g$ as defined in
Section~\ref{confsec}, above.  Proposition~\ref{Embedding} will also
refer to the ``natural'' elevation of tensor indices.  For $\nd$ and
the unit conormal field, this means that the index is down.  The
Riemannian curvature tensor naturally has one index up and three
indices down.  For $\g$, it is natural to have both indices up or both
down.
%
%
\begin{proposition}\label{Embedding}
  Let a family of natural hypersurface operators $\psi$ be given, and
  suppose that $\psi$ satisfies Hypotheses~\ref{NaturalHyp}.  Let
  $\lowerw$, $\Dimensions$, and $w$ be as in
  Hypotheses~\ref{NaturalHyp}, and suppose that $w$ and $n$ do not
  appear in the universal symbolic formula for $\psi$.  In this
  formula, suppose that all indices appear at their natural
  elevations, and consider $(\xig)^{\lowerw}\psi$ as acting on
  densities of weight zero only.  For simplicity, we will write $\psi$
  instead of $(\xig)^{\lowerw}\psi$.  Let $m,\,m'\in\Dimensions$ be
  given, and suppose that $m<m'$.  Finally, let a Riemannian manifold
  $\Mg$ of dimension $m$, a hypersurface $\sm$ of $M$, and $V\in
  C^{\infty}(M)$ be given.  Then there is a Riemannian manifold
  $(M',\amr)$ of dimension $m'$, a hypersurface $\sm'$ of $M'$, and
  $V'\in C^{\infty}(M')$ such that the following holds: For all
  $p\in\sm$, there is a $p'\in\sm'$ such that
\begin{equation}\label{EmbeddingRelation}
  (\psi V)(p)=(\psi V')(p')~.
\end{equation}
\end{proposition}
\noindent The proof of Proposition~\ref{Embedding} appears in
Appendix~\ref{EmbedApp}, below.

We now turn to the definition of two general families of natural
operators and one general family of operator classes.  These families
will lead to the bases that we mentioned above.  Let
$K\in\Integers_{>0}$, $c, w\in\Reals$, and an infinite set
$\Dimensions\subseteq\Integers_{>0}$ be given.  Let
$\RealCoeff_{K,c,w}$ denote the set of all natural operators
$\psi:\ce[w]\rightarrow\ice[w-c]$ of \mass $K$ defined in all
dimensions $n\in\Dimensions$.  We will assume that every
$\psi\in\RealCoeff_{K,c,w}$ satisfies all of the properties in
Hypotheses~\ref{NaturalHyp} with the following exception: The value of
$w$ is always equal to the fixed value given above.  For this choice
of $w$, $\psi$ always maps $\ce[w]$ to $\ice[w-\lowerw]$.  We will
also assume that the coefficients in the symbolic formula for $\psi$
are independent of $n$.  These coefficients are thus real constants.
We emphasise that every $\psi\in\RealCoeff_{K,c,w}$ is a rule which
determines a hypersurface operator for every hypersurface $\sm$ in
every Riemannian manifold $\Mg$ of every dimension $n\in\Dimensions$.
We never consider $\psi$ in dimensions $n\notin\Dimensions$.  If
$\psi\in\RealCoeff_{K,c,w}$ is zero, this means that $\psi$ is the
zero operator in all dimensions $n\in\Dimensions$.  It is clear that
$\RealCoeff_{K,c,w}$ is a vector space over $\Reals$.
\begin{proposition}\label{IsomorphismProp}
Let $K\in\Integers_{>0}$ and $c$, $w_1$, and $w_2\in\Reals$ be given,
and define a mapping
$\Phi:\RealCoeff_{K,c,w_1}\rightarrow\RealCoeff_{K,c,w_2}$ as follows.
For all $\psi\in\RealCoeff_{K,c,w_1}$, let
$\Phi(\psi):\ce[w_2]\rightarrow\ice[w_2-c]$ be the operator given by the
universal symbolic formula for $\psi$.  Then $\Phi$ is a vector space
isomorphism.
\end{proposition}
\begin{proof}
Note that $\Phi$ is clearly linear and surjective.  To show that
$\Phi$ is injective, let $\psi\in\RealCoeff_{K,\lowerw,w_1}$ be given,
and suppose that $\Phi(\psi)=0$.  Also let $V\in\ce[w_1]$ be given.
Since powers of $\xig$ are parallel, it follows that
\[
\psi
V=(\xig)^{w_1-w_2}\psi(\xig)^{w_2-w_1}V=
(\xig)^{w_1-w_2}\Phi(\psi)(\xig)^{w_2-w_1}V=0~.
\]
Thus $\psi$ is the zero element of $\RealCoeff_{K,\lowerw,w_1}$, and
$\Phi$ is injective.
\end{proof}

Now let $K$, $\lowerw$, $w$, and $\Dimensions$ again be as above.  Let
$\RationalCoeff_{K,\lowerw,0}$ denote the set of all natural
differential operators $\psi:\ce[0]\rightarrow\ice[-\lowerw]$ of mass
$K$ which satisfy all of the properties of Hypotheses~\ref{NaturalHyp}
except as follows.  First, $w$ will always be zero.  Thus every
$\psi\in\RationalCoeff_{K,\lowerw,0}$ will always map $\ce[0]$ to
$\ice[-\lowerw]$, and the coefficients in the universal symbolic
polynomial formula for $\psi$ will be real rational functions of $n$.
Also, in the universal symbolic formula for any given
$\psi\in\RationalCoeff_{K,\lowerw,0}$, we will allow the coefficients
to be singular in some dimensions $n\in\Dimensions$.  We will only
consider $\psi$ in dimensions $n\in\Dimensions$.  It may be helpful to
think of $\RationalCoeff_{K,\lowerw,0}$ as a set of symbolic operator
formulae.

Let $\FieldOfFuns$ denote the field of all real rational functions of
$n$.  We would like to view $\RationalCoeff_{K,c,0}$ as a vector space
over $\FieldOfFuns$, but singularities in elements of $\FieldOfFuns$
make this difficult.  To overcome this difficulty, we define a new
space as follows.  For each $\psi_1,\psi_2\in\RationalCoeff_{K,c,0}$,
we identify $\psi_1$ and $\psi_2$ if there is an $n_0\in\Dimensions$
such that in all dimensions $n\in\Dimensions\cap\Integers_{\geq n_0}$,
$\psi_1$ and $\psi_2$ are both regular and $\psi_1=\psi_2$.  Let
$\Classes_{K,c,0}$ denote the set of equivalence classes that result
from this identification.  Then $\Classes_{K,c,0}$ is clearly a vector
space over $\FieldOfFuns$.  For all $\psi\in\RationalCoeff_{K,c,0}$,
let $[\psi]$ denote the equivalence class of $\psi$ in
$\Classes_{K,c,0}$.  For any $\psi\in\RationalCoeff_{K,\lowerw,0}$,
$[\psi]$ is the zero element of $\Classes_{K,\lowerw,0}$ if and only
if there is an $n_0\in\Dimensions$ such that $\psi$ is the zero
operator in all dimensions $n\in\Dimensions\cap\Integers_{\geq n_0}$.

We will find bases for $\RealCoeff_{K,\lowerw,w}$ and
$\Classes_{K,\lowerw,0}$.  We begin with $\RealCoeff_{K,\lowerw,w}$.
Let $K$, $\lowerw$, and $w$ be given, as above.  Let $S$ denote the
set of all monomials $\RiemInv_{\g}$ in $\Nv_a$, $R_{ab}{}^c{}_d$,
$\bg_{ab}$, $\bg^{ab}$, $\mmc$, and $\nd$ which have the following
properties.  First, each $\RiemInv_{\g}\in S$ has mass $K$ and
coefficient $1$.  Each $\RiemInv_{\g}$ determines a natural operator
$\psi:\ce[w]\rightarrow\ice[w-\lowerw]$, and this operator $\psi$
satisfies Hypotheses~\ref{NaturalHyp} (but with the value of $w$
fixed).
In each case, $\RiemInv_{\g}$
has the properties of the universal symbolic formula for
$\psi$ discussed in Hypotheses~\ref{NaturalHyp}.  We include the
subscript $\g$ in the notation to indicate the metric used to define
$\Nv_a$, $R_{ab}{}^c{}_d$, $\mmc$, and $\nd$ in the symbolic formula
for $\RiemInv_{\g}$.  This metric is a representative of the conformal
class used in the definition of $\bg_{ab}$ and $\bg^{ab}$.  One can
easily verify that there are only finitely many ways of constructing a
monomial $I_{\g}$ of the above type; the key to doing so is the fact
that $I_{\g}$ must have mass $K$.  It follows that $S$ is a finite
set.  Thus $S=\{\RiemInv_{\alpha,\g}\}_{\alpha\in A_1}$, where $A_1$
is a finite index set.  The set $S$ clearly spans
$\RealCoeff_{K,\lowerw,w}$.

\begin{proposition}\label{RealBasis}
Let $\{\RiemInv_{\alpha,\g}\}_{\alpha\in A_1}$ be as above.  Then
there is a finite set $A\subseteq A_1$ such that $\RealBasis$ is a
basis for $\RealCoeff_{K,c,w}$.  The set $A$ and the symbolic formulae
for the operators $\RiemInv_{\alpha,\g}$ are independent of $w$.
\end{proposition}
\begin{proof}The claim that $A$ and the symbolic formulae are
  independent of $w$ follows from Proposition~\ref{IsomorphismProp}.
\end{proof}
\begin{corollary}\label{WZeroCombination}
Let $\RealBasis$ be as above, and let
$\psi\in\RationalCoeff_{K,\lowerw,0}$ be given.  Then
$\psi=\sum_{\beta\in A}f_{\beta}(n)\RiemInv_{\beta,\g}$, where for
each $\beta\in A$, $f_{\beta}(n)$ is some real rational function of
$n$.  We may assume that the functions $f_{\beta}(n)$ have the
following property: If the coefficients of $\psi$ are regular in some
dimension $n_0\in\Dimensions$, then each function $f_{\beta}(n)$ is
regular for $n=n_0$.
\end{corollary}
\begin{proof}
Let $S=\{\RiemInv_{\alpha,\g}\}_{\alpha\in A_1}$ be as above, with
$w=0$.  By using the ideas that we discussed in our construction of
this set, above, one can show that $\psi=\sum_{\gamma\in
  A_1}h_{\gamma}(n)\RiemInv_{\gamma,\g}$.  Here each coefficient
$h_{\gamma}(n)$ is a real rational function of $n$.  By
Proposition~\ref{RealBasis}, as applied to $\RealCoeff_{K,\lowerw,0}$,
there exist constants $C_{\gamma,\beta}\in\Reals$ such that for all
$\gamma\in A_1$, $\RiemInv_{\gamma,\g}=\sum_{\beta\in
  A}C_{\gamma,\beta}\RiemInv_{\beta,\g}$.  Thus $\psi=\sum_{\beta\in
  A}(\sum_{\gamma\in
  A_1}C_{\gamma,\beta}h_{\gamma}(n))\RiemInv_{\beta,\g}$.
\end{proof}
\begin{proposition}\label{RationalBasis}
Let $K\in\Integers_{>0}$ and $c\in\Reals$ be given, and let notation
be as above.  Then $\ClassesBasis$ is a basis for $\Classes_{K,c,0}$.
\end{proposition}
\begin{proof}
Let $\psi\in\RationalCoeff_{K,c,0}$ be given.  Then
$\psi=\sum_{\beta\in A}f_{\beta}(n)\RiemInv_{\beta,\g}$, by
Corollary~\ref{WZeroCombination}.  Here notation is as in the
corollary.  Thus $[\psi]=\sum_{\beta\in
  A}f_{\beta}(n)[\RiemInv_{\beta,\g}]$, and we see that
$\ClassesBasis$ spans $\Classes_{K,c,0}$.

We will use proof by contradiction to show that $\ClassesBasis$ is
linearly independent.  Suppose there is a set
$\{f_{\beta}(n)\}_{\beta\in A}\subseteq\FieldOfFuns$ of real rational
functions of $n$ which are not all zero such that
\begin{equation}\label{ZeroCombination}
\sum_{\beta\in
  A}f_{\beta}(n)[\RiemInv_{\beta,\g}]=0~.
\end{equation}
We may assume that all of the functions $f_{\beta}(n)$ are in fact
polynomials in $n$.  Let $\ell$ denote the largest degree of any
of these polynomials.  Let $A_2$ denote the set of all $\beta\in A$
such that $f_{\beta}(n)$ is nonzero, and for all $\beta\in A_2$, let
$d_{\beta}$ and $b_{\beta}$ denote the degree and lead coefficient,
respectively, of $f_{\beta}(n)$.  Let $A_3=\{\beta\in
A_2\,|\,d_{\beta}=\ell\}$.

By \nn{ZeroCombination}, there is an $n_0\in\Dimensions$ such that
$\sum_{\beta\in A}(1/n^{\ell})f_{\beta}(n)\RiemInv_{\beta,\g}$ is the
zero operator in all dimensions $n\in\Dimensions\cap\Integers_{\geq
  n_0}$.  Let a dimension $m\in\Dimensions$ be given.  Also let a
Riemannian manifold $\Mg$ of dimension $m$, a hypersurface $\sm$ in
$M$, a point $p\in\sm$, and $V\in C^{\infty}(M)$ be given.  Next, let
$n_1=\max\{n_0,m\}$.  Finally, let
$m'\in\Dimensions\cap\Integers_{>n_1}$ be given.  Then by
Proposition~\ref{Embedding}, there is a Riemannian manifold
$(M',\amr)$ of dimension $m'$, a hypersurface $\sm'$ of $M'$, a
function $V'\in C^{\infty}(M')$, and a point $p'\in\sm'$ such that
\[
(\xig)^{\lowerw}\left(\sum_{\beta\in
  A_2}\frac{f_{\beta}(m')}{(m')^{\ell}}
\RiemInv_{\beta,\g}V\right)(p)
  =
(\xir)^{\lowerw}\left(\sum_{\beta\in
  A_2}\frac{f_{\beta}(m')}{(m')^{\ell}}
\RiemInv_{\beta,\amr}V'\right)(p')=0~.
\]
Thus
\[
0=\lim_{n\rightarrow\infty}(\xig)^{\lowerw}\left(\sum_{\beta\in
  A_2}\frac{f_{\beta}(n)}{n^{\ell}}
\RiemInv_{\beta,\g}V\right)(p)=
(\xig)^{\lowerw}
\left(\sum_{\beta\in A_3}b_{\beta}\RiemInv_{\beta,\g}V\right)(p)~.
\]
Since $m\in\Dimensions$ was arbitrary, it follows that
$\sum_{\beta\in A_3}b_{\beta}\RiemInv_{\beta,\g}$ is the zero operator
for all $m\in\Dimensions$.  This contradicts the linear independence
of $\RealBasis$ in $\RealCoeff_{K,c,0}$.
\end{proof}

When we work with $\RealCoeff_{K,\lowerw,w}$,
$\RationalCoeff_{K,\lowerw,0}$, and $\Classes_{K,\lowerw,0}$,
operators will act on densities of a fixed weight.  But many of the
operators in this paper will act on densities of any weight.  For some
$c\in\Reals$ and every $w\in\Reals$, these operators will map $\ce[w]$
to $\ice[w-c]$.  To deal with such operators, we will need the
following two propositions.
\begin{proposition}\label{PsiToI}
Let $K\in\Integers_{>0}$, $\lowerw\in\Reals$, and a family of natural
hypersurface operators $\psi:\ce[w]\rightarrow\ice[w-\lowerw]$ of mass
$K$ be given, and suppose that $\psi$ satisfies
Hypotheses~\ref{NaturalHyp}.  Suppose also that the universal symbolic
formula for $\psi$ does not refer to $w$ or to $n$.  Finally, let
$\RealBasis$ be the basis for $\RealCoeff_{K,\lowerw,0}$ described in
Proposition~\ref{RealBasis}.  Then there exist real constants
$C_{\beta}$ such that for all $w\in\Reals$,
$\psi:\ce[w]\rightarrow\ice[w-\lowerw]$ is given by the symbolic
formula $\psi=\sum_{\beta\in A}C_{\beta}\RiemInv_{\beta,\g}$.
\end{proposition}

\begin{proof}
  Let $w\in\Reals$ be given, and let
  $\Phi:\RealCoeff_{K,c,w}\rightarrow\RealCoeff_{K,c,0}$ be as in
  Proposition~\ref{IsomorphismProp}.  The universal symbolic formula
  for $\psi$ determines an element of $\RealCoeff_{K,c,w}$, and there
  are real constants $C_{\beta}$ such that $\Phi(\psi)=\sum_{\beta\in
    A}C_{\beta}\RiemInv_{\beta,\g}$.  Since $\Phi$ is an isomorphism,
  $\psi:\ce[w]\rightarrow\ice[w-c]$ satisfies $\psi=\sum_{\beta\in
    A}C_{\beta}\RiemInv_{\beta,\g}$.  But the universal formula for
  $\psi$ does not refer to $w$, so the constants $C_{\beta}$ are
  independent of $w$.  It follows that for all $w\in\Reals$,
  $\psi:\ce[w]\rightarrow\ice[w-c]$ satisfies $\psi=\sum_{\beta\in
    A}C_{\beta}\RiemInv_{\beta,\g}$.
\end{proof}
\begin{proposition}\label{ACombination}
Let $K\in\Integers_{>0}$ and a family of natural operators $\psi$ be
given.  Suppose that this family satisfies Hypotheses~\ref{NaturalHyp}
and that the operators $\psi$ have mass $K.$ Let $\lowerw$ and
$\Dimensions$ be as in Hypotheses~\ref{NaturalHyp}, and let
$\RealBasis$ be the basis for $\RealCoeff_{K,c,0}$ described in
Proposition~\ref{RealBasis}.  Then there exist real functions
$f_{\beta}(w,n)$ such that for all $w\in\Reals$ and all
$n\in\Dimensions$, $\psi:\ce[w]\rightarrow\ice[w-\lowerw]$ is given by
the symbolic formula $\psi=\sum_{\beta\in
  A}f_{\beta}(w,n)\RiemInv_{\beta,\g}$.  For each $\beta\in A$,
$f_{\beta}(w,n)$ is polynomial in $w$ and rational in $n$.  If the
parameter $w$ does not appear in the universal formula for $\psi$,
then each function $f_{\beta}(w,n)$ is independent of $w$.
\end{proposition}
\begin{proof}
The result follows from Proposition~\ref{PsiToI} and an argument
similar to the one we used in the proof of
Corollary~\ref{WZeroCombination}, above.  
\end{proof}
We will use the following proposition in Section~\ref{Qsec}, below.

\begin{proposition}\label{ZeroAlln}
Let $K\in\Integers_{>0}$ and $\lowerw\in\Reals$ be given.  Let
$\Dimensions$ be as above, and let $\psi\in\RationalCoeff_{K,c,0}$ be
given.  Suppose that the coefficient of every term in the universal
symbolic formula for $\psi$ is regular for every $n\in\Dimensions$.
Suppose also that for some $n_1\in\Integers_{>0}$, $\psi$ is the zero
operator in all dimensions $n\in\Dimensions\cap\Integers_{>n_1}$.
Then $\psi$ is the zero operator in all dimensions $n\in\Dimensions$.
\end{proposition}
\begin{proof}
By Corollary~\ref{WZeroCombination}, we may may write
$\psi=\sum_{\beta\in A}f_{\beta}(n)\RiemInv_{\beta,\g}$, where each
coefficient $f_{\beta}(n)$ is a real rational function of $n$.  Here
$\RiemInv_{\beta,\g}$ and $A$ are as above.  We may assume that
$f_{\beta}(n)$ is regular at every $n\in\Dimensions$.

Let $m\in\Dimensions$ be given.  Also let a Riemannian manifold $\Mg$
of dimension $m$, a hypersurface $\sm$ in $M$, a point $p\in\sm$, and
$V\in C^{\infty}(M)$ be given.  Let $n_2=\max\{m,n_1\}$, and let
$m'\in\Dimensions\cap\Integers_{>n_2}$ be given.  Then by
Proposition~\ref{Embedding}, there is a Riemannian manifold
$(M',\amr)$ of dimension $m'$, a hypersurface $\sm'$ of $M'$, a
function $V'\in C^{\infty}(M')$, and a point $p'\in\sm'$ such that
\[
(\xig)^{\lowerw}
\left(\sum_{\beta\in A}f_{\beta}(m')\RiemInv_{\beta,\g}V\right)(p)=
(\xir)^{\lowerw}
\left(\sum_{\beta\in A}f_{\beta}(m')\RiemInv_{\beta,\amr}V'\right)(p')=
0~.
\]
Since $m'\in\Dimensions\cap\Integers_{>n_2}$ was arbitrary,
$(\xig)^{\lowerw}(\sum_{\beta\in
  A}f_{\beta}(m')\RiemInv_{\beta,\g}V)(p)$ is a real rational function
of $m'$ which is zero for infinitely many distinct values of $m'$.  Thus
$(\sum_{\beta\in A}f_{\beta}(m)\RiemInv_{\beta,\g}V)(p)=0$.
Since $p$ and $V$ were arbitrary, it follows that $\psi$ is the zero
operator in dimension $n=m$.
\end{proof}
%
%

%
%
%

\section{Invariant operator constructions\\along a
  hypersurface} \label{confi}

The basic example of a conformal boundary operator is the conformal
Robin operator, as given by the formula \nn{cr}. In this section, we
develop the $\delta_K$, $\ConFlatOp{K}$ and $\deltaJk$ operator
families that we discussed in Section~\ref{IntroSect}, above.  Each
gives, for most weights, higher-order analogues of $\delta$. The first
key observation is that the conformal Robin operator is {\em strongly
  conformally invariant}, which means that by coupling formula \nn{cr}
to the tractor connection, we obtain a conformally invariant operator
along $\Sigma$ which acts on any weighted tractor bundle.  We thus
obtain
 \begin{equation}\label{Robin}
 \left.\delta_1=\delta:\ct^\Phi[w]\to \ct^\Phi[w-1]\right|_{\sm}
 \end{equation} 
 along $\sm$ for any weight $w$. We use this to build the higher analogues of $\delta$.
These new constructions will lead to solutions to
problems~\ref{GenProblem} and \ref{CQprob}.  The motivation for many
steps of our procedure will center around the paramount Problem~\ref{CQprob}.
As we
noted in Remark~\ref{IdentificationRem}, above, we may identify
conformally covariant and conformally invariant operators.  This will
be an important point in what follows.

%
%

\subsection{Preliminary work}\label{main}%

As above, we work locally along a hypersurface $\sm$ in a Riemannian
conformal manifold $\Mcc$ with $\dim(M)=n\geq 3$.  For convenience,
the normal field, second fundamental form, and so on are extended off
$\sm$ via a foliation, as discussed in sections~\ref{RH} and
\ref{confhy}, above.  All quantities are thus defined on an open
neighbourhood of $\Sigma$.  We use that $\ce$ is the same as the space of smooth
sections of $\ce[0]$.  We begin by introducing a generalisation of
the set $E$ of Problem~\ref{GenProblem}.
\begin{definition}\label{ExSet}
Let $m\in\Integers_{>0}$, $\lowerw\in\Reals$, and a family of natural
hypersurface differential operators
$\psi:\tbPhi[w]\rightarrow\tbPhi[w-\lowerw]|_{\sm}$ of order at most
$m$ be given.  Let $E(\psi)$ denote the set of all weights $w$ such
that $\psi$ fails to have transverse order $m$.
\end{definition}
The set $E(\psi)$ of Definition~\ref{ExSet} may depend on $n$.  When
we discuss $E(\psi)$ in this paper, the value of $m$ will always be
clear from the context.  If $m=0$, then $E(\psi)=\emptyset$.
We will always find that $E(\psi)$ is independent of the bundle
$\tbPhi$.

The first main step in our operator constructions will be to define
the $\delta_K$ operator family.  We need some additional definitions
and terminology that we introduce now.
\begin{definition}\label{lotsdef}
Let $m\in\Integers_{\geq 3}$, a set
$\Dimensions\subseteq\Integers_{\geq 3}$, and a family of natural
differential operator $\psi$ be given, and suppose that $\psi$ is
well-defined in all dimensions $n\in\Dimensions$.  Also suppose that
$\psi$ has order at most $m$ in all dimensions $n\in\Dimensions$ and
for operands of all possible weights.  Then $\lots$ and $\ltots$ will
denote terms in a universal
symbolic formula
for $\psi$ which, in all dimensions $n\in\Dimensions$ and for operands
of all possible weights, have order less than $m$ and transverse
order less than $m$, respectively.
\end{definition}
To simplify our discussions it will be convenient to introduce a
notion that is very specific to the details of the formulae that we
use. First observe that a symbolic formula for a natural differential operator may
involve terms that move tractor indices in a way that we now describe.  Let
$T^{ab}\in\ce^{ab}$ be given, and define an operator
$\psi:\tbn^A[w]\rightarrow\tbn^A[w]|_{\sm}$ by letting
\begin{equation}\label{NoMove}
  \psi V^A=T^{ab}\nd_a\nd_bV^A
\end{equation}
for all $V^A\in\tbn^A[w]$.  Then
\begin{equation}\label{YesMove}
\psi V^A=T^{ab}\nd_b\nd_aV^A+T^{ab}\Omega_{ab}{}^A{}_BV^B~,
\end{equation}
by \nn{TracCurv}.  In one term of \nn{YesMove}, the tractor index $A$
has moved off of $V$ and onto $\Omega$.  We say that the operator
formula given by the right-hand side of \nn{YesMove} \textit{moves}
tractor indices.  The formula given by the right-hand side of
\nn{NoMove} \textit{fixes} tractor indices.  One can extend these
ideas to other operators and to tractor bundles of higher rank in the
obvious way.  We emphasise that the property of fixing or moving
tractor indices is a property of an operator formula and not a
property of an operator.

At many points in our operator constructions, we will work with a
symbolic operator formula which satisfies the following hypotheses:
\begin{hypotheses}\label{FixingHyp}
The symbolic operator formula fixes tractor indices and is a
polynomial in $\Nv^a$, $R_{ab}{}^{c}{}_d$, $\bg_{ab}$, $\bg^{ab}$,
$\mc$, and the coupled Levi-Civita tractor connection $\nd$.  The
coefficients of this polynomial are real functions of $w$ and $n$
which are polynomial in $w$ and rational in $n$.  Within the symbolic
operator formula, $\nd$ never explicitly hits $\Nv^a$ or $\mc$.  For
some $c\in\Reals$, the symbolic operator formula defines a family of
natural differential operators
$\psi:\tbPhi[w]\rightarrow\left.\tbPhi[w-c]\right|_{\sm}$.
\end{hypotheses}

All of our operator constructions will begin with the operator family
$\delta_K$ of the following lemma.
\begin{lemma}\label{basicversion}
For any $j\in\Integers_{\geq 0}$, let $K=j+1$.  Then the formula
\begin{equation}\label{deltajplusone}
\delta_{j+1}:= N^{A_1}N^{A_2}\cdots
N^{A_{j}}\delta D_{A_1}D_{A_2}\cdots D_{A_j}
\end{equation}%
defines a family of natural hypersurface operators
$\left.\delta_K:\tbPhi[w] \to \tbPhi [w-K]\right|_{\sm}$ of mass $K$.
These operators are conformally invariant and may be given by a
universal symbolic formula which satisfies Hypotheses~\ref{FixingHyp}
and is polynomial in $n$.
\end{lemma}
\begin{proof}
The claims of the lemma are clear by the formulae of \nn{cr},
\nn{connids}, \nn{tractorD}, \nn{NormalTractor}, \nn{Robin}, and
Figure~\ref{TrIProd}.
\end{proof}

We will show that $\delta_K$ has order at most $K$, and we will find
$E(\delta_{K})$.  To do this, we will need several results.  For any
$k\in\Integers_{>0}$, let $\nd_{\Nv}^k$ denote
$\Nv^{a_1}\cdots\Nv^{a_k}\nd_{a_1}\cdots\nd_{a_k}$, where $\nd$ is the
coupled Levi-Civita tractor connection.  This operator has transverse
order $k$.
\begin{lemma}\label{NNabnkD}
Let $s\in\Reals$ and $k\in\Integers_{>0}$ be given.  Then for any
$V\in\tbPhi[s]$,
\begin{equation}\label{NNabnkDDisplay}
\Nt^A\nd_n^kD_A V =
(n+2s-2-k)\nd_{\Nv}^{k+1}V+k(\nd_{\Nv}^{k+1}-\nd_{\Nv}^{k-1}\Delta)V
+
\lots~.
\end{equation}
Here the \lots\ can be given by a
symbolic formula which satisfies Hypotheses~\ref{FixingHyp}.
\end{lemma}
\begin{proof} This is a trivial consequence of the formula \nn{tractorD} for the
tractor $D$-operator. Verification in detail uses  \nn{connids},
\nn{NormalTractor}, and the rules for the tractor metric given
in Figure~\ref{TrIProd}.
\end{proof}

We want to base an induction around this result. To do this, we must
deal with terms similar to those in the second expression on the
\RHS\ of \nn{NNabnkDDisplay}.  The following is result is easily
verified.
\begin{lemma}\label{NablaLaplacian}
Let nonnegative integers $p$, $q$, $r$, $s$, and $k$ be given, and
suppose $p+2q=r+2s=k>0$.  Then $\nd_n^p\Delta^q-\nd_n^r\Delta^s$, as
an operator acting on densities or weighted tractors, is an operator
of transverse order less than $k$.
\end{lemma}
Our induction will also require the following lemma.
\begin{lemma}\label{LTOTS}
Let $K\in\Integers_{>0}$, $c\in\Reals$, and a family of natural
differential operators
$Op\,:\tbPhi[w]\to\left.\tbPhi[w-c]\right|_{\sm}$ be given, and
suppose that $Op$ can be given by a symbolic formula which
satisfies Hypotheses~\ref{FixingHyp}.  Suppose also that, for all
$w\in\Reals$, $Op$ has order at most $K$.  Then $\Nt^A\,Op\,D_A$
%
%
can also be given by a symbolic
formula which satisfies Hypotheses~\ref{FixingHyp}, and when acting on
tractors of any weight $w\in\Reals$, this operator has order at most
$K+1$.  For any given $w\in\Reals$, if
$Op\,:\tbPhi[w]\to\left.\tbPhi[w-c]\right|_{\sm}$ has transverse order
at most $K-1$, then for this value of $w$,
\[
\Nt^A\,Op\,D_A:\tbPhi[w+1]\to\left.\tbPhi[w-c]\right|_{\sm}
\]
has transverse order at most $K$.
\end{lemma}
\begin{proof}
The result follows from \nn{connids}, \nn{tractorD},
\nn{NormalTractor}, and Figure~\ref{TrIProd}.  The key points are as
follows.  First, the leading part of $D_A$ is $-X_A\Delta$.  We may
assume that $Op$ is given by a symbolic formula which satisfies
Hypotheses~\ref{FixingHyp}, so in the symbolic formula for
$-\Nt^A\,Op\,X_A$, we may assume that the tractor indices $A$ are in
the locations shown.  But $\Nt^AX_A=0$, so in $-\Nt^A\,Op\,X_A$, we
may assume that at least one $\nd$ hits $X_A$.  It follows that
$-\Nt^A\,Op\,X_A$ is an operator of order at most $K+1$.

Now let $w\in\Reals$ be given, and suppose that
$Op:\tbPhi[w]\to\left.\tbPhi[w-c]\right|_{\sm}$ has transverse order
at most $K-1$.  To complete the proof of the lemma, we will show that
$\Nt^A\,Op\,D_A:\tbPhi[w+1]\to\left.\tbPhi[w-c]\right|_{\sm}$ has
transverse order at most $K$.  To do this, we need only consider the
leading part $-X_A\Delta$ of $D_A$.  As above, we assume that the
symbolic formula for $Op$ satisfies Hypotheses~\ref{FixingHyp}, so we
may write
\[
Op=T^{i_1i_2\ldots i_K}\nd_{i_1}\nd_{i_2}\ldots\nd_{i_K}+\lots~,
\]
for some component functions $T^{i_1i_2\ldots i_K}$.  Let a defining
function $\dfn$ for $\sm$ and a section $V$ of $\tbPhi[w+1]$ be given,
and consider
\begin{equation}\label{L20Feb11a}
\Nt^A T^{i_1\ldots i_K}\nd_{i_1}\ldots\nd_{i_K}(-X_A\Delta\dfn^{K+1}V)~.
\end{equation}
Along $\sm$, this is a linear combination of terms of the form
\begin{equation}\label{L27May12a}
\Nt^A T^{i_1\ldots
  i_K}(-Z_{A{i_a}})(\nd_{i_1}\dfn)\cdots(\nd_{i_{\hat{a}}}\dfn)\cdots
(\nd_{i_K}\dfn)(\nd_j\dfn)(\nd^j\dfn)V~.
\end{equation}
Here $\hat{a}$ indicates omission of $(\nd_{i_a}\dfn)$.  By
\nn{pndef}, $\nd_i\dfn$ is $(\xig)^{-1}n_i$, up to multiplication by a
nowhere zero function, so \nn{L27May12a} is equal to
\begin{equation}\label{L27May12b}
-(\xig)^{-K}\Nv_{i_1}\cdots\Nv_{i_K}T^{i_1\ldots i_K}(\xig)^{-1}V~,
\end{equation}
up to a nonzero scale.  But \nn{L27May12b}, and hence also
\nn{L20Feb11a}, vanish along $\sm$, since
$Op:\tbPhi[w]\to\left.\tbPhi[w-c]\right|_{\sm}$ has transverse order
at most $K-1$.
\end{proof}
\begin{proposition}\label{norder}
Let $K\in\Integers_{>0}$ be given.  Then for every real number $w$,
$\delta_K:\tbPhi[w]\rightarrow\tbPhi[w-K]|_{\sm}$ has order at most
$K$, and along $\sm$,
\begin{equation}\label{norderDisplay}
\delta_{K}
=
\big[ \prod_{i=1}^{K-1}(n+2w-K-i)\big] \nabla^{K}_{\Nv}
+\ltots~.
\end{equation}
Here the $\ltots$ can by given by a symbolic formula which satisfies
Hypotheses~\ref{FixingHyp}.  If $K=1$, we take the explicit product in
\nn{norderDisplay} to be $1$.
\end{proposition}
\begin{proof}
This follows by induction and lemmas~\ref{NNabnkD},
\ref{NablaLaplacian}, and \ref{LTOTS}.
\end{proof}

Let $K\in\Integers_{>0}$ be given.  By Proposition~\ref{norder}, it
follows that $E(\delta_K)$ is the set of all $w$ that solve
$\prod_{i=1}^{K-1}(n+2w-K-i)=0$.  Thus $E(\delta_1)=\emptyset$, and
for any $K\in\Integers_{\geq 2}$,
\begin{equation}\label{E0}
E(\delta_K) =
\left\{\frac{2K-1-n}{2},\frac{2K-2-n}{2}, \cdots , \frac{K+1-n}{2}\right\}~.
\end{equation}
For all $w\in\Reals\backslash E(\delta_K)$, the operator $\delta_K$
has transverse order $K$, and for all $w\in E(\delta_K)$, it has
transverse order less than $K$.  Note that if $n=K+1$, then $0$ is an
exceptional weight.  Our solutions to Problem~\ref{CQprob} will
therefore require some additional work, as we noted above.  One of the
keys to our work will be Proposition~\ref{GJMSTractor}, below.  In
this proposition, $P_{2k}:\ce[k-n/2]\rightarrow\ce[-k-n/2]$ is the
GJMS operator of \cite{GJMS} of order $2k$.  To ensure that this
operator exists, we assume that $n$, $k$, and $\Mcc$ satisfy the
following condition:
\begin{condition}\label{nkcCondition}
(1)~$n$ is odd, or (2)~$n$ is even and $k\leq n/2$, or (3)~$\Mcc$ is
  conformally flat.
\end{condition}
\noindent We will refer to Condition~\ref{nkcCondition} at several
points in our work.
%
%
\begin{proposition}\label{GJMSTractor}
Let $k\in\Integers_{>0}$ be given.  There is a family of natural
conformally invariant differential operators
$$\LowTrac{1}{k-1}:\ce[k-n/2]\rightarrow\tbn_{A_1\cdots
  A_{k-1}}[-1-n/2]$$
such that for any
$V\in\ce[k-n/2]$,
\begin{equation}\label{GJMStractorformula}
  (-1)^{k-1}X_{A_1}\cdots X_{A_{k-1}}P_{2k}V=
  \Box D_{A_1}\cdots D_{A_{k-1}}V+
  \LowTrac{1}{k-1} V~.
\end{equation}
These
operators are well-defined for all $n$ and $\Mcc$ which satisfy
Condition~\ref{nkcCondition}.  For the case in which $k\in\{1,2\}$ or
$\Mcc$ is conformally flat, $\LowTrac{1}{k-1}=0$.  In other cases,
$\LowTrac{1}{k-1}=\LowTracPsi{1}{k-1}{PQ}D_PD_Q$, where
$\LowTracPsi{1}{k-1}{PQ}$ is given by a universal symbolic tractor
formula.  This formula is a polynomial in $X$, $D$, $W$, $\tm$, and
$\tmco$ whose coefficients are real rational functions of $n$.
Similarly, $\LowTrac{1}{k-1}$ is also given by such a tractor formula;
every term of this formula is of degree at least 1 in $W$ and has mass
$k+1$ and weight $-k-1$.  In every term of the tractor formula for
$\LowTrac{1}{k-1}$, at least
one of the indices $A_1$, \ldots, $A_{k-1}$ appears on a $W$.
\end{proposition}
\begin{proof}
The proposition follows from a careful study of the statement and
proof of Proposition~4.5 of \cite{GP-CMP}.  Note that $W_{ABCE}$
vanishes if $\Mcc$ is conformally flat.
\end{proof}
\begin{remark}\label{GJMSTracRem}
The proof of Proposition~\ref{GJMSTractor} uses the tractor formula
for $P_{2k}$ described in \cite{GP-CMP}.  The construction of this
tractor formula in general involves a finite number of choices;  we 
assume that we have made and fixed these choices.  As a result, the
operators $\LowTrac{1}{k-1}$ and $\LowTracPsi{1}{k-1}{PQ}$ of
Proposition~\ref{GJMSTractor}, as well as all operators and curvatures
derived from them, will always be uniquely determined.
\end{remark}
%

%
%
\subsection{Refining the boundary family in the conformally flat
  case}\label{refinef}

In the conformally flat case, it turns out that for every second
weight in \nn{E0}, $\delta_K$ is the zero operator.
Dividing $\delta_K$ by the corresponding factors yields an improved
family of operators. We state this precisely as follows.
\begin{theorem}\label{cflatkey}
Let $K\in\Integers_{>0}$ be given.  There is a family of natural
conformally invariant differential operators $\ConFlatOp{K}: \ce[w]\to
\ice[w-K]$ of order at most $K$ on conformally flat conformal
manifolds $\Mcc$ of dimension $n\geq 3$ determined by the equation
\[
\big[ \prod_{j=1}^{\lfloor \frac{K-1}{2}\rfloor}(n+2w-2K+2j)\big]
\ConFlatOp{K}=\delta_K
\]
and polynomial continuation in $w$.  Here $\lfloor(K-1)/2\rfloor$
denotes the integer part of $(K-1)/2$.
The $\ConFlatOp{K}$ family
satisfies Hypotheses~\ref{NaturalHyp} with
$\Dimensions=\Integers_{\geq 3}$.
\end{theorem}
\begin{proof}
We may assume that $K\geq 3$.  We will apply the results of
Section~\ref{OpBases} with $\Dimensions=\Integers_{\geq 3}$.  Let
$\RealBasis$ be the basis for $\RealCoeff_{K,K,0}$ described in
Proposition~\ref{RealBasis}.  By Proposition~\ref{ACombination}, we
may write $\delta_K=\sum_{\beta\in A}\fkbeta(w,n)\RiemInv_{\beta,\g}$.
Here each coefficient $f_{\beta}(w,n)$ is a polynomial in $w$ whose
coefficients are real rational functions of $n$.  By the polynomial
division algorithm from elementary algebra, we may thus write
\begin{equation}\label{LongDivTwo}
\fkbeta(w,n)=\big[ \prod_{j=1}^{\lfloor
    \frac{K-1}{2}\rfloor}(n+2w-2K+2j)\big]\qkbetan(w)+\rkbetan(w)~.
\end{equation}
Here $\qkbetan(w)$ and $\rkbetan(w)$ are polynomials in $w$ whose
coefficients are real rational functions of $n$.  In performing the
polynomial division, we are dividing by a polynomial with leading term
$2^bw^b$, where $b=\lfloor (K-1)/2\rfloor$.  It follows that
$\qkbetan(w)$ and $\rkbetan(w)$ are regular at all $n\in\Dimensions$,
and $\rkbetan(w)$ has degree less than $\lfloor(K-1)/2\rfloor$.

Now let $j\in\Integers$ and $V\in\ce[K-j-n/2]$ be given, and
suppose that $1\leq j\leq\lfloor(K-1)/2\rfloor$.  An easy computation
shows that $j<K-j$.  By \nn{deltajplusone},
\begin{equation}\label{deltaK}
\delta_KV= N^{A_j}N^{A_{j+1}}\cdots N^{A_{K-1}}\delta_j
D_{A_{j}}D_{A_{j+1}}\cdots D_{A_{K-1}}V~.
\end{equation}
By \nn{tractorD}, \nn{YamOp}, and Proposition~\ref{GJMSTractor},
\begin{equation}\label{L7July14a}
(-1)^{K-j}X_{A_{j}}X_{A_{j+1}}\cdots
X_{A_{K-1}} P_{2(K-j)} V
=
D_{A_{j}}D_{A_{j+1}}\cdots D_{A_{K-1}} V~.
\end{equation}
Next, note that $N^AX_A=0$, by \nn{NormalTractor} and
Figure~\ref{TrIProd}.  Also note that $\delta_j$ has order at most
$j$, by Proposition~\ref{norder}.  Thus $\delta_KV=0$, by 
\nn{deltaK} and \nn{L7July14a}.

Now let $V_0\in\ce[0]$ be given, and let $j$ be as above.  Then
\[
\delta_K(\xig)^{K-j-n/2}V_0=\sum_{\beta\in
  A}f_{\beta}(K-j-n/2,n)\RiemInv_{\beta,\g}(\xig)^{K-j-n/2}V_0=0
\]
along $\sm$.  But $\nd((\xig)^{K-J-n/2})=0$, so $\sum_{\beta\in
  A}f_{\beta}(K-j-n/2,n)\RiemInv_{\beta,\g}V_0=0$ along $\sm$.  For
all $\beta\in A$, $\RiemInv_{\beta,\g}$ determines an element
$[\RiemInv_{\beta,\g}]$ of $\Classes_{K,K,0}$.  Since $V_0$ was
arbitrary, our work shows that $\sum_{\beta\in
  A}f_{\beta}(K-j-n/2,n)[\RiemInv_{\beta,\g}]$ is the zero element of
$\Classes_{K,K,0}$.  Thus by Proposition~\ref{RationalBasis}, it
follows that $f_{\beta}(K-j-n/2,n)=0$ for all $\beta\in A$ and all
$n\in\Integers$.  But then $r_{\beta,n}(K-j-n/2)=0$ for all $\beta\in
A$ and all $n\in\Dimensions$, by \nn{LongDivTwo}.  Let $\beta\in A$ be
given.  Then for all $n\in\Dimensions$, $\rkbetan(w)$ is a polynomial
in $w$ having $\lfloor(K-1)/2\rfloor$ distinct zeros.  The
coefficients of this polynomial are thus zero for all
$n\in\Dimensions$ and hence for all $n\in\Integers$.  Thus
$r_{\beta,n}(w)=0$ for all $w\in\Reals$ and all $n\in\Integers$.

We now let $\ConFlatOp{K}=\sum_{\beta\in
  A}\qkbetan(w)\RiemInv_{\beta,\g}$.  A polynomial continuation
argument shows that for all $w\in\Reals$, this operator is
well-defined and conformally invariant.
\end{proof}

\begin{corollary}\label{cflatEwts}
Let $E(\ConFlatOp{K})$ be as in Definition~\ref{ExSet}.  Then
$E(\ConFlatOp{1})=\emptyset$, and for all $K\in\Integers_{\geq 2}$,
\[
\textstyle
E(\ConFlatOp{K}) =
\left\{
\frac{2K-1-n}{2}, \frac{2K-1-n}{2}-1, \cdots ,\frac{2K-1-n}{2}-
\left\lfloor \frac{K-2}{2}\right\rfloor
\right\}~.
\]
For all $w\in\Reals\backslash E(\ConFlatOp{K})$, the operator
$\ConFlatOp{K}$ has transverse order $K$, and for all $w\in
E(\ConFlatOp{K})$, it has transverse order less than $K$.
\end{corollary}
%
%
\begin{proof}
By Theorem~\ref{cflatkey} and Proposition~\ref{norder}, it follows
that
\begin{equation}\label{deltaKNaughtExp}
\Big[ \prod_{j=1}^{\lfloor \frac{K-1}{2}\rfloor}(n+2w-2K+2j)\Big]
\ConFlatOp{K}=
\Big[ \prod_{i=1}^{K-1}(n+2w-K-i)\Big]\nd_n^K+\ltots~.
\end{equation}
By Proposition~\ref{ACombination}, $\ltots=\sum_{\beta\in
  A}h_{\beta}(w,n)\RiemInv_{\beta,\g}$.  Here $A$ and
$\RiemInv_{\beta,\g}$ are as above.  Each coefficient $h_{\beta}(w,n)$
is a polynomial in $w$ whose coefficients are real rational functions
of $n$.  By reasoning as in the proof of Theorem~\ref{cflatkey}, we
find that
%
%
\[
\ltots=
\Big[ \prod_{j=1}^{\lfloor
    \frac{K-1}{2}\rfloor}(n+2w-2K+2j)\Big]\sum_{\beta\in
  A}s_{\beta,n}(w)\RiemInv_{\beta,\g}~,
\]
%
%
where $s_{\beta,n}(w)$ is polynomial in $w$ and rational in $n$.
Polynomial continuation in $w$ shows that for all $w\in\Reals$, the
operator $\sum_{\beta\in A}s_{\beta,n}(w)\RiemInv_{\beta,\g}$ has
transverse order less than $K$.  To obtain $\ConFlatOp{K}$, we divide
\nn{deltaKNaughtExp} by the coefficient of $\ConFlatOp{K}$ on the
left-hand side of \nn{deltaKNaughtExp}.  The corollary then follows.
\end{proof}
We will use Corollary~\ref{cflatEwts} to identify a critical operator
in the proof of Theorem~\ref{SolnTheorem}, below.

%
%
\subsection{The refinement in the general case}
\label{refine}
We now consider arbitrary metrics
and develop an additional refinement of
$\delta_K:\ce[w]\rightarrow\ice[w-K]$.  The key to this new
construction will be the operator family $P_{A_1\cdots A_k}$ of
Proposition~\ref{GJMSp}, below.  In the definition of $\delta_K$ in
\nn{deltajplusone}, we will replace some of the $D$-operators with
$P_{A_1\cdots A_k}$.  Then, depending on parameters, we will divide
the resulting operator by a polynomial in $w$ in much the same way as
we did in our construction of $\ConFlatOp{K}$, above.  Under certain
conditions, this polynomial will be $(n+2w-1-K)$.  Suppose that we
\textit{do} in fact divide by $(n+2w-1-K)$.  Then for cases in which
$n$ is even, $K=n-1$, and $w=0$, the refined operator will be a
critical operator.

The key to the construction of $P_{A_1\cdots A_k}$ is the operator family
$\LowTrac{1}{k-1}$ of Proposition~\ref{GJMSTractor}.  In that
proposition, $\LowTrac{1}{k-1}$ acts on a section of $\ce[k-n/2]$.
Any symbolic tractor formula for $\LowTrac{1}{k-1}$, however,
determines a family of conformally invariant differential operators on
$\ce[w]$ for general weights $w\in\Reals$.  In the next proposition,
we will use this fact to construct $P_{A_1\cdots A_k}$.
\begin{proposition}\label{GJMSp}
  Let $k\in\Integers_{>0}$ be given.  By using the symbolic formula
  for $\LowTrac{2}{k}$, define a family of operators
$$
P_{A_1\cdots A_k}:\ce[w]\to \tbn_{A_1\cdots A_k}[w-k] 
$$
by letting
\begin{equation}\label{DefineGJMSp}
P_{A_1\cdots A_k}V=D_{A_1}\ldots D_{A_k}V-X_{A_1}\LowTrac{2}{k}V
\end{equation}
for all $V\in\ce[w]$.  Then for any $V\in\ce[k-n/2]$,
\begin{equation}\label{XXXP}
P_{A_1\cdots  A_k}V=(-1)^k X_{A_1}\cdots X_{A_k}P_{2k}V~.
\end{equation}
The operators $P_{A_1\cdots A_k}$ are conformally invariant natural
differential operators given by a universal symbolic formula which is
a polynomial in $\bg$, $Y$, $Z$, $X$, $\tm$, $\tmco$, $\nd$, and the
Riemannian curvature of $\g$, where $\g$ is a representative of the
conformal structure on $M$.  The coefficients of this polynomial are
polynomial in $w$ and rational in $n$, and each term of the polynomial
has mass $k$.  The operators $P_{A_1\cdots A_{k}}$ are well-defined
for all $n$ and $\Mcc$ which satisfy Condition~\ref{nkcCondition}.
\end{proposition}
%
%
\begin{proof}
Let $V\in\ce[k-n/2]$ be given.  By \nn{tractorD} and
\nn{GJMStractorformula}, it follows that
\[
(-1)^k X_{A_1}\cdots X_{A_k}P_{2k}V =
D_{A_1}\cdots D_{A_k}V-X_{A_1}\LowTrac{2}{k}V~.
\]
Thus \nn{XXXP} holds.  The other claims of the proposition follow from
\nn{TracCurveOne}, \nn{TracCurveTwo}, \nn{tractorD}, and
Proposition~\ref{GJMSTractor}.
\end{proof}

Our new refinement of $\delta_K$ will require additional groundwork.
For any $E\subseteq\Reals$ and $x\in \Reals$, let $E+x=\{y+x \mid y\in
E\}$.
%
%
%
\begin{lemma}\label{DoubleO}
Let $J\in\Integers_{>0}$, a real number $\lowerw$, and a family of
natural differential operators
$\BasicOp:\tbPhi[w]\rightarrow\left.\tbPhi[w-\lowerw]\right|_{\sm}$ be
given.
Suppose that $\BasicOp$ has mass $J$ and order at most $J$ and can be
given by a universal symbolic formula which satisfies
Hypotheses~\ref{FixingHyp}.  Suppose also that $\BasicOp$ has
transverse order less than $J$ for all $w\in
E(\!\hspace{0.17ex}\BasicOp\hspace{0.2ex})$.  Then there exist
families of natural differential operators $\OpOne$ and $\OpTwo$
having the following properties:
\begin{enumerate}
\item For all $w\in\Reals$, $\OpOne$ and $\OpTwo$ map $\tbPhi[w]$ to
  $\left.\tbPhi[w-\lowerw-1]\right|_{\sm}$.
\item $\Nt^A\BasicOp D_AV=(n+2w-2-J)\OpOne V+\OpTwo V$ for all
  sections $V\in\tbPhi[w]$.
\item $\OpOne$ and $\OpTwo$ have mass $J+1$ and order at most $J+1$
  and can be given by universal symbolic formulae which satisfy
  Hypotheses~\ref{FixingHyp}.
\item $\OpOne$ has transverse order less than $J+1$ for all $w\in
  E(\OpOne)$, and
  $E(\OpOne)=E(\!\hspace{0.17ex}\BasicOp\hspace{0.2ex})+1$.
\item For all $w\in\Reals$, $\OpTwo$ has transverse order less than
  $J+1$.
\end{enumerate}
\end{lemma}
\begin{proof}
Let $w\in\Reals$ be given.
We will assume that $\BasicOp$ is given by a symbolic formula which
satisfies Hypotheses~\ref{FixingHyp}.  Thus
$\BasicOp=T^{i_1\ldots i_J}\nd_{i_1}\ldots\nd_{i_J}+\lots$,
where $T^{i_1\ldots i_J}$ are some component functions.  By
Lemma~\ref{LTOTS}, $\Nt^A\lots\,D_A$ is an operator of order less than
$J+1$, which we include in $\OpTwo$.  Let $V\in\tbPhi[w]$ and a
defining function $\dfn$ for $\sm$ be given, and consider the
behaviour of $\Nt^AT^{i_1\ldots
  i_J}\nd_{i_1}\ldots\nd{_{i_J}}D_A\dfn^{J+1}V$ along $\sm$.  We
consider the various parts of $D_A$, beginning with $-X_A\Delta$.
Along $\sm$, we have
\begin{equation}\label{L26June11a}
\begin{array}{ll}
\lefteqn{\Nt^AT^{i_1\ldots
    i_J}\nd_{i_1}\ldots\nd_{i_J}(-X_A)\Delta\dfn^{J+1}V=}\vspace{1mm}
\\
&-\sum_{a=1}^J(J+1)!\,\Nt^{A}T^{i_1\ldots
  i_J}Z_{Ai_a}\bg^{kl}
(\nd_{i_1}\dfn)\cdots(\nd_{i_{\hat{a}}}\dfn)\cdots
(\nd_{i_J}\dfn)(\nd_{k}\dfn)(\nd_{l}\dfn)V=\vspace{1mm}
\\
&-\sum_{a=1}^J(J+1)!\,\Nv_{i_a}T^{i_1\ldots i_J}\bg^{kl}
(\nd_{i_1}\dfn)\cdots(\nd_{i_{\hat{a}}}\dfn)\cdots
(\nd_{i_J}\dfn)(\nd_{k}\dfn)(\nd_{l}\dfn)V~.
\end{array}
\end{equation}
Here $\hat{a}$ denotes omission of the factor.  By \nn{ndef},
\[
\Nv_{i_a}\bg^{kl}(\nd_kt)\nd_lt=
\frac{\nd_{i_a}t}{|dt|_{\g}}\bg^{kl}(\nd_kt)\nd_lt=
    (\nd_{i_a}t)\bg^{kl}(\nd_kt)\Nv_l=
  (\nd_{i_a}t)\Nv^k\nd_kt~.
\]
Thus by \nn{L26June11a},
\[
\begin{array}{ll}
\lefteqn{\Nt^AT^{i_1\ldots
    i_J}\nd_{i_1}\ldots\nd_{i_J}(-X_A)\Delta\dfn^{J+1}V=}\vspace{0.8ex}
\\
&-J(J+1)!\,T^{i_1\ldots
  i_J}\Nv^k(\nd_{i_1}\dfn)\cdots(\nd_{i_J}\dfn)(\nd_k\dfn)V
\end{array}
\]
along $\sm$.

Now consider the part of $D_A$ corresponding to
$(n+2w-2)Z_A{}^{k}\nd_k$.  Along $\sm$, we have
\[
\begin{array}{ll}
\lefteqn{\Nt^{A}T^{i_1\ldots i_J}
\nd_{i_1}\ldots\nd_{i_J}(n+2w-2)Z_{A}{}^{k}\nd_{k}\dfn^{J+1}V=}\vspace{1mm}
\\
&
(n+2w-2)T^{i_1\ldots i_J}\Nv^{k}
\nd_{i_1}\ldots\nd_{i_J}\nd_{k}\dfn^{J+1}V=\vspace{1mm}
\\
&
(n+2w-2)(J+1)!\,T^{i_1\ldots i_J}\Nv^k
(\nd_{i_1}\dfn)\cdots(\nd_{i_J}\dfn)(\nd_k\dfn)V~.
\end{array}
\]
Thus
\[
\begin{array}{ll}
\lefteqn{\Nt^{A}T^{i_1\ldots i_J}\nd_{i_1}\ldots\nd_{i_J}D_A
  \dfn^{J+1}V=}\vspace{1mm}
\\
&
(n+2w-2-J)(J+1)!\,T^{i_1\ldots
  i_J}\Nv^{k}(\nd_{i_1}\dfn)\cdots(\nd_{i_J}\dfn)(\nd_k\dfn)V
\end{array}
\]
along $\sm$.  Let $\OpOne:=T^{i_1\ldots
  i_J}\Nv^k\nd_{i_1}\ldots\nd_{i_J} \nd_k$.  Then
\[
\begin{array}{ll}
\lefteqn{\Nt^{A}T^{i_1\ldots i_J}\nd_{i_1}\ldots\nd_{i_J}D_A=}\vspace{0.3ex}
\\
&
(n+2w-2-J)\OpOne+\Nt^{A}T^{i_1\ldots
  i_J}\nd_{i_1}\ldots\nd_{i_J}D_{A}\vspace{0.3ex}
-(n+2w-2-J)\OpOne~.
\end{array}
\]
The indicated sum of the second and third operators on the \RHS\ of
this equation is an operator of transverse order less than $J+1$, by
our above work, and we include this operator in $\OpTwo$.

Let
$w\in\Reals\backslash(E(\!\hspace{.17ex}\BasicOp\hspace{0.2ex})+1)$
and $p\in\sm$ be given, and choose a section $V$ of $\tbPhi[w-1]$ such
that $T^{i_1\ldots i_J}\nd_{i_1}\ldots\nd_{i_J}\dfn^JV$ is nonzero at
$p$.  Then $T^{i_1\ldots
  i_J}(\nd_{i_1}\dfn)\cdots(\nd_{i_J}\dfn)V$ is nonzero at $p$.  Thus
$\OpOne\dfn^{J+1}\xig V$ is nonzero at $p$, and hence
$w\in\Reals\backslash E(\OpOne)$.

Finally, let $w\in E(\!\hspace{0.17ex}\BasicOp\hspace{0.2ex})+1$ and
$V\in\tbPhi[w]$ be given.  Along $\sm$, we have
\begin{equation}\label{TranslateEwt}
\OpOne\dfn^{J+1}V
=
(J+1)\Nv^k(\nd_k\dfn)\xig T^{i_1\ldots
  i_J}\nd_{i_1}\ldots\nd_{i_J}\dfn^{J}(\xig)^{-1}V~.
\end{equation}
But $w-1\in E(\!\hspace{0.17ex}\BasicOp\hspace{0.2ex})$, and
$\BasicOp$ has transverse order less than $J$ for all weights in
$E(\!\hspace{0.17ex}\BasicOp\hspace{0.2ex})$.  It follows that
\nn{TranslateEwt} vanishes along $\sm$.  Thus $w\in E(\OpOne)$, and
$\OpOne:\tbPhi[w]\to\tbPhi[w-\lowerw-1]|_{\sm}$ has transverse order
less than $J+1$.
\end{proof}

The following theorem gives our refinement of $\delta_K$ in the
general case.
%
%
\begin{theorem}\label{key}
Let $J,\,k\in\Integers_{>0}$ be given.  There is a family of natural
conformally invariant differential operators
$
\deltaJk:  \ce[w]\to \ice[w-k-J]
$
determined as follows.
%
%
For $k\leq J$,
\[
  \deltaJk =N^{A_1}\cdots N^{A_k}\delta_J P_{A_1\cdots A_k}~.
\]
%
%
If $k>J$, then $\deltaJk$ is determined by the equation
\[
  (n+2w-2k)\deltaJk =N^{A_1}\cdots N^{A_k}\delta_J P_{A_1\cdots A_k}
\]
and polynomial continuation in $w$.  The $\deltaJk$ family
satisfies Hypotheses~\ref{NaturalHyp}, and the operators $\deltaJk$
have mass $J+k$ and order at most $J+k$ and are well-defined for all
$n$ and $\Mcc$ which satisfy Condition~\ref{nkcCondition}.  Moreover,
\begin{equation}\label{exceptw}
E(\deltaJk)=E(\delta_{J+k})\backslash\{k-n/2\}~,
\end{equation}
and $\deltaJk$ has transverse order $J+k$ for all $w\in
\Reals\backslash E(\deltaJk)$ and transverse order less than $J+k$ for
all $w\in E(\deltaJk)$.
\end{theorem}
\begin{proof} Suppose first that
$k>J$ and suppose also that $V\in\ce[k-n/2]$.  Then by
Proposition~\ref{GJMSp},
\[
\Nt^{A_1}\cdots\Nt^{A_k}\delta_J P_{A_1\ldots A_k}V=
\Nt^{A_1}\cdots\Nt^{A_k}\delta_J (-1)^kX_{A_1}\cdots X_{A_k}P_{2k}V~.
\]
This is zero along $\sm$, since $\Nt^AX_A=0$.  We will use this later
in the proof.

Now consider general positive integers $J$ and $k$.  Let $w\in\Reals$
be given, and suppose that $V\in\ce[w]$.  By the definition of
$P_{A_1\ldots A_k}$,
\begin{equation}\label{ExpandOutB}
\begin{array}{ll}
\lefteqn{\Nt^{A_1}\cdots\Nt^{A_k}\delta_J P_{A_1\ldots A_k}V=}
  \vspace{0.5ex}
\\
&
\Nt^{A_1}\cdots\Nt^{A_k}\delta_J D_{A_1}\ldots D_{A_k}V
\vspace{0.5ex}
+\Nt^{A_1}\cdots\Nt^{A_k}\delta_J(-X_{A_1}\LowTrac{2}{k}V)~.
\end{array}
\end{equation}
If $\LowTrac{2}{k}$ is nonzero, then every term of $\LowTrac{2}{k}V$
contains at least one occurrence of the tractor curvature $W$.  We may
use \nn{TracCurveOne}, \nn{TracCurveTwo}, and \nn{tractorD} to expand
all occurrences of $W$ and the tractor $D$-operator in the second term
on the right-hand side of \nn{ExpandOutB}.  We may then use the rules
in Figure~\ref{TrIProd} to eliminate all occurrences of $X$, $Y$, and
$Z$ in the resulting expression.  By \nn{TracCurveOne} and
\nn{TracCurveTwo}, each of the resulting terms will contain at least
one occurrence of the Weyl tensor or the Schouten tensor.
Since each term of
\nn{ExpandOutB} has mass $J+k$, it follows that the second term on the
right-hand side of \nn{ExpandOutB} has order less than $J+k$ for all
$w\in\Reals$.  This second term satisfies the conditions in
Hypotheses~\ref{NaturalHyp}.

By lemmas~\ref{LTOTS} and \ref{DoubleO},
\begin{equation}\label{L27June11a}
\begin{array}{ll}
    \Nt^{A_1}\cdots\Nt^{A_k}\delta_J D_{A_1}\ldots
  D_{A_k}V=
\left[
\prod_{i=1}^k(n+2w-2k+i-1-J)
\right]O_1V+O_2V
\end{array}
\end{equation}
for some natural operators $O_1$ and $O_2$ of mass $J+k$ and order at
most $J+k$.  The operators $O_1$ and $O_2$ may be given by symbolic
formulae which satisfy Hypotheses~\ref{FixingHyp}.  The operator $O_1$
has transverse order less than $J+k$ for all $w\in E(O_1)$, and
$E(O_1)=E(\delta_J)+k$.  The operator $O_2$ has transverse order less
than $J+k$ for all $w\in\Reals$.  By \nn{ExpandOutB} and
\nn{L27June11a},
\begin{equation}\label{ExpandOutC}
\begin{array}{ll}
\lefteqn{\Nt^{A_1}\cdots\Nt^{A_k}\delta_J P_{A_1\cdots A_k}V=}\vspace{1mm}
\\
&
\left[
\prod_{i=1}^k(n+2w-2k+i-1-J)
\right]O_1V+O_2V+\lots~.
\end{array}
\end{equation}

Suppose now that $k\leq J$, and let $S$ denote the set of all $w$
for which the explicit product in \nn{ExpandOutC} is zero.  Then
$E(\deltaJk)=S\cup(E(\delta_J)+k)$.  By \nn{E0} and an algebraic
exercise, it follows that $E(\deltaJk)=E(\delta_{J+k})$.  Since $k\leq
J$, it follows that $k-n/2\notin E(\delta_{J+k})$.  Thus \nn{exceptw}
holds in the case $k\leq J$.

Now suppose instead that $k>J$.  Let $O_2V+\lots$ be as in
\nn{ExpandOutC}.  Then $O_2V+\lots=O_3V$ for some natural operator
$O_3$ of mass $J+k$.  We apply Proposition~\ref{ACombination} with
$K=c=J+k$.  We conclude that $O_3=\sum_{\beta\in
  A}f_{\beta}(w,n)\RiemInv_{\beta,\g}$, where $A$ and
$\RiemInv_{\beta,\g}$ are as in Proposition~\ref{ACombination}.  Here
$f_{\beta}(w,n)$ is a polynomial in $w$ whose coefficients are real
rational functions of $n$.  By the polynomial division algorithm from
elementary algebra, we may write
\[
f_{\beta}(w,n)=(n+2w-2k)q_{\beta,n}(w)+r_{\beta,n}
\]
for each $\beta\in A$.  Here $q_{\beta,n}(w)$ is a polynomial in $w$
whose coefficients are real rational functions of $n$, and
$r_{\beta,n}$ is a real rational function of $n$.  We are dividing by
a polynomial whose leading term is $2w$.  Thus if $f_{\beta}(w,n)$ is
regular at some value of $n$, then $q_{\beta,n}(w)$ and $r_{\beta,n}$
are also regular for this value of $n$.

We now work with the same operator formula, but we trivialise the
density bundles as in Section \ref{OpBases}. So we operate now on
$V_0\in\ce[0]$.  Then
\[
O_3(\xig)^{k-n/2}V_0=\sum_{\beta\in
  A}f_{\beta}(k-n/2,n)\RiemInv_{\beta,\g}(\xig)^{k-n/2}V_0=0
\]
along $\sm$, by \nn{ExpandOutC} and our work at the beginning of this
proof.  Since $\nd((\xig)^{k-n/2})=0$, it follows that $\sum_{\beta\in
  A}f_{\beta}(k-n/2,n)\RiemInv_{\beta,\g}V_0=0$ along $\sm$.  For each
$\beta\in A$, $\RiemInv_{\beta,\g}$ determines an element
$[\RiemInv_{\beta,\g}]$ of $\Classes_{J+k,\,J+k,\,0}$.  Since
$V_0\in\ce[0]$ was arbitrary, it follows that $\sum_{\beta\in
  A}f_{\beta}(k-n/2,n)[\RiemInv_{\beta,\g}]$ is the zero element of
$\Classes_{J+k,\,J+k,\,0}$.  Thus by Proposition~\ref{RationalBasis},
we see that $f_{\beta}(k-n/2,n)=r_{\beta,n}=0$ for all $\beta\in A$
and all $n\in\Integers$.  We may therefore conclude that
\[
O_3=(n+2w-2k)\sum_{\beta\in A}q_{\beta,n}(w)\RiemInv_{\beta,\g}
\]
for general weights $w$.  Let $O_4=\sum_{\beta\in
  A}q_{\beta,n}(w)\RiemInv_{\beta,\g}$.  We may then let
\[
\textstyle
\deltaJk : =
\left[
\prod_{i=1,\,i\neq J+1}^k(n+2w-2k+i-1-J)
\right]O_1+O_4~.
\]
Polynomial continuation in $w$ shows that if $w=k-n/2$, then $O_4$ has
transverse order less than $J+k$.  Thus $O_4$ has transverse order
less than $J+k$ for all $w\in\Reals$.  This establishes \nn{exceptw}.
Polynomial continuation also shows that $\deltaJk$ is well-defined and
conformally invariant. \end{proof}

A simple key case is as follows.
\begin{corollary}\label{crit}
  Let $n_0\in\Integers_{\geq 4}$ and $n\in\Integers_{\geq 3}$ be
  given.  Also let $J=\lfloor(n_0-2)/2\rfloor$, and let
  $k=\lfloor(n_0+1)/2\rfloor$.  Suppose that $k$, $n$, and $\cc$
  satisfy Condition~\ref{nkcCondition}.  Then for all $w\in\Reals$,
  $\deltaJk:\ce[w]\rightarrow\ice[w-n_0+1]$ is well-defined and has
  order at most $n_0-1$.

  Now suppose in addition that $n=n_0$.  If the dimension $n$ is odd,
  then $E(\deltaJk)=\{0\}\cup\{i/2\,|\,2\leq i\leq
  n-3,\,i\in\mathbb{Z}\}$.  On the other hand, if $n$ is even, then
  $E(\deltaJk)=\{i/2\,|\,1\leq i\leq n-3,\,i\in\mathbb{Z}\}$.
\end{corollary}
We will use Corollary~\ref{crit} to identify a critical operator in
the proof of Theorem~\ref{SolnTheorem}, below.
%

%
%

%
%
%
\section{Conformal Dirichlet-to-Neumann Operators}\label{DToNOps}

In this section, we construct conformally invariant
Dirichlet-to-Neumann operators on a Riemannian conformal manifold with
boundary $(M,\Sigma,\cc)$.  Our construction uses the GJMS operators
$P_{2k}$ on $\Mcc$ and the hypersurface operator families
of this paper.  We begin with some groundwork in
Section~\ref{FSAsec}. We use standard theory of elliptic boundary
problems as in \cite{Grubb,H}; a summary of the key results needed
from those sources is given in \cite[Section 6]{BrGonon}.

%
%
%
%
\subsection{Formally self-adjoint boundary problems}\label{FSAsec}

For any $K\in\Integers_{>0}$, we will use $\FSAOp{K}$ to denote any of
the operators of order and transverse order $K$ from
Lemma~\ref{basicversion} and theorems~\ref{cflatkey} and \ref{key},
above,
as available according to context.  For any $w\in\Reals$, let
$\FSAOp{0}:\ce[w]\rightarrow\ice[w]$ be given by
$\FSAOp{0}u=u|_{\sm}$.  We also let $\delta_0$ denote $\FSAOp{0}$.
The following lemma describes some cases in which the operators
$\FSAOp{K}$ are available.
%
%
\begin{lemma}\label{deltaTK}
Let a positive integer $k$ be given, and suppose $n$, $k$, and $\Mcc$
satisfy Condition~\ref{nkcCondition}.  Let $i\in\Integers$ be given,
and suppose that $0\leq i\leq 2k-1$.  Then there is an operator
$\delta^{T}_{i}:\ce[k-n/2]\rightarrow\ice[k-n/2-i]$ of order and
transverse order $i$.  In particular, $k-n/2\notin E(\delta^T_i)$.
\end{lemma}
\begin{proof}
If $i=0$, let $\FSAOp{i}=\delta_0$.  Now suppose instead that $1\leq
i\leq k$.  Then $k-n/2\notin E(\delta_i)$, by \nn{E0}, so we may let
$\FSAOp{i}=\delta_i$.  Finally, suppose that $k+1\leq i\leq 2k-1$.
Let $j=i-k$.  Then $k-n/2\notin E(\delta_{j,k})$, by
Theorem~\ref{key}, so we may let $\FSAOp{i}=\delta_{j,k}$.
\end{proof}

Let $M$ be a compact conformal $n$-manifold with smooth boundary
$\Sigma=\partial M$. Recall that a boundary problem $(A,B)$ is {\em
  formally self-adjoint} (FSA), or {\em symmetric}, if and only if
\begin{equation*}
\int_M[(Au)v-uAv] =0
\end{equation*}
whenever $Bu=Bv=0$ on $\Sigma$.  

Recall that each GJMS operator 
$
P_{2k}:\ce[k-n/2]\to\ce[-k-n/2]
$
is conformally invariant and FSA \cite{FGQandP,GrZ}.
For each such GJMS operator, we will consider the associated {\em
  conformal generalised Dirichlet problem} $(P_{2k},B)$, where $B$ is
the multi-boundary operator
\begin{equation}\label{MultiBOp}
B=(\FSAOp{0},\FSAOp{1},\cdots,\FSAOp{k-1})~.
\end{equation}
We will need the following observation.
\begin{lemma}\label{jet} Let $k\in\Integers_{>0}$ and a section $u$ of
  $\ce[k-n/2]$ be given.  Suppose that $u$ is in the kernel of $B$,
  i.e.
$$
0=\FSAOp{0}u =\FSAOp{1}u = \cdots =\FSAOp{k-1}u  \quad
\mbox{along}\quad \Sigma~.
$$
Then the {\upshape($k-1$)}-jet of $u$ is zero at every point of $\Sigma$.
\end{lemma}
\begin{proof}Let $s$ be any local defining function for $\sm$.  We
  will use induction to show that for all
  $\ell\in\{1,\,2,\,\ldots,\,k\}$, there is a smooth section
  $u_{\ell}$ of $\ce[k-n/2]$ such that $u=s^{\ell}u_{\ell}$.  Let
  $c\in\{2,\,3\,\ldots,\,k\}$ be given, and suppose that
  $u=s^{c-1}u_{c-1}$ for some $u_{c-1}\in\ce[k-n/2]$.  Then along
  $\sm$, $\FSAOp{c-1}u=F^{i_1\cdots
    i_{c-1}}(\nd_{i_1}s)\cdots(\nd_{i_{c-1}}s)u_{c-1}=0$ for some
  component functions $F^{i_1\cdots i_{c-1}}$.  Since $\FSAOp{c-1}$
  has transverse order $c-1$ in this context, it follows that
  $u_{c-1}=su_{c}$ for some section $u_c$ of $\ce[k-n/2]$.%
\end{proof}

This leads to the following result.
\begin{theorem}\label{FSAthm}
Let $k\in\Integers_{>0}$ be given, and let $B$ be as in \nn{MultiBOp}.
Suppose that $n$, $k$, and $\Mcc$ satisfy
Condition~\ref{nkcCondition}.  Then the conformal generalised
Dirichlet problem $(P_{2k},B)$ is FSA.
\end{theorem}
\begin{proof} 
Let sections $u$ and $v$ of $\ce[k-n/2]$ be given, and fix a metric
$\g\in\cc$ for the purpose of calculating.  Since $P_{2k}$ is FSA, we
have
\begin{equation}\label{IntParts}
\int_M[(P_{2k}u)v-u P_{2k}v] =\int_{\Sigma} {\rm sk}_{P_{2k}}(u,v)~,
\end{equation}
where ${\rm sk}_{P_{2k}}(u,v)$ is a skew bilinear form on the Cauchy
data and its tangential derivatives at the boundary $\Sigma$.  To see
this, repeatedly integrate terms of the left-hand side of
\nn{IntParts} by parts in the usual way.  All integrals over $M$
eventually cancel out, and the boundary integral of a bilinear form
$b(u,v)$ remains.  But this boundary integral is skew-symmetric in $u$
and $v$, so in this integral, we may replace $b(u,v)$ with its skew
part.  In each term of ${\rm sk}_{P_{2k}}(u,v)$, either $u$ or $v$ is
differentiated at most $k-1$ times, so the result follows from
Lemma~\ref{jet}.
\end{proof}

To construct our conformal Dirichlet-to-Neumann operators, we will
work with well-posed boundary problems, in particular, boundary
problems which satisfy the so-called Lopatinski-Shapiro condition (see
e.g. \cite{ADN,BrGonon,H}).  We will assume that $M$ is compact with
nonempty boundary $\sm$.  We define the Lopatinski-Shapiro condition
below.  We will use the definition given in \cite{BrGonon}.  Our
definition begins with the idea of elliptic and properly elliptic
operators.  Let a differential operator $A$ of order $j$ acting on
$\ce[w]$ be given, and let $a_j(x,\xi)$ denote the leading symbol of
$A$.  We say that $A$ is \textit{elliptic} if there is a constant
$C>0$ such that $|a_j(x,\xi)|\geq C|\xi|^j$ on $M$.  Here
$|\xi|^2=\g^{ab}\xi_a\xi_b$.  In this context, we will always assume
that $j$ is even.

Now consider arbitrary $x\in\sm$ and $\eta_a\in\ce_a$, as well as the
polynomial $p(\tau)=a_j(x,\eta_a+\tau\Nv_a)$.  We say that $A$ is
\textit{properly elliptic} if for all $x\in\sm$ and all $\eta_a$
\textit{not} parallel to $\Nv_a$ at $x$, the zeros of $p(\tau)$ are
separated by the real axis.  This means that one may label the zeros
as $\tau^{\pm}_i$, for $i=1,\ldots,j/2$, and do this in such a way
that $\Im \tau^+_i>0$ and $\Im\tau^-_i<0$ for all $i$.

We may now define the Lopatinski-Shapiro condition.  Let
$k\in\Integers_{>0}$, a properly elliptic differential operator $A$ of
order $2k$, and a multi-boundary operator $O=(O_1,\ldots,O_k)$ be
given.  For each $i$, suppose that $O_i$ has order $m_i$ and leading
symbol $b_i(x,\xi)$.  Suppose also that $0\leq m_1<m_2<\cdots<m_k<2k$.
For every $i\in\{1,\ldots,k\}$, let
$p_i(\tau)=b_i(x,\eta_a+\tau\Nv_a)$.  Let $\tau_i^{+}$ be as above,
and let $M^+=\Pi_{i=1}^k(\tau-\tau_i^+)$.  Let $I$ denote the
principal ideal in $\ComplexNumbers[\tau]$ generated by $M^+$.  Each
$f(\tau)$ in $\ComplexNumbers[\tau]$ determines an element $[f(\tau)]$
of $\ComplexNumbers[\tau]/I$.  We say that $(A,O)$ satisfies the
\textit{Lopatinski-Shapiro} condition if, for all $x\in\sm$ and all
$\eta_a$ \textit{not} parallel to $\Nv_a$ at $x$, the  elements
$[p_i(\tau)]$ are linearly independent in $\ComplexNumbers[\tau]/I$.

If $n$, $k$, and $\cc$ satisfy Condition~\ref{nkcCondition}, then
$P_{2k}$ is elliptic and properly elliptic.  If $A=P_{2k}$, then there
is an $\alpha\in\ComplexNumbers$ such that $M^+=(\tau-\alpha)^k$.  We
will use this formula for $M^+$ in the proof of the following lemma.
\begin{lemma}\label{LSLemma1}
Let $k\in\Integers_{>0}$ be given, and let $O$ denote the
multi-boundary operator
$(\delta_0,\,\Nv^a\nd_a,\,(\Nv^a\nd_a)^2,\,\ldots,\,(\Nv^a\nd_a)^{k-1})$.
If $n$, $k$, and $\Mcc$ satisfy Condition~\ref{nkcCondition}, then the
boundary problem $(P_{2k},O)$ satisfies the Lopatinski-Shapiro
condition.
\end{lemma}
\begin{proof}
Let $j\in\{0,1,\ldots,k-1\}$ be given, and note that
$(\Nv^a\nd_a)^{j}$ has leading symbol $b_j(x,\xi)=C_j(\Nv^a\xi_a)^j$
for some nonzero constant $C_j$.  Let $x\in\sm$ be given, and suppose
that $\eta_a$ is not parallel to $\Nv_a$ at $x$.  Then
$b_j(x,\eta_a+\tau\Nv_a)=C_j(\Nv^a\eta_a+\tau)^j$.

Now let $I$ denote the principal ideal in $\ComplexNumbers[\tau]$
generated by $M^{+}$.  We claim that the elements
$[C_j(\Nv^a\eta_a+\tau)^j]$, $0\leq j\leq k-1$, are linearly
independent in $\ComplexNumbers[\tau]/I$.  To see this, suppose there
are constants $C_0'$,\,\ldots,\,$C_{k-1}'$, not all zero, such that
$\sum_{j=0}^{k-1}C_j'[C_j(\Nv^a\eta_a+\tau)^j]=I$.  Then there is a
polynomial $h(\tau)$ in $\ComplexNumbers[\tau]$ such that
\begin{equation}\label{DependentSum}
\sum_{j=0}^{k-1}C_j'C_j(\Nv^a\eta_a+\tau)^j=h(\tau)(\tau-\alpha)^k~.
\end{equation}
But then the left-hand side of \nn{DependentSum} is a polynomial of
lower degree than the polynomial on the right-hand side of
\nn{DependentSum}.  This is a contradiction.  The linear independence
claim therefore holds, and thus $(P_{2k},O)$ satisfies the
Lopatinski-Shapiro condition.
\end{proof}
\begin{lemma}\label{LSLemma2}
Let $k\in\Integers_{>0}$ be given, and suppose that $n$, $k$, and
$\Mcc$ satisfy Condition~\ref{nkcCondition}.  Let $B$ be as in
\nn{MultiBOp}.  Then the boundary problem $(P_{2k},B)$ satisfies the
Lopatinski-Shapiro condition.
\end{lemma}
\begin{proof}
Let $j\in\{0,\ldots,k-1\}$ be given.  By Riemannian invariant theory,
there exist intrinsic natural differential operators $D_m$ on $\sm$
(for $0\leq m\leq j-1$) of order at most $j-m$ such that
\[
\FSAOp{j}=c_j(\Nv^a\nd_a)^j+\sum_{m=0}^{j-1}D_m(\Nv^a\nd_a)^m
\]
along $\sm$.  Here $c_j$ is a nonzero constant.
The result then follows by Lemma~\ref{LSLemma1},
above, and Lemma~6.3 of \cite{BrGonon}.
\end{proof}
Our construction of Dirichlet-to-Neumann operators will also require
the following definition:

\begin{definition}\label{NormalDef}
Let $k\in\Integers_{>0}$ and a multi-boundary operator
$O=(O_1,\ldots,O_k)$ be given.  Suppose that $O_1$,\,\ldots,\,$O_k$
are differential operators of orders $m_1$,\,\ldots,\,$m_k$,
respectively, and suppose that $0\leq m_1<\cdots<m_k<2k$.  Finally,
suppose that for each $i\in\{1,\ldots,k\}$, the operator $O_i$ has
transverse order $m_i$.  Then we say that $O$ is a \textit{normal}
multi-boundary operator.
\end{definition}
Definition~\ref{NormalDef} is equivalent to the definition of a normal
system of boundary operators given on page~34 of \cite{BrGonon}.  The
multi-boundary operator $B$ of \nn{MultiBOp} is normal.

\subsection{The fraction Laplacians}\label{FracLap}

The next theorem gives our construction of conformally invariant
Dirichlet-to-Neumann operators.
\begin{theorem}\label{nonloc}
Let a compact Riemannian conformal manifold with boundary
$(M,\Sigma,\cc)$ be given. Let $k\in \mathbb{Z}_{>0}$ and $\ell\in
\{0,1,\cdots ,k-1\}$ be given, and suppose that $n$, $k$, and $\Mcc$
satisfy Condition~\ref{nkcCondition}.  Let $B$ be as in \nn{MultiBOp},
and suppose that the conformal generalised Dirichlet problem
$(P_{2k},B)$ has trivial kernel.  Then there is a well-defined
conformally invariant Dirichlet-to-Neumann operator
$$
P^{T,k}_{2m}:
\ice\left[m-\frac{\sdim}{2}\right]
\to
\ice\left[-m-\frac{\sdim}{2}\right]
$$
given by 
$$
\Gamma\left(\ice\left[m-\frac{\sdim}{2}\right]\right)
\ni f\mapsto \FSAOp{2k-1-\ell} u~.
$$
Here $m: = k-1/2-\ell$, and $u$ solves the conformal generalised
Dirichlet problem
\begin{equation}\label{bp}
P_{2k}u=0,\quad\FSAOp{\ell} u=f,\quad
\mbox{$\FSAOp{j}u=0$ for $j\neq \ell$ and $0\leq j\leq k-1$}~.
\end{equation}
The operator $P^{T,k}_{2m}$ has leading term $(-\ila)^m$.
\end{theorem}
\begin{proof}
The problem $(P_{2k},B)$ is properly elliptic, normal, and FSA.  It
satisfies the Lopatinski-Shapiro condition, and by assumption, it has
trivial kernel.  Thus by standard theory it follows that \nn{bp} has a
unique solution $u$ (see e.g.\ \cite[Section~20.1]{H}).  Thus the map
$f\mapsto \FSAOp{2k-1-\ell} u $ gives a well-defined non-local
operator $P_{2m}^{T,k}:\ice[m-(\sdim)/2]\to \ice[-m-(\sdim)/2]$. By
construction this is conformally invariant.

In the case of the dimension $n$ hemisphere, equipped with the
standard conformal structure of the round sphere, it is well-known
that there is a unique intertwinor between boundary conformal density
bundles and with the domain space $\Gamma (\bar{\ce} [m-(\sdim)/2]
)$. This takes values in $\Gamma (\bar{\ce} [-m-(\sdim)/2] )$ and has
leading term $(-\ila)^m$.  See \cite{BrGonon}. Thus by the
universality of the construction, the final claim follows.
\end{proof}
%
%

%
%
%
\section{$Q$-Type Curvatures}\label{Qsec}%

In this section, we define and investigate the $Q$-type curvatures
that we discussed in Section~\ref{IntroSect}, above.  Our definition
is very general.  It associates $Q$-type curvatures to the operator
families $\delta_K$, $\ConFlatOp{K}$, and $\deltaJk$ of
Section~\ref{confi} and to any other operator families satisfying
appropriate properties.
This general $Q$-type curvature will exist in all dimensions in which
the associated operators exist.  The construction of our $Q$-type
curvature resembles the construction of Branson's $Q$-curvature in
\cite{Brsharp}, and under conformal change of metric, our $Q$-type
curvature will transform in essentially the same way as Branson's
$Q$-curvature.  Note, however, that the integral of Branson's
$Q$-curvature over a compact manifold without boundary is a conformal
invariant; we do \textit{not} claim that our $Q$-type
curvatures have this property.  Our curvature definition applies to
families of natural hypersurface operators $\psiOp$ which satisfy the
following hypotheses:
\begin{hypotheses}\label{genQhyp}
For some $\lowerw\in\Reals$, $K\in\Integers_{>0}$, and
$\Dimensions\subseteq\Integers_{>0}$, the family of operators
$\psiOp:\ce[w]\rightarrow\ice[w-\lowerw]$ satisfies
Hypotheses~\ref{NaturalHyp}.  The operators $\psiOp$ are conformally
invariant and have mass $K$ and order at most $K$.  The set
$E(\psiOp)$ of Definition~\ref{ExSet} may depend on $n$ but is always
finite.  Finally, in all dimensions $n\in\Dimensions$, $\psiOp$
annihilates constant densities of weight zero.  Specifically, if $f$
is a constant section of $\ce[0]$, then $\psiOp f=0$ along $\sm$.
\end{hypotheses}
The definition of our $Q$-type curvatures requires some groundwork.
\begin{definition}\label{IndexSets}
Let $\{\RiemInv_{\beta,\g}\}_{\beta\in A}$ be the basis for
$\RealCoeff_{K,c,w}$ that we discussed in Proposition~\ref{RealBasis},
above.  Let $B$ denote the set of all $\beta\in A$ such that
$\RiemInv_{\beta,\g}$ is a composition of the form $O_{\beta,\g}\o
\nabla$ for some natural operator $O_{\beta,\g}$.  Let $C$ denote the
set of all $\beta\in A$ such that $\RiemInv_{\beta,\g}$ is a
zero-order operator.
\end{definition}
\noindent Note that $A=B\cup C$ and $B\cap C=\emptyset$.  Now recall
that if $\g$ is any Riemannian metric, then $\xig$ denotes the scale
density associated to $\g$.
\begin{definition}\label{psigwnotation}
Let $\lowerw\in\Reals$ and a family of conformally invariant natural
differential operators $\psi:\ce[w]\rightarrow\ice[w-\lowerw]$ be
given.  For all $w\in\Reals$ and all $\g\in\cc$, define
$\psigw{w}:C^{\infty}(M)\rightarrow C^{\infty}(\sm)$ by letting
$\psigw{w}f=\xigw{-w+\lowerw}\psiOp\xigw{w}f$ for all sections $f$ of
$\ce[0]$.  Now let $\RealBasis$ be as in Proposition~\ref{RealBasis}.
For every $w\in\Reals$ and $\beta\in A$, let
$\URiemInv_{\beta,\g}=\xigw{-w+\lowerw}\RiemInv_{\beta,\g}\xigw{w}$.
\end{definition}
\noindent In Definition~\ref{psigwnotation}, $\psigw{w}$ is
conformally covariant of bidegree $(-w,-w+\lowerw)$, by
Proposition~\ref{ReplaceCCov}.  The definition of our $Q$-type
curvatures will rely on the
following proposition.
\begin{proposition}\label{KillConsForm}
Let a family of natural hypersurface operators $\psiOp$ be given, and
suppose that $\psiOp$ satisfies Hypotheses~\ref{genQhyp}.  Let
notation be as in definitions~\ref{IndexSets} and \ref{psigwnotation}.
Then for all $w\in\Reals$,
\begin{equation}\label{Lab2Nov9a}
\psigw{w}=
\displaystyle
\sum_{\beta\in
  B}\fkbeta(w,n)\,\URiemInv_{\beta,\g}
+
\sum_{\beta\in C}w\,\qkbetan(w)\,\URiemInv_{\beta,\g}~,
\end{equation}
where both $\fkbeta(w,n)$ and $\qkbetan(w)$ are polynomial in $w$ and
rational in $n$.
\end{proposition}
\begin{proof}
By Proposition~\ref{ACombination},
$\psiOp=\sum_{\beta\in B\cup
  C}\fkbeta(w,n)\RiemInv_{\beta,\g}$,
where each function $\fkbeta(w,n)$ is polynomial in $w$ and rational
in $n$.  Let $\beta\in C$ be given.  By the polynomial division
algorithm from elementary algebra, we may write
$
f_{\beta}(w,n)=w q_{\beta,n}(w)+r_{\beta,n}
$.
Here $q_{\beta,n}(w)$ and $r_{\beta,n}$ are rational in $n$, and
$q_{\beta,n}(w)$ is polynomial in $w$.  Note that $q_{\beta,n}(w)$ and
$r_{\beta,n}$ are clearly regular at all $n\in\Dimensions$, where
$\Dimensions$ is as in Hypotheses~\ref{genQhyp}.  Since $\psiOp$
annihilates constant densities of weight zero, it follows that
$\sum_{\beta\in C}f_{\beta}(0,n)\RiemInv_{\beta,\g}=0$ for all
$n\in\Dimensions$.  Thus $\sum_{\beta\in
  C}f_{\beta}(0,n)[\RiemInv_{\beta,\g}]$ is the zero element of
$\Classes_{K,c,0}$.  Here $[\RiemInv_{\beta,\g}]$ denotes the element
of $\Classes_{K,c,0}$ determined by $\RiemInv_{\beta,\g}$.  Thus
$f_{\beta}(0,n)=r_{\beta,n}=0$ for all $\beta\in C$ and all
$n\in\Integers$, by Proposition~\ref{RationalBasis}.  It follows that
$f_{\beta}(w,n)=wq_{\beta,n}(w)$ for $\beta\in C$, all $w\in\Reals$,
and all $n\in\Dimensions$.  The result follows.
\end{proof}
The next definition gives our new general $Q$-type curvatures.
%
%
\begin{definition}\label{generalQ}
Let a family of natural hypersurface operators $\psiOp$ be given, and
suppose that $\psiOp$ satisfies Hypotheses~\ref{genQhyp}.  Let
notation be as above.  We define the following curvature quantity
along $\sm$:
\begin{equation}\label{genQdef}
Q_{\g}(\psiOp):=-\sum_{\beta\in C}\qkbetan(0)\,\URiemInv_{\beta,\g}~.
\end{equation}
\end{definition}
\begin{remark}
The curvature quantity $Q_{\g}(\psiOp)$ is a family of hypersurface
invariants parametrised by the dimension $n$.  We will only be
interested in the values of $Q_{\g}(\psiOp)$ at points $p\in\sm$, and
$Q_{\g}(\psiOp)$ is uniquely determined at such points.  Note,
however, that the right-hand side of \nn{genQdef} is defined on an
open neighbourhood of $M$; this will be important in the proof of
Proposition~\ref{H0}, below.
\end{remark}

The following proposition describes a transformation rule for
$Q_{\g}(\psiOp)$ under conformal change of metric.
%
%
\begin{proposition}\label{genQtran}
Let $\psiOp$ be as in Definition~\ref{generalQ}, and let notation be
as above.  Then
\[
e^{\lowerw\Upsilon}Q_{\hat{\g}}(\psiOp)
=
Q_{\g}(\psiOp)+\psigw{0}\Upsilon~.
\]
If we view $Q_{\g}(\psi)$ as a density of weight $-\lowerw$, then the
transformation rule becomes
\[
Q_{\hat{\g}}(\psi)=
Q_{\g}(\psi)+\psi\Upsilon~.
\]
\end{proposition}
%
%
\begin{proof}
We apply $e^{(-w+\lowerw)\Upsilon}\psighw{w}$ to the constant
function $1$.  By conformal covariance, we obtain
\begin{equation}\label{L2July12a}
\begin{array}{lll}
\lefteqn{\textstyle e^{(-w+\lowerw)\Upsilon}\sum_{\beta\in
    C}w\,\qkbetan(w)\,\URiemInv_{\beta,\hatg}=}&&\vspace{0.4ex}
\\
&&
-w\,e^{-w\Upsilon}\sum_{\beta\in
    B}\fkbeta(w,n)\,\URiemInv_{\beta,\g}\,\Upsilon
+
w^2\,e^{-w\Upsilon}\,\ell_{\g}(w,n,\Upsilon)\vspace{0.4ex}
\\&&
+w\,e^{-w\Upsilon}\sum_{\beta\in C}\qkbetan(w)\,\URiemInv_{\beta,\g}~.
\end{array}
\end{equation}
Here $\ell_{\g}(w,n,\Upsilon)$ is defined as follows.  Let $\beta\in
B$ be given.  Then
$$
\URiemInv_{\beta,\g}e^{-w\Upsilon}
=
e^{-w\Upsilon}(-w\,\URiemInv_{\beta,\g}\Upsilon+w^2k_{\beta,\g}(w,\Upsilon))~,
$$
where $k_{\beta,\g}(w,\Upsilon)$ is polynomial in $w$.  Let
$\ell_{\g}(w,n,\Upsilon)=\sum_{\beta\in
  B}f_{\beta}(w,n)k_{\beta,\g}(w,\Upsilon)$.  Now divide
\nn{L2July12a} by $-w\,e^{-w\Upsilon}$ to obtain
\begin{equation}\label{QtranB}
\begin{array}{lll}
\lefteqn{\textstyle -e^{\lowerw\Upsilon}\sum_{\beta\in
    C}\qkbetan(w)\,\URiemInv_{\beta,\hatg}=}&&\vspace{0.4ex} \\ &&
\sum_{\beta\in B}\fkbeta(w,n)\,\URiemInv_{\beta,\g}\,\Upsilon
-w\,\ell_{\g}(w,n,\Upsilon)
%
-\sum_{\beta\in
    C}\qkbetan(w)\,\URiemInv_{\beta,\g}~.
\end{array}
\end{equation}
This equation holds in the case $w=0$, since both sides are polynomial
in $w$.  In this case, we may use \nn{Lab2Nov9a} to identify the
presence of $\psigw{0}$ on the \RHS\ of \nn{QtranB}, and we
see that the proposition holds.
\end{proof}
%
%
%

We now apply Definition~\ref{generalQ} to the operator families
$\delta_K$, $\ConFlatOp{K}$, and $\deltaJk$.
%
%
\begin{proposition}\label{KillConstants}

Let $K\in\Integers_{>0}$ be given, and let
$\psi:\ce[w]\rightarrow\ice[w-K]$ denote the $\delta_K$ operator
family of Lemma~\ref{basicversion}, as acting on densities, the
$\ConFlatOp{K}$ family of Theorem~\ref{cflatkey}, or the $\deltaJk$
family of Theorem~\ref{key}.  (In the latter case, $n$, $k$, and
$\Mcc$ must satisfy Condition~\ref{nkcCondition}, and $J+k$ must equal
$K$.)  Then $\psi$ satisfies Hypotheses~\ref{genQhyp}, and
Definition~\ref{generalQ} gives a $Q$-type curvature $Q_{\g}(\psi)$
associated to $\psi$ and a representative $\g$ of the conformal
structure on $M$.
\end{proposition}
\begin{proof}
Suppose first that $\psi=\delta_K$.  The proposition then follows
trivially, since by \nn{tractorD}, the Thomas $D$-operator annihilates
constant functions.

The proofs for $\psi=\ConFlatOp{K}$ and $\psi=\deltaJk$ are similar.
We will give the proof for $\psi=\deltaJk$ only.  We begin by letting
$\psi=\deltaJk$.  We will consider the family of conformally invariant
operators $\OJk:\ce[w]\rightarrow\ice[w-J-k]$ given by
$\OJk:=\Nt^{A_1}\cdots\Nt^{A_k}\delta_JP_{A_1\cdots A_k}$.  By
Proposition~\ref{GJMSp}, we know that $P_{A_1\cdots A_k}$ is a family
of conformally invariant natural operators of mass $k$ given by a
universal symbolic formula.  It follows that $\OJk$ has mass $J+k$ and
satisfies Hypotheses~\ref{NaturalHyp}.  By
propositions~\ref{GJMSTractor} and \ref{GJMSp}, $P_{A_1\cdots A_k}$ is
given by either
\[
D_{A_1}\ldots D_{A_k}\mbox{\ \ \ \ or\ \ \ \ }D_{A_1}\ldots
D_{A_k}-X_{A_1}\LowTracPsi{2}{k}{PQ}D_PD_Q~.
\]
Here $\LowTracPsi{2}{k}{PQ}$ is as in Proposition~\ref{GJMSTractor}.
Thus $\OJk$ annihilates constant functions, and the proposition
follows in the case $k\leq J$.  We may thus assume that $k>J$.

Let $\Dimensions$ denote the set of all dimensions $n$ in which
$\deltaJk$ is well-defined.  We apply Proposition~\ref{ACombination}
with $\Dimensions$ and with $K=c=J+k$.  We may write
$\OJk=\sum_{\beta\in B\cup C}\fkbeta(w,n)\RiemInv_{\beta,\g}$, where
$\RiemInv_{\beta,\g}$, $B$, and $C$ are as in Definition~\ref{IndexSets}.
For each $\beta\in B\cup C$, $\fkbeta(w,n)$ is a polynomial in $w$
whose coefficients are real rational functions of $n$.  For all
$\beta\in C$, we may use polynomial division to symbolically compute
functions $\skbetan(w)$ and $\tkbetan(w)$, such that
\begin{equation}\label{LongDivThree}
\fkbeta(w,n)=w(n+2w-2k)\skbetan(w)+\tkbetan(w)~.
\end{equation}
In this polynomial division, we treat $\fkbeta(w,n)$ and $w(n+2w-2k)$
as polynomials in $w$.  Thus $\skbetan(w)$ and $\tkbetan(w)$ are
polynomials in $w$ whose coefficients are real rational functions of
$n$.  These rational functions are regular at all $n\in\Dimensions$,
since the divisor has leading term $2w^2$.  The degree of
$\tkbetan(w)$ is at most 1.

For all $\beta\in B\cup C$, let $[\RiemInv_{\beta,\g}]$ denote the
element of $\Classes_{J+k,\,J+k,\,0}$ determined by
$\RiemInv_{\beta,\g}$.  Since $\OJk$ annihilates constant densities of
weight zero, it follows that $\sum_{\beta\in
  C}\fkbeta(0,n)\RiemInv_{\beta,\g}$ is the zero operator for all $n$
in $\Dimensions$.  It therefore follows that $\sum_{\beta\in
  C}f_{\beta}(0,n)[\RiemInv_{\beta,\g}]$ is the zero element of
$\Classes_{J+k,\,J+k,\,0}$.  Thus $\fkbeta(0,n)=0$ for all $\beta\in
C$ and all $n\in\Integers$, by Proposition~\ref{RationalBasis}, so
$\tkbetan(0)=0$ for all $\beta\in C$ and all $n\in\Dimensions$, by
\nn{LongDivThree}.

Now consider any $V_0\in\ce[0]$.  Then
$(\xig)^{k-n/2}V_0\in\ce[k-n/2]$, so by Theorem~\ref{key},
$$
\sum_{\beta\in B\cup C}
f_{\beta}(k-n/2,n)\RiemInv_{\beta,\g}(\xig)^{k-n/2}V_0=
\OJk(\xig)^{k-n/2}V_0=0
$$
along $\sm$.  But $\nd((\xig)^{k-n/2})=0$, so $\sum_{\beta\in B\cup C}
f_{\beta}(k-n/2,n)\RiemInv_{\beta,\g}V_0$ vanishes along $\sm$.  It
then follows that $\sum_{\beta\in B\cup
  C}f_{\beta}(k-n/2,n)[\RiemInv_{\beta,\g}]$ is the zero element of
$\Classes_{J+k,\,J+k,\,0}$.  Thus $\fkbeta(k-n/2,n)=0$ for all
$\beta\in B\cup C$ and for all $n\in\Integers$, by
Proposition~\ref{RationalBasis}, so $\tkbetan(k-n/2)=0$ for all
$\beta\in C$ and all $n\in\Dimensions$, by \nn{LongDivThree}.

Let $\beta\in C$ and $n\in\Dimensions\backslash\{2k\}$ be given.  Then
$\tkbetan(w)$ is a polynomial in $w$ with two distinct zeros, so the
coefficients of $\tkbetan(w)$ must vanish in dimension $n$.  Since
$\Dimensions\backslash\{2k\}$ is an infinite set, these coefficients
must therefore vanish for all $n\in\Integers$.  Thus
$f_{\beta}(w,n)=w(n+2w-2k)s_{\beta,n}(w)$ for all $n\in\Dimensions$
and all $w\in\Reals$.

Similar reasoning shows that for all $\beta\in B$ and all
$n\in\Dimensions$, we have $\fkbeta(w,n)=(n+2w-2k)h_{\beta,n}(w)$.
Here $h_{\beta,n}(w)$ is polynomial in $w$ and rational in $n$.  Thus
by polynomial continuation in $w$,
\[
\deltaJk=\sum_{\beta\in B}h_{\beta,n}(w)\RiemInv_{\beta,\g}
+\sum_{\beta\in C}w\,\skbetan(w)\RiemInv_{\beta,\g}~.
\]
Thus $\deltaJk$ annihilates constant densities of weight zero.
\end{proof}
%
%
\begin{corollary}\label{GenSolveCorollary}
For all $K\in\Integers_{>0}$, the operator families $\delta_K$,
$\ConFlatOp{K}$, and $\deltaJk$ solve Problem~\ref{GenProblem}.  (In
the case of $\deltaJk$, we assume that $n$, $k$, and $\Mcc$ satisfy
Condition~\ref{nkcCondition} and that $J+k=K$.)
\end{corollary}

The following theorem gives higher-order analogues of $\CQOp$ in the
conformally flat setting and in the case of general Riemannian
metrics.
\begin{theorem}\label{SolnTheorem}
Let an even integer $n_0\geq 4$ be given, and let $K=n_0-1$.  Then
$\ConFlatOp{n_0-1}$ and $\delta_{(n_0-2)/2,\,n_0/2}$ solve
Problem~\ref{CQprob} in dimension $n=n_0$.
\end{theorem}
\begin{proof}
If $n=n_0$, then $0\notin E(\ConFlatOp{n_0-1})$ and $0\notin
E(\delta_{(n_0-2)/2,\,n_0/2})$, by corollaries~\ref{cflatEwts} and
\ref{crit}.  So if $n=n_0$, then $\ConFlatOp{n_0-1}$ and
$\delta_{(n_0-2)/2,\,n_0/2}$ are critical operators.
\end{proof}
%

%
%
The next two theorems give symbolic tractor formulae associated to
$Q_{\g}(\deltaJk)$, $Q_{\g}(\delta_K)$, and $Q_{\g}(\ConFlatOp{K})$.
%
%
\begin{theorem}\label{TFormula}
Let $J,\,k\in\Integers_{>0}$ be given, and view $Q_{\g}(\deltaJk)$ as
a density of weight $-J-k$.  Let $I_{E}=(n-2)Y_{E}-\J X_E$.  Then
there is a family of conformally invariant differential operators
$\POneAB:\tbn_E[w]\rightarrow\tbn_{A_1\dots A_k}[w-k+1]$ with the
following properties:
\begin{itemize}
\item It is well-defined for all $n$ and $\Mcc$ satisfying
  Condition~\ref{nkcCondition}.
\item It is given by a polynomial in $D$, $W$,
  $X$, $\tm$, and $\tmco$ of mass $k-1$ whose coefficients are
  rational in $n$.
\item If $k>J$, then
\begin{equation}\label{DefineQ}
(n-2k)Q_{\g}(\deltaJk)=-\Nt^{A_1}\cdots\Nt^{A_k}\delta_J\POneAB I_E~.
\end{equation}
\item If $k\leq J$, then
\[
Q_{\g}(\deltaJk)=-\Nt^{A_1}\cdots\Nt^{A_k}\delta_J\POneAB I_E~.
\]
\end{itemize}
\end{theorem}
\begin{proof}
Let $P_{A_1\ldots A_k}$ be as in Proposition~\ref{GJMSp}, above, and
suppose first that $k>J$.  Then by Proposition~\ref{GJMSTractor}, we
may write $P_{A_1\ldots A_k}=\POneAB D_E$, where $\POneAB$ is a
polynomial in $D$, $W$, $X$, $\tm$, and $\tmco$
whose coefficients are rational in $n$.  Let $\Dimensions$ denote the
set of all dimensions $n$ in which $\deltaJk$ is defined.  By
Proposition~\ref{KillConsForm},

\[
\delta_{J,k,\g,w}=\sum_{\beta\in
  B}\fkbeta(w,n)\,\URiemInv_{\beta,\g}+ \sum_{\beta\in
  C}w\,\qkbetan(w)\,\URiemInv_{\beta,\g}~.
\]
Here notation is as in definitions~\ref{IndexSets} and
\ref{psigwnotation} and Proposition~\ref{KillConsForm}.  By
\nn{tractorD},
\begin{equation}\label{L7June14a}
\begin{array}{ll}
\lefteqn{(n+2w-2k)\delta_{J,k,\g,w}
1=
\sum_{\beta\in
  C}w\,(n+2w-2k)\,\qkbetan(w)\,\URiemInv_{\beta,\g}}\vspace{0.9ex}
\\
&=
(\xig)^{-w+J+k}\Nt^{A_1}\cdots\Nt^{A_k}\delta_J
\POneAB w\tilde{I}_E(\xig)^{w}~,\hspace{2.9ex}
\end{array}
\end{equation}
where $\tilde{I}_E=(n+2w-2)Y_E-\J X_E$.  We may use \nn{connids},
\nn{TracCurveOne}, \nn{TracCurveTwo}, \nn{tractorD},
\nn{NormalTractor}, and the rules of Figure~\ref{TrIProd} to expand
the rightmost expression in \nn{L7June14a}.  Then by polynomial
continuation in $w$, we find that
\begin{equation}\label{L7June14b}
\begin{array}{ll}
\lefteqn{-(n+2w-2k)(\xig)^{-J-k}\sum_{\beta\in
  C}\qkbetan(w)\,\URiemInv_{\beta,\g}
=}
\\
&
-(\xig)^{-w}\Nt^{A_1}\cdots\Nt^{A_k}\delta_J\POneAB\tilde{I}_E(\xig)^w
\hspace{5ex}
\end{array}
\end{equation}
for all $w\in\Reals$.  If we let $w=0$ in \nn{L7June14b}, then
\nn{DefineQ} follows.

The proof for the case $k\leq J$ is straightforward.
\end{proof}
%
%
\begin{theorem}\label{TFormulaTwo}
Let $I_E$ be as in Theorem~\ref{TFormula}, and for all
$K\in\Integers_{>0}$, view $Q_{\g}(\delta_K)$ and
$Q_{\g}(\ConFlatOp{K})$ as densities of weight $-K$.  Then for all
$K\in\Integers_{\geq 2}$,
\[
Q_{\g}(\delta_K)=-\Nt^{A_1}\cdots N^{A_{K-1}}\delta_1 D_{A_1}\cdots
D_{A_{K-2}}I_{A_{K-1}}
\]
and
\[
\big[
  \prod_{j=1}^{\left\lfloor\frac{K-1}{2}\right\rfloor}(n-2K+2j)
  \big]Q_{\g}(\ConFlatOp{K})=
-\Nt^{A_1}\cdots\Nt^{A_{k-1}}\delta_1D_{A_1}\cdots
D_{A_{K-2}}I_{A_{K-1}}~.
\]
In the above equations, if $K=2$, then $D_{A_1}\ldots D_{A_{K-2}}$
is the identity operator.  On the other hand, $Q_{\g}(\delta_1)=\mc$
and $Q_{\g}(\ConFlatOp{1})=\mc$.
\end{theorem}
\begin{proof}
The proof is similar to the proof of Theorem~\ref{TFormula}.
\end{proof}

One may use theorems~\ref{TFormula} and \ref{TFormulaTwo} to construct
symbolic formulae for $Q_{\g}(\deltaJk)$, $Q_{\g}(\delta_K)$, and
$Q_{\g}(\ConFlatOp{K})$.  In some cases, this requires a special
procedure.  For example, consider $Q_{\g}(\deltaJk)$ when $k>J$.  By
applying \nn{connids}, \nn{TracCurveOne}, \nn{TracCurveTwo},
\nn{tractorD}, \nn{NormalTractor}, and the rules in
Figure~\ref{TrIProd}, above, one may convert the right-hand side of
\nn{DefineQ} into a polynomial in $\Nv_a$, $R_{ab}{}^{c}{}_{d}$,
$\bg_{ab}$, $\bg^{ab}$, the modified mean curvature $\mmc$, and the
Levi-Civita connection $\nd$.  By commuting covariant derivatives and
applying standard tensor identities, one may (in principle) convert
this polynomial into an expression of the form $(n-2k)\psi$, where
$\psi$ is again a polynomial of the above type.  Thus
$(n-2k)Q_{\g}(\deltaJk)=(n-2k)\psi$ in all dimensions
$n\in\Dimensions$, where $\Dimensions$ is the set of all dimensions in
which $\deltaJk$ is defined.  Now let $K=J+k$, and note that
$Q_{\g}(\deltaJk)$ and $\psi$ define differential operators of order
zero.  These operators are elements of $\RationalCoeff_{K,K,0}$, and
they are equal in all dimensions $n\in\Dimensions\backslash\{2k\}$.
Thus $Q_{\g}(\deltaJk)=\psi$ in dimension $n=2k$, by
Proposition~\ref{ZeroAlln}.

We have used \nn{DefineQ} and the above procedure to find an explicit
symbolic formula for $Q_{\g}(\deltaJk)$ in a few specific low-order
cases.

Next we shall see that the behaviour of the curvatures of $\delta_K$
and $\deltaJk$, as defined by \nn{genQdef}, or equivalently by
theorems~\ref{TFormula} and \ref{TFormulaTwo}, varies strongly
according to dimension parity.  This suggests that we introduce some
terminology.
\begin{definition} \label{Tdef}
Let $\lowerw\in\Reals$, $K\in\Integers_{\geq 0}$,
and a Riemannian
hypersurface invariant $T^{\g}\in\ice[-\lowerw]$ be given.  Suppose
that $T^{\g}$ has a conformal transformation of the form
$
T^{\hatg}= T^{\g}+ P\Upsilon
$,
where $P:\ce[0]\rightarrow\ice[-\lowerw]$ is a conformally invariant
natural hypersurface differential operator of order and transverse
order $K$.  Then we shall say that $T^{\g}$ is a {\em hypersurface
  T-curvature} (or {\em transverse curvature}) of order $K$. We shall
call $(P,T)$ a {\em T-curvature pair}.
\end{definition}
This terminology is inspired by the Chang-Qing pair $(\CQOp,T^{CQ})$
which we discussed in Section~\ref{IntroSect}, above.  This pair is a
$T$-curvature pair in the sense of Definition~\ref{Tdef}. This is very
different from the Branson $Q$-curvature of the hypersurface, so we
shall say ``$Q$-curvature'' for quantities closer to that
behaviour. More precisely, we make the following definition.

\begin{definition}\label{QQdef}
Let $\lowerw\in\Reals$, $K\in\Integers_{>0}$,
and a Riemannian
hypersurface invariant $Q^{\g}\in\ice[-\lowerw]$ be given.  Suppose
that $Q^{\g}$ has a conformal transformation of the form
$
Q^{\hatg}= Q^{\g}+P\Upsilon
$,
where $P:\ce[0]\rightarrow\ice[-\lowerw]$ is a conformally
invariant natural hypersurface differential operator of order $K$ and
transverse order {\em less than} $K$.  Then we shall say that $Q^{\g}$
is a {\em hypersurface $Q$-curvature} of order $K$. We shall call
$(P,Q)$ a {\em $Q$-curvature pair}.
\end{definition}

It is clear that the $T$-curvatures of Definition~\ref{Tdef} are
recovered from the operators $\delta_K$, $\ConFlatOp{K}$, and
$\deltaJk$ in all cases in which these operators are defined and the
weight 0 is {\em not} in the respective exceptional list
$E(\delta_K)$, $E(\ConFlatOp{K})$, or $E(\deltaJk)$.
Theorem~\ref{T-thm-1}, above, thus follows from \nn{E0} and
Theorem~\ref{key}.

Let notation be as in Theorem~\ref{T-thm-1}, and let an even integer
$n_0\geq 4$ be given.  Then in dimension $n=n_0$, we should think of
the {\em critical pair} $(\delta_{(n_0-2)/2,\,n_0/2},\,T_{n_0-1})$ as
an analogue of the Chang-Qing $(\CQOp,T^{CQ})$-pair.  Note that
$T_1^{\g}$ is simply $H^{\g}$, the mean curvature of the hypersurface,
so the $T$-curvatures may also be thought of as (conformal
geometry-inspired) higher-order generalisations of the mean
curvature. This also suggests an interesting application of the
$T$-curvatures, which we describe as follows.
\begin{proposition}\label{H0}
Suppose that $n$ is even, $\sm$ is closed as a subset of $M$, and
$\sm$ and $M$ are orientable.  For all $i\in\Integers_{>0}$, let
$T_i^{\g}$ be as in Theorem~\ref{T-thm-1}.  Finally, let $\g\in\cc$ be
given.  Then for all $m\in\Integers_{>0}$, there is a metric
$\hatg\in\cc$ such that $\hatg$ induces $\ig$ and
\begin{equation}\label{vanishingT}
T_1^{\hatg}=T_2^{\hatg}=\cdots=T_m^{\hatg}=0\mbox{\ \ along $\sm$}~.
\end{equation}
Along $\sm$, the metric $\hatg$ is uniquely determined by $\ig$ to
order $m$.

The above results also hold if $\sm$ is not closed as a subset of $M$.
In this case, however, we only claim that the metric $\hatg$ is
defined on an open neighbourhood $O$ of $\sm$ in $M$.  Thus
$\hatg=e^{2\Upsilon}\g$ on $O$ for some $\Upsilon\in C^{\infty}(O)$.
The neighbourhood $O$ may depend on $m$.
\end{proposition}
\begin{proof}
Suppose first that $\sm$ is closed.  For any $p\in\sm$, we may define
local coordinates $(x^1,\ldots,x^n)$ on a neighbourhood of $p$ in $M$
such that $x^n$ is a local defining function for $\sm$.  Suppose that
$(y^1,\ldots,y^n)$ is any other local coordinate system, and suppose
that $y^n$ is also a local defining function for $\sm$.  We may assume
that $\partial y^n/\partial x^n>0$ at all points of $\sm$ at which
both coordinate systems are defined; this follows from the fact that
$M$ and $\sm$ are both orientable.  By using such coordinate systems
along with an appropriate partition of unity, one can show that for
some open neighbourhood $U$ of $\sm$, there is a defining function
$\dfn$ for $\sm$ on $U$.

For any $\Upsilon$, $\Upsilon'$, and $\Upsilon''$ in $\ce$, we will
let $\hatg=e^{2\Upsilon}\g$, $\hatg'=e^{2\Upsilon'}\g$, and
$\hatg''=e^{2\Upsilon''}\g$.  For all $i\in\Integers_{>0}$, let
$\psi_i$ denote the operator corresponding to $T^{\g}_{i}$ in
Theorem~\ref{T-thm-1}.

To establish the existence of the desired $\hatg$, we proceed by
induction on $m$.  Let $\indind\in\Integers_{\geq 0}$ be given.  If
$\indind=0$, let $\Upsilon\equiv 0$ on $M$.  Otherwise, suppose there
is an $\Upsilon\in\ce$ such that $\Upsilon$ vanishes along $\sm$, and
\nn{vanishingT} holds in the $m=\indind$ case.  Since $0\notin
E(\psi_{\indind+1})$, there is an open set $V\subseteq M$ such that
(1)~$\sm\subseteq V\subseteq U$, and
(2)~$\psi_{\indind+1}t^{\indind+1}$ is nonzero on $V$.  Choose any
$\phi\in\ce$ with support in $V$ such that $\phi\equiv 1$ on $\sm$,
and let $\Upsilon'=\Upsilon-t^{\indind+1}\phi
T^{\hatg}_{\indind+1}/(\psi_{\indind+1} t^{\indind+1})$.  Then
$T^{\hatg'}_1=T^{\hatg'}_2=\cdots=T^{\hatg'}_{\indind+1}=0$ along
$\sm$, by Proposition~\ref{genQtran}.

To establish uniqueness, it suffices to show that for every
$m\in\Integers_{>0}$ and every $\Upsilon',\Upsilon''\in\ce$, if
$\hatg=\hatg'$ and $\hatg=\hatg''$ induce $\ig$ and satisfy
\nn{vanishingT}, then $\Upsilon''-\Upsilon'$ vanishes to order $m$
along $\sm$.  We again proceed by induction on $m$.  Let
$\indind\in\Integers_{\geq 0}$ be given.  If $\indind>0$, suppose that
the desired statement holds in the case $m=\indind$.  Let $\Upsilon'$,
$\Upsilon''\in\ce$ be given, and suppose that $\hatg=\hatg'$ and
$\hatg=\hatg''$ induce $\ig$ and satisfy \nn{vanishingT} in the
$m=\indind+1$ case.  Then $\Upsilon''-\Upsilon'$ vanishes to order
$\indind$ along $\sm$.  But
$\hatg''=e^{2(\Upsilon''-\Upsilon')}\hatg'$.  Thus
$\psi_{\indind+1}(\Upsilon''-\Upsilon')=0$ along $\sm$, by
Proposition~\ref{genQtran}.  Recall that $0\notin
E(\psi_{\indind+1})$.  So if we work in local coordinates of the type
that we discussed above, a short argument shows that
$\Upsilon''-\Upsilon'$ vanishes to order $\indind+1$ along $\sm$.

If $\sm$ is \textit{not} closed, the proof is similar to the above
proof.
\end{proof}
\begin{remark}
In even dimensions, Proposition~\ref{H0} generalises a previous result
which has occasionally been used in the literature of general
relativity.  For even $n\geq 4$, Proposition~\ref{H0} gives a way of
normalising the conformal scale so that in the new scale, $\sm$ is
minimal (i.e.\ $H=0$) and also satisfies related higher-order
conditions.  By Theorem~\ref{Q-thm}, below, similar results hold in
the case in which $M$ is odd-dimensional, but the range of available
$T$-curvatures is smaller.  A statement of a zero-order version of
these results appears in Proposition~4.1 of \cite{Go-al}.

\end{remark}

In the proof of Theorem~\ref{Q-thm}, below, we will use the following
lemma to determine the order of $\delta_K:\ce[0]\rightarrow\ice[-K]$
in odd dimensions.
\begin{lemma}\label{FullOrder}
Let $K\in\Integers_{>0}$ and $w\in\Reals$ be given, and suppose that
$V$ is any section of $\tbPhi[w]$.  Then
\[
\delta_{K}V=
\big[
\prod_{i=1}^{K-1}(n+2w-2i)
\big]
\nd_{\Nv}^KV
+
\sum_{\beta\in A}C_{\beta}(n,w)P_{\beta}V
+
\lots~.
\]
Here $A$ is a finite index set, and for each $\beta\in A$,
$C_{\beta}(n,w)$ is a real polynomial in $n$ and $w$.  For each
$\beta\in A$, there is an $\ell\in\Integers_{>0}$ such that
$P_{\beta}=n^{i_1}\cdots n^{i_{K-2\ell}}O^{\beta}_{i_1\ldots
  i_{K-2\ell}}$, where $O^{\beta}$ is the composition of the operators
$\nd_{i_1}$, \ldots, $\nd_{i_{K-2\ell}}$, and $\ell$ copies of
$\Delta$, in some order.  Finally, $\lots$ can be given by a
symbolic formula which satisfies Hypotheses~\ref{FixingHyp}.
\end{lemma}
\begin{proof}
The lemma follows from \nn{connids}, \nn{tractorD},
\nn{NormalTractor}, Figure~\ref{TrIProd}, lemmas~\ref{NNabnkD} and
\ref{LTOTS}, and an inductive argument.
\end{proof}
\begin{theorem}\label{Q-thm}
Let $K\in\Integers_{>0}$ be given, and suppose that $\dim(M)=n$ is
odd.  Let $\psi:\ce[w]\rightarrow\ice[w-K]$ denote the operator family
$\delta_K$ of Lemma~\ref{basicversion} or the operator family
$\deltaJk$ of Theorem~\ref{key}.  (In the latter case $J+k$ must equal
$K$.)  If
\[
\frac{n+1}{2}\leq K\leq n-1~,
\]
then there is a canonical $Q$-curvature pair $(\psi,Q_K^{\g})$ of
order $K$.  Here $Q_K^{\g}:=Q_{\g}(\psi)$.  In particular, in the
critical-order case in which $K=n-1$, we obtain a $Q$-curvature pair
$(\psi,Q_{n-1}^\g)$.  If
\[
K\leq\frac{n-1}{2} \quad \mbox{or} \quad n \leq K~,
\]
then there is a $T$-curvature pair $(\psi,\TgK)$ of order $K$.  Here
$\TgK:=Q_{\g}(\psi)$.
\end{theorem}
\begin{proof}
Suppose first that $K\leq(n-1)/2$ or $n\leq K$.  Then $0\notin
E(\delta_K)$, by \nn{E0}.  Thus if $\psi=\deltaJk$, then $0\notin
E(\psi)$, by Theorem~\ref{key}.  Thus $\psi:\ce[0]\rightarrow\ice[-K]$
has order and transverse order $K$, and $Q_{\g}(\psi)$ is a
$T$-curvature.

Now suppose instead that $(n+1)/2\leq K\leq n-1$.  Then $0\in
E(\delta_{K})$, by \nn{E0}.  For the case in which $\psi=\deltaJk$,
note that $k-n/2\neq 0$, since $n$ is odd; thus $0\in E(\deltaJk)$, by
Theorem~\ref{key}.  So to complete the proof, it suffices to show that
$\psi:\ce[0]\rightarrow\ice[-K]$ has order $K$.

Let $V\in\ce[0]$ be given, and apply Lemma~\ref{FullOrder}.  By
commuting covariant derivatives, if necessary, we may write
\[
\delta_KV=\big[\prod_{i=1}^{K-1}(n-2i)\big]\nd_{\Nv}^KV
+\sum_{i=1}^{\lfloor
  K/2\rfloor}C_i'(n)\nd_{\Nv}^{K-2i}\Delta^iV+\lots~.
\]
Here $C_i'(n)$ is a real polynomial in $n$.  It follows that
$\delta_K$, as acting on $V$, has order $K$.  Now suppose that
$\psi=\deltaJk$.  Since $w=0$, it follows from Theorem~\ref{key} that
$\psi$ is a nonzero multiple of
$\Nt^{A_1}\cdots\Nt^{A_k}\delta_JP_{A_1\ldots A_k}$.  But by
\nn{ExpandOutB} and the discussion following \nn{ExpandOutB},
$\Nt^{A_1}\cdots\Nt^{A_k}\delta_JP_{A_1\ldots A_k}=\delta_K+\lots$.
Thus $\psi:\ce[0]\rightarrow\ice[-K]$ has order $K$.
\end{proof}

%
%

%
%
%
\section{Examples} \label{exs}

By using the constructions of Lemma~\ref{basicversion} and
theorems~\ref{cflatkey} and \ref{key}, one may
construct explicit formula for $\delta_K$, $\ConFlatOp{K}$, and
$\deltaJk$ of the type described in Hypotheses~\ref{NaturalHyp}.  One
may also use theorems~\ref{TFormula} and \ref{TFormulaTwo} to
construct similar explicit formulae for $Q_{\g}(\delta_K)$,
$Q_{\g}(\ConFlatOp{K})$, and $Q_{\g}(\deltaJk)$.  The construction of
such formulae for $\deltaJk$ and $Q_{\g}(\deltaJk)$ uses the algorithm
of \cite{GP-CMP} for constructing tractor formulae for the GJMS
operators $P_{2k}$.  We have constructed explicit formulae of
the above types in a few low-order cases.  In this section, we discuss
these constructions and give some of the resulting explicit formulae.
In some cases, we have incorporated the trace-free second fundamental
form $\tfsff{ab}$ into the formulae.  In constructing the explicit
formulae, we used the results described in Section~\ref{back}, above.
In many cases, we also used \textit{Mathematica}, together with
J. Lee's \textit{Ricci} package \cite{Lee,Wolfram}.

We will need specific tractor formulae for two of the operator
families of Proposition~\ref{GJMSp}.
By \cite{GP-CMP}, we
have $P_{AB}=D_{A}D_{B}$ and
\begin{equation}\label{PABC}
P_{ABC}=D_AD_BD_C-\frac{2}{n-4}X_AW_B{}^E{}_C{}^FD_ED_F~.
\end{equation}

\subsection{Second-Order Operator}\label{ExamplesOrder2}
The operator family
$\delta_2:\tbPhi[w]\rightarrow\left.\tbPhi[w-2]\right|_{\sm}$ of
Lemma~\ref{basicversion} is given by \nn{d2}.  By
Definition~\ref{generalQ},
$$
Q_{\g}(\delta_2)=\J+(n-2)\mc^2-(n-2)\Nv^a\Nv^b\Rho_{ab}~.
$$
From \nn{E0}, it follows that $E(\delta_{2})=\{(3-n)/2\}$, so for all
$n\geq 4$, $Q_{\g}(\delta_2)$ is a hypersurface
$T$-curvature in the sense of Definition~\ref{Tdef}.  In the $n=3$
case, $Q_{\g}(\delta_2)$ is a $Q$-curvature in the sense of
Definition~\ref{QQdef}.

In \cite{Grant}, D.H. Grant discusses a related family of conformally
invariant second-order hypersurface differential operators
$\iD^A\Pi_A{}^{B}D_B:\ce[w]\rightarrow\ice[w-2]$.  For any
$w\in\Reals$ and any $f\in\ce[w]$,
$$
\begin{array}{ll}
\lefteqn{\iD^A\Pi_{A}{}^BD_Bf=
(n+w-3)\left(\rule{0mm}{2.5ex}\right.(n+2w-2)\ibop f}
\\
&\hspace*{5mm}
+(n+2w-3)(\ (n+2w-2)(\mc\Nv^a\nd_af-\frac{w}{2}\mc^2f)\ -\ \bop f\ )
\left.\rule{0mm}{2.5ex}\right)~,
\end{array}
$$
by \cite{Grant}.  A computation shows that
\[
(n+w-3)\delta_2 f
+\iD^A\Pi_A{}^{B}D_Bf
=
-\frac{w(n+w-3)(n+2w-2)}{2(n-2)}\tfsff{ab}\tfsffup{ab}f~.
\]
If we divide by $n+w-3$, polynomial continuation shows that
\[
\delta_2 f=-\ibop f+\frac{n-3}{4(n-2)}\tfsff{ab}\tfsffup{ab}f~,
\hspace{2ex}\mbox{for\ }f\in\ce\left[\frac{3-n}{2}\right]~.
\]
In this case, $\ibop$ is the intrinsic Yamabe operator.%
%
%
%
%
%
\subsection{Third-Order Operator}\label{ExamplesOrder3}
Consider $\deltaOneTwo:\ce[w]\rightarrow\ice[w-3]$.  This family of
operators is well-defined in all ambient dimensions $n\geq 3$, and
from Theorem~\ref{key}, it follows that $E(\deltaOneTwo)=\{(5-n)/2\}$.
By Theorem~\ref{key}, $(n+2w-4)\deltaOneTwo=\Nt^A\Nt^B\delta D_A D_B$.
For $n\geq 4$, symbolic computations show that $\deltaOneTwo$ is given
by the formula in Figure~\ref{ThirdOrderOperator}.%
\newc{\OneTwoSpace}{\rule{0mm}{4mm}}
%
%
\begin{figure}
\[
\begin{array}{ll}
\lefteqn{(n+2w-5)\delta\Box f}
\\
\lefteqn{\textstyle-(n+2w-2)\left(\OneTwoSpace\right.
\ibop\delta f
-\frac{2}{3(n-3)}\LDD f
\vspace{0mm}
%
+\frac{3n-5}{4(n-2)}\tfsff{ab}\tfsffup{ab}\cRobin f}
\\&
-\frac{n-5}{2(n-3)}\tfsff{ab}\sffo^{a}{}_{c}\tfsffup{bc}f
-\frac{n-5}{2(n-3)}\Nv^a\Nv^b\tfsff{cd}\C_a{}^c{}_b{}^d f
\left.\OneTwoSpace\right)
\vspace{1mm}
\\
\lefteqn{\textstyle+(n+2w-2)(n+2w-5)\left(\OneTwoSpace\right.
-\mc\Nv^a\Nv^b\nd_a\nd_b f
-\frac{2(n+2w-1)}{3}\tfsffup{ab}\nd_a\nd_b f
}
\\
&
+\frac{7-4n-12w+4nw+8w^2}{3}(\ilc_a\mc)\nd^a f
+\frac{3(w-1)}{2}\mc^2\Nv^a\nd_a f
\\
&
+\frac{-7+4n+12w-4nw-8w^2}{3}\Nv^a\Nv^b\Nv^c\Rho_{ab}\nd_c f
+\frac{2(5-2n-6w+2nw+4w^2)}{3}\Nv^a\Rho_a{}^b\nd_bf
\\
&-\frac{1}{4(n-2)}\tfsff{ab}\tfsffup{ab}\Nv^c\nd_cf
-\frac{w(5-2n-6w+2nw+4w^2)}{3(n-3)}(\ila\mc)f
-\frac{w(w-1)}{2}\mc^3 f
\\
&
+\frac{w(9+2n-2n^2-6nw+2n^2w+4nw^2)}{3(n-3)}\mc\Nv^a\Nv^b\Rho_{ab}f
\\
&
+\frac{w(5-2n-6w+2nw+4w^2)}{3(n-3)}\Nv^a\Nv^b\Nv^c(\nd_a\Rho_{bc})f
-\frac{w(5-2n-6w+2nw+4w^2)}{3(n-3)}\mc\J f
\\
&
-\frac{w(5-2n-6w+2nw+4w^2)}{3(n-3)}\Nv^a(\nd_a\J)f
-\frac{2w(n+2w-2)(n+2w-4)}{3(n-3)}\tfsffup{ab}\Rho_{ab}f
\\
&
-\frac{w(-31+33n-8n^2+48w-40nw+8n^2w-32w^2+16nw^2)}{12(n-2)(n-3)}
\mc\tfsff{ab}\tfsffup{ab}f
\\
&
+\frac{(2w+1)(9-3n-8w+2nw+4w^2)}{6(n-3)^2}
\tfsff{ab}\sffo^{a}{}_{c}\sffo^{bc}f
\\
&
+\frac{(2w+1)(9-3n-8w+2nw+4w^2)}{6(n-3)^2}
\Nv^a\Nv^b\tfsff{cd}\C_a{}^c{}_b{}^df
\left.\OneTwoSpace\right)
\end{array}
\]
\caption{$\deltaOneTwo f$ for dimensions $n\geq 4$ and $f\in\ce[w]$}
\label{ThirdOrderOperator}
\end{figure}
%
%
Here $\LDD$ is the conformally invariant operator $\tsff^{AB}D_AD_B$;
additional symbolic computations show that for all $w\in\Reals$ and all
$f\in\ce[w]$,
\[
\begin{array}{ll}
\lefteqn{\LDD f=(n+2w-2)(n+2w-4)\left(\rule{0mm}{4mm}\right.
(n-3)\sffo^{ab}\nd_a\nd_b f}
\vspace{1mm}
\\
&
-2(n-3)(w-1)(\ilc_a\mc)\nd^a f
+w(w-1)(\ila\mc)f
\vspace{1mm}
\\
&
-n(w-1)w\mc\Nv^a\Nv^b\Rho_{ab}f
-w(w-1)\Nv^a\Nv^b\Nv^c(\nd_a\Rho_{bc})f
\vspace{1mm}
\\
&
+w(w-1)\mc\J f
+2(n-3)(w-1)\Nv^a\Nv^b\Nv^c\Rho_{ab}\nd_c f
\vspace{1mm}
\\
&
+w(w-1)\Nv^a(\nd_a\J)f
-2(n-3)(w-1)\Nv_a\Rho^{ab}\nd_bf
\vspace{1mm}
\\
&
+w(n+2w-5)\Rho_{ab}\sffo^{ab}f
+w(w-1)\mc\sffo_{ab}\sffo^{ab}f
\vspace{1mm}
\\
&
-\frac{w(w-1)}{n-3}\sffo_{ab}\sffo^a{}_c\sffo^{bc}f
-\frac{w(w-1)}{n-3}\Nv^a\Nv^b\sffo_{cd}\C_a{}^c{}_b{}^df
\left.\rule{0mm}{4mm}\right)~.
\end{array}
\]
Figure~\ref{ThirdOrderOperator} shows that for $f\in\ce[1-n/2]$, we have
$\deltaOneTwo f=-3\cRobin\Box f$.  Thus
$\deltaOneTwo:\ce[1-n/2]\rightarrow\ice[-2-n/2]$ is
the composition of two conformally invariant operators.
Similarly,
Figure~\ref{ThirdOrderOperator} shows that for $f\in\ce[5/2-n/2]$,
\[
\begin{array}{ll}
\lefteqn{\deltaOneTwo f=\textstyle-3\ibop\delta f
+\frac{2}{n-3}\LDD f
-\frac{3(3n-5)}{4(n-2)}\tfsff{ab}\tfsffup{ab}\cRobin f}
\\
&+\frac{3(n-5)}{2(n-3)}\tfsff{ab}\sffo^{a}{}_{c}\tfsffup{bc}f
 +\frac{3(n-5)}{2(n-3)}\Nv^a\Nv^b\tfsff{cd}\C_a{}^c{}_b{}^d f~.
\end{array}
\]
Here $\ibop$ acts on a density of weight $1-(\sdim)/2$.  Thus $\ibop$
is conformally invariant in this case, and we have again expressed
$\deltaOneTwo$ in terms of conformally invariant operators.
We gave a symbolic formula for $Q_{\g}(\delta_{1,2})$ in \nn{OneTwoQ}.

In Section~\ref{IntroSect}, above, we
discussed the operators $\delta^{G}_{3}$ of \cite{Grant}.  These
operators exist in all dimensions $n\geq 4$ and map $\ce[(4-n)/2]$ to
$\ice[(-2-n)/2]$.  The leading term of $\delta^G_{3}$ equals the
leading term of $\delta_{1,2}:\ce[(4-n)/2]\rightarrow\ice[(-2-n)/2]$,
up to a nonzero scale.

Now suppose that $M$ is a compact 4-manifold with boundary $\sm$, and
suppose that the unit normal vector $\Nv^a$ points inward.  Symbolic
computations using \cite{Lee} and \cite{Wolfram} show that for all
$f\in\ce[0]$, $\delta_{1,2}$ and the third-order boundary operator
$\CQOp$ of \cite{CQ} satisfy the following:
\[
\CQOp
f=-\frac{1}{2}\,\delta_{1,2}f
-\tsff^{AB}\iD_B\iD_Af+\tfsff{ab}\tfsffup{ab}\delta f~.
\]

\subsection{Fourth-Order Operator}\label{ExamplesOrder4}

Consider $\deltaOneThree:\ce[w]\rightarrow\ice[w-4]$.  By
Theorem~\ref{key}, together with \nn{PABC},
\[
(n+2w-6)\deltaOneThree f=
\Nt^{A}\Nt^{B}\Nt^{C}\cRobin
(
D_AD_BD_Cf-\frac{2}{n-4}X_A W_B{}^{E}{}_{C}{}^FD_ED_Ff
)
\]
for all $f\in\ce[w]$.  For $n\geq 5$, symbolic computations show that
$\deltaOneThree$ has the form given in Figure~\ref{GeneralWeight}.
%
%
\newc{\JuneTenSpace}{\vspace{1mm}}%
\begin{figure}
\[
\begin{array}{ll}
%
%
\lefteqn{\textstyle\frac{(n+2w-2)(n+2w-5)(n+2w-7)}{2}(-Y^A\Box D_A f)}
&
\hspace{110mm}
\JuneTenSpace
\\
%
%
\lefteqn{\textstyle
+\frac{(n+2w-4)(n+2w-5)(n+2w-7)}{10}\,\delta_2\Box f}
\\
%
%
\lefteqn{\textstyle
-\frac{(n+2w-2)(n+2w-4)(n+2w-7)}{2}
\left(\rule{0mm}{3.5mm}\right.%
-\YS^A\ibop\iD_A f+\frac{4}{3(n-3)}\tsff^{AB}\cRobin D_A D_Bf
}
\\&
-\frac{2(n-4)}{3(n-3)}\tfsff{ab}\tfsffup{ab}\delta_2 f
-\frac{n-7}{2(n-3)^2(n-2)}\tsff^{AB}\ibop\tsff_{AB}f
\\&
-\frac{2(n-5)}{3(n-3)(n-4)}\Nt^C\Nt^D\tracW^A{}_{C}{}^B{}_D D_A D_B f
+\frac{2(n-1)}{3(n-3)^3}\tsff^{A}{}_{C}\tsff^{CB}D_{A}D_{B} f
\\&
+\frac{n-5}{n-3}\tfsff{ac}\tfsffgen^a{}_d\tfsffup{cd}\cRobin f
+\frac{n-5}{(n-4)(n-3)^2}\Nt^A\Nt^B\tsff_{CD}\tracW_{A}{}^C{}_B{}^D\cRobin
f
\\&
+
\frac{4}{3(n-4)(n-3)}\Nt^B\tsff_{CD}\tracW^{AC}{}_{B}{}^D D_A f
-\frac{1}{2(n-3)^2(n-2)}K(L^{AB},L^{CD})f
\\&
+\frac{n-5}{(n-3)^2}\tfsff{ab}\tfsff{cd}\tfsffup{ac}\tfsffup{bd}f
+\frac{n-5}{(n-3)^2}
\Nv^a\Nv^b\Nv^c\Nv^d\C_{aebi}\C_c{}^e{}_d{}^i f
\\&
+\frac{(n-5)(n^3-15n^2+51n-53)}{16(n-3)^2(n-2)^2}
\tfsff{ab}\tfsffup{ab}\tfsff{cd}\tfsffup{cd}f
+\frac{2(n-5)}{(n-3)^2}\Nv^{a}\Nv^{b}\tfsff{cd}\tfsffgen^{c}{}_{e}
\C_{a}{}^{d}{}_{b}{}^{e}f%
\left.\rule{0mm}{3.5mm}\right)
\\
%
%
\lefteqn{\textstyle
-\frac{(n+2w-2)(n+2w-4)(n+2w-5)}{10}\left(\rule{0mm}{3.5mm}\right.%
\overline{\Box}\,\delta_2f
-\frac{4}{3(n-3)}\tsff^{AB}\cRobin D_AD_Bf}
\\&
+\frac{7n-13}{4(n-2)}\tfsff{ab}\tfsffup{ab}\cRobin_2f
+\frac{2}{3(n-4)}\Nt^C\Nt^D\tracW^{A}{}_{C}{}^{B}{}_{D}D_AD_Bf
\\&
-\frac{2}{3(n-3)^2}\tsff^{A}{}_{C}\tsff^{CB}D_A D_B f
-\frac{4}{(n-4)(n-3)}\Nt^A\tsff_{BC}\tracW^{EB}{}_{A}{}^{C}D_Ef
\\&
-\frac{5(n-5)}{n-3}\tfsff{bc}\tfsffgen^{b}{}_{d}\tfsffup{cd}\cRobin f
-\frac{5(n-5)}{(n-4)(n-3)^2}\Nt^A\Nt^{B}\tsff_{CD}\tracW_{A}{}^{C}{}_{B}{}^{D}\cRobin
f
\left.\rule{0mm}{3.5mm}\right)
%
%
\\
\lefteqn{\textstyle+
(n+2w-2)(n+2w-4)(n+2w-5)(n+2w-7)\left(\rule{0mm}{3.5mm}\right.}
\\&
-\frac{1}{10}\Nv^a\Nv^b\Nv^c\Nv^d\nd_a\nd_b\nd_c\nd_d f
+
\lots
\left.\rule{0mm}{3.5mm}\right)
\end{array}
\]
\caption{$\deltaOneThree f$ for dimensions $n\geq 5$ and $f\in\ce[w]$}
\label{GeneralWeight}
\end{figure}
%
%
In this figure, $K(L^{AB},L^{CD})$ is a scalar invariant which we
define as follows.  First, note that by \nn{tractorD} and
Figure~\ref{TrIProd}, $(\iD_AL_{BC})\iD^{A}L^{BC}$ contains a factor
$(n-5)$.
Divide the symbolic formula for $\left( \iD_AL_{BC} \right)
\iD^AL^{BC}$ by $(n-5)$, and let $K(L^{AB},L^{CD})$ be defined by the
resulting symbolic formula.  The quantity defined in this way is
conformally invariant in all dimensions $n\geq 5$.  In dimension~5,
this follows from dimensional continuation.

Inspection of Figure~\ref{GeneralWeight} reveals the particular form
that $\deltaOneThree f$ takes when the weight $w$, of $f$,
is equal to $2-n/2$, $1-n/2$, $5/2-n/2$, or $7/2-n/2$.  In each of
these cases, the formula in the figure expresses $\deltaOneThree$ in
terms of conformally invariant operators.  To see this, note first
that by \cite{GP-CMP}, the (conformally invariant) Paneitz operator
$P_4:\ce[2-n/2]\rightarrow\ce[-2-n/2]$ satisfies $P_4=-Y^A\Box D_A$.
Thus if $w=2-n/2$ in Figure~\ref{GeneralWeight}, then
$\delta_{1,3}=3P_4$.  Similarly, if $w=5/2-n/2$, then
$-\YS^A\ibop\iD_A$, in Figure~\ref{GeneralWeight},  is the
intrinsic Paneitz operator.  On the other hand,
$\Box:\tbPhi[1-n/2]\rightarrow\tbPhi[-1-n/2]$ is conformally 
invariant.  The operators $\tsff^{AB}\delta D_AD_B$,
$\Nt^C\Nt^D\tracW^A{}_C{}^B{}_DD_AD_B$, and
$\tsff^{A}{}_{C}\tsff^{CB}D_AD_B$ in the figure have order less than
four.

The $Q$-type curvature associated to $\deltaOneThree$ is as follows:
$$
\begin{array}{ll}
\lefteqn{Q_{\g}(\delta_{1,3})=
6(n-4)(n-2)\mc^4
-36(n-4)(n-2)\mc^2\Nv^a\Nv^b\Rho_{ab}}
\\
&-12(n-4)(n-2)\mc\Nv^a\Nv^b\Nv^c\nd_a\Rho_{bc}
\\
&
-(n-4)(n-2)\Nv^a\Nv^b\Nv^c\Nv^d\nd_a\nd_b\Rho_{cd}
\\
&+(n-2)\Nv^a\Nv^b\De\Rho_{ab}
+36(n-4)\mc^2\J
-3\De\J
+6\J^2
\\
&
+36(n-4)\mc\Nv^a\nd_a\J
+2(n-6)(n-2)\Nv^a\Nv^b\Rho_{ac}\Rho_b{}^c
\\
&
+6(n-4)(n-2)\Nv^a\Nv^b\Nv^c\Nv^d\Rho_{ab}\Rho_{cd}
-12(n-4)\Nv^a\Nv^b\Rho_{ab}\J
\\
&
+(5n-22)\Nv^a\Nv^b\nd_a\nd_b\J
-2(n-2)\Rho_{ab}\Rho^{ab}
+2(n-2)\Nv^a\Nv^b\Rho_{cd}\C_a{}^c{}_b{}^d~.
\end{array}
$$

\subsection{Fifth-Order Operator}

Finally, consider $\deltaTwoThree:\ce[w]\rightarrow\ice[w-5]$.
By \nn{PABC}, we see that
\begin{equation}\label{TwoThreeTractor}
\begin{array}{ll}
\lefteqn{(n+2w-6)\deltaTwoThree f=}\vspace{1mm}
\\
&
\Nt^{A}\Nt^{B}\Nt^{C}\Nt^{E}\cRobin D_E
(
D_AD_BD_Cf-\frac{2}{n-4}X_A W_B{}^{F}{}_{C}{}^GD_FD_Gf
)~.
\end{array}
\end{equation}
One may use \nn{TwoThreeTractor} to express $\deltaTwoThree f$, in
terms of $\nd$, $\sffo_{ab}$, $\mc$, the normal vector $\Nv$, the
curvature of a representative $\g$ of the conformal structure, and
$f$.  By setting $w$ equal to $2-n/2$ in the resulting expanded
formula, one can show that $\deltaTwoThree=-60\delta P_4$, where $P_4$
is the (conformally invariant) Paneitz operator.  Similarly, for
$f\in\ce[1-n/2]$, we have $\deltaTwoThree f = -5\delta_3\Box f$; note
that in this context, $\Box$ is conformally invariant.  An expanded
formula for $Q_{\g}(\delta_{2,3})$ appears in Figure~\ref{Q23}.
%
%
\begin{figure}
\[
\begin{array}{l}
24(-8+n)(-4+n)(-2+n)\mc^5
\\
-240(-8+n)(-4+n)(-2+n)\mc^3\Nv{}^{a}\Nv{}^{b}\Rho{}_{a}{}_{b}
\\
-120(-8+n)(-4+n)(-2+n)\mc^2\Nv{}^{a}\Nv{}^{b}\Nv{}^{c}
\nd_a\Rho{}_{b}{}_{c}
\\
-20(-8+n)(-4+n)(-2+n)\mc\Nv{}^{a}
\Nv{}^{b}\Nv{}^{c}\Nv{}^{d}\nd_a\nd_b\Rho{}_{cd}
\\
+20(-7+n)(-2+n)\mc\Nv{}^{a}\Nv{}^{b}\Delta\Rho{}_{a}{}_{b}
\\
-(-8+n)(-4+n)(-2+n)\Nv{}^{a}\Nv{}^{b}\Nv{}^{c}\Nv{}^{d}
\Nv{}^{e}\nd_a\nd_b\nd_c\Rho{}_{d}{}_{e}
\\
+(-2+n)(-22+3n)\Nv{}^{a}\Nv{}^{b}\Nv{}^{c}
\De\nd_a\Rho{}_{b}{}_{c}
\\
+(-14+n)(-2+n)(-14+3n)\Nv{}^{a}\Nv{}^{b}\Nv{}^{c}
\Rho{}_{ad}\nd^d\Rho{}_{bc}
\\
+20(-2+n)(96-29n+2n^2)\mc\Nv{}^{a}\Nv{}^{b}\Rho{}_{a}{}_{c}
\Rho{}_{b}{}^{c}
\\
+2(-8+n)(-2+n)(-32+3n)\Nv{}^{a}\Nv{}^{b}
\Nv{}^{c}\Rho{}_{a}{}_{d}\nd_c\Rho{}_{b}{}^{d}
\\
+120(-8+n)(-4+n)(-2+n)\mc\Nv{}^{a}\Nv{}^{b}\Nv{}^{c}
\Nv{}^{d}\Rho{}_{a}{}_{b}\Rho{}_{cd}
\\
+24(-8+n)(-4+n)(-2+n)\Nv{}^{a}\Nv{}^{b}\Nv{}^{c}\Nv{}^{d}
\Nv{}^{e}\Rho{}_{a}{}_{b}\nd_c\Rho{}_{de}
\\
-20(-2+n)(-17+2n)\mc\Rho{}_{a}{}_{b}\Rho{}^{a}{}^{b}
+8(-9+n)(-2+n)\Nv{}^{a}\Rho{}^{bc}\nd_b\Rho_{ac}
\\
-20(-9+n)(-2+n)\Nv{}^{a}\Rho{}_{b}{}_{c}\nd_{a}\Rho{}^{b}{}^{c}
+240(-8+n)(-4+n)\mc^3\J
\\
-240(-8+n)(-4+n)\mc\Nv{}^{a}\Nv{}^{b}\Rho{}_{a}{}_{b}\J
\\
+(-684+268n-23n^2)\Nv{}^{a}\Nv{}^{b}\Nv{}^{c}
\J\nd_a\Rho{}_{b}{}_{c}
+120(-8+n)\mc\J^2
\\
+360(-8+n)(-4+n)\mc^2\Nv{}^{a}\nd_a\J
+75(-8+n)\Nv{}^{a}\J\nd_a\J
\\
-2(1236-452n+37n^2)\Nv{}^{a}\Nv{}^{b}\Nv{}^{c}\Rho{}_{a}{}_{b}
\nd_c\J
\\
+(-312+34n+n^2)\Nv{}^{a}\Rho{}_{a}{}_{b}\nd^b\J
+20(178-63n+5n^2)\mc\Nv{}^{a}\Nv{}^{b}\nd_a\nd_b\J
\\
-60(-8+n)\mc\Delta\J
+(276-92n+7n^2)\Nv{}^{a}\Nv{}^{b}\Nv{}^{c}
\nd_a\nd_b\nd_c\J
\\
-15(-8+n)\Nv{}^{a}\De\nd_a\J
+20(-10+n)(-2+n)\mc\Nv{}^{a}
\Nv{}^{b}\Rho{}_{cd}\C{}_{a}{}^{c}{}_{b}{}^{d}
\\
+4(-9+n)(-2+n)\Nv{}^{a}\Nv{}^{b}\Nv{}^{c}\Rho{}_{de}
\nd_a\C{}_{b}{}^{d}{}_{c}{}^{e}
\\
+8(-9+n)(-2+n)\Nv{}^{a}\Nv{}^{b}\Nv{}^{c}
\C{}_{b}{}^{d}{}_{c}{}^{e}\nd_e\Rho{}_{ad}
\\
+4(-9+n)(-2+n)\Nv{}^{a}\Nv{}^{b}\Nv{}^{c}
\C{}_{b}{}^{d}{}_{c}{}^{e}\nd_a\Rho{}_{de}
\end{array}
\]
\caption{$Q_{\g}(\delta_{2,3})$}
\label{Q23}
\end{figure}
%
%
%

%
%
%
\section{Final remarks}\label{FR}
Here we mention some related points that we find very interesting.  We
only touch on these points however, as a full treatment lies beyond
the scope of the current work.

\subsection{Poincar\'e-Einstein manifolds} \label{PEm}

We recall the well-known notions of conformally compact and
Poincar\'e-Einstein manifolds. For simplicity, we assume here that all
relevant structures are orientable.
\begin{definition}\label{PE-Def}
Let a compact manifold $M$ of dimension $n$ with boundary $\sm$ be
given.  Also let a Riemannian metric $\gplus$ on the interior,
$\Mplus$, of $M$ be given, and suppose that for some nonnegative
defining function $x$ for $\sm$, the metric $\g=x^2\gplus$ extends to
a Riemannian metric on $M$.  Then we say that $(\Mplus,\gplus)$ is
\textit{conformally compact}.  The restriction of $\g$ to $T\sm$
determines a conformal structure $\ic$ on $\sm$, and we say that
$(\sm,\ic)$ is the \textit{conformal infinity} of $\Mplus$.  Finally,
suppose that the metric $\gplus$ on $\Mplus$ is Einstein and that
$\Rc^{\gplus}=-(n-1)\gplus$.  Then we say that $M$ is a
\textit{Poincar\'{e}-Einstein} manifold.
\end{definition}
In fact, we will assume, as above, that $M$, the defining function
$x$, and $\g$ are extended to a collar neighbourhood of $\sm$.  The
assumption that $x$ is nonnegative will only apply to points of the
original manifold.

Conformally compact manifolds have, in particular, the structure of a
conformal manifold with boundary. Thus all the results from the
earlier sections apply to this setting. It is worth noting, however,
that in the special case of a Poincar\'e-Einstein manifold, rather
strong simplifications arise. (This is also true for manifolds which
are asymptotically Einstein to sufficient order. We leave the details
of this to the reader, but some related results are discussed in
\cite{CG}.)
This stems from the following result.
\begin{lemma}\label{ILemma}
Let $\sm$ be the conformal infinity of a Poincar\'{e}-Einstein
manifold $M$, and let $x$ be a nonnegative defining function for $\sm$
as in Definition~\ref{PE-Def}.  Let $\sigma=x\xig$, and let
$I^A=\frac{1}{n}D^A\sigma$.  Then $I^A|_{\sm}=\Nt^{A}|_{\sm}$,
and $I^A$ has the following properties on $M$: (1)~$I^A$ is parallel,
(2)~$I_AI^A=1$, and (3)~the contraction of $I^A$ with any index of
$W_{BCDE}$ is zero.
\end{lemma}
\begin{proof}
The fact that $I^A|_{\sm}=N^A|_{\sm}$ and properties~(1) and (2) are
proved in \cite{Go-Conf-D-N}.  (See also \cite{CG} for a recent
treatment.)  Note that \cite{Go-Conf-D-N} and \cite{CG} use $d$ and
$n$ to denote the dimensions of $M$ and $\sm$, respectively.
By \cite{GP-CMP}, $W_{BCDE}$ has Weyl tensor symmetries.  Thus
$W_{BCDE}=-W_{BCED}$ and $W_{BCDE}=W_{DEBC}$.  Since $I^A$ is parallel,
property~(3) follows from \nn{TracCurv} and \nn{TracCurveTwo}.
\end{proof}

Moreover since $I^A|_\Sigma=N^A$ is parallel along $\Sigma$, it follows
that $\Sigma$ is totally umbilic \cite{Go-al}. The Einstein condition,
i.e.\ the fact that $I^A$ is parallel, also forces other hypersurface
invariants to vanish, cf. \cite{CG,GW-willmore}. We should thus expect
the formulae
for the boundary operators and associated curvatures to
simplify. Indeed, we have the following:
\begin{proposition}\label{PELemma}
Let $\sm$ be the conformal infinity of a Poincar\'{e}-Einstein
manifold $M$.  Let $\g$ and $n$ be as in Definition~\ref{PE-Def}, and
let $\cc$ denote the conformal class associated to $\g$.  Let
$j,\ k\in\Integers_{>0}$ be given, and suppose that $n$ and $\Mcc$
satisfy Condition~\ref{nkcCondition}. Finally, let $P$ be as in
Proposition~\ref{GJMSp}.  Then
\begin{equation}\label{PtoD}
\Nt^{A_1}\cdots\Nt^{A_k}\delta_j P_{A_1\cdots A_k}=
\Nt^{A_1}\cdots\Nt^{A_k}\delta_j D_{A_1}\ldots D_{A_k}
\end{equation}
along $\sm$.
\end{proposition}
\begin{proof}
From Lemma~\ref{ILemma}, it follows that
$$
\Nt^{A_1}\cdots\Nt^{A_k}\delta_j P_{A_1\cdots A_k}=
\delta_j I^{A_1}\cdots I^{A_{k}}P_{A_1\cdots A_{k}}
$$
along $\sm$.  But by Proposition~\ref{GJMSTractor}, in every term of
the tractor formula for $\LowTrac{2}{k}$, at least one of the indices
$A_2$, \ldots, $A_{k}$ appears on a $W$.  Thus \nn{PtoD} follows from
\nn{DefineGJMSp} and Lemma \ref{ILemma}.
\end{proof}

The simplification expressed in \nn{PtoD} enables, in this setting,
further refinements beyond Theorem \ref{key}. This will be taken up
elsewhere. Meanwhile, inspired by Proposition~\ref{PELemma}, we now
consider the very simple operator
\begin{equation}\label{BigIDOp}
\Nt^{A_1}\cdots\Nt^{A_k} D_{A_1}\ldots D_{A_k} =I^{A_1}\cdots I^{A_k}
D_{A_1}\ldots D_{A_k}
\end{equation}
along the boundary $\Sigma$ of a Poincar\'e-Einstein manifold. The
right-hand side of \nn{BigIDOp} is defined not only along $\sm$, but
on all of $M$ as well.  Since $I$ is parallel, the right-hand side of
\nn{BigIDOp} is simply $(I\cdot D)^k$.  Here $I\cdot D = I^AD_A$.
This idea was studied in \cite{Go-Powers} (to produce GJMS operators
of the interior manifold) and in
e.g.\ \cite{GoW-boundary,GW-willmore}. These operators have a number
of applications.  For example, for any $m\in\Integers_{>0}$,
\begin{equation}\label{tan-GJMS}
(I\cdot D)^{2m} :\ce\left[m-\frac{n-1}{2}\right] \to
\ce\left[-m-\frac{n-1}{2}\right]~,
\end{equation}
acts tangentially along $\sm$ \cite[Theorem 4.1]{GoW-boundary}, and for
suitable $m$ recovers nonzero multiples of the order-$2m$ GJMS operators
of the boundary \cite[Theorem 4.5]{GoW-boundary}.  Here \textit{tangential} is as
defined in \cite{GoW-boundary}. (In comparing with \cite{GoW-boundary}, note that with $I$ parallel it follows that $I^2:=I^AI_A$ is constant, and in the Poincar\'e-Einstein case this constant is non-zero.) From this one can, in an obvious way,
find ``factorisations'' of the $I\cdot D$ powers at appropriate
weights. For
example, along $\sm$ and for $k\in\Integers_{>0}$,
$$
(I\cdot D)^{2m+k}=(I\cdot D)^{2m}\circ (I\cdot D)^{k}:
\ce\left[m+k-\frac{n-1}{2}\right] \to \ce\left[-m-\frac{n-1}{2}\right]
$$
evidently factors as a composition of a GJMS operator of $\sm$ applied
to the restriction to $\Sigma$ of $(I\cdot D)^{k}$ acting on
$\ce[m+k-(n-1)/2]$.  Factorisations along these lines have been sought
in the work \cite{JB} of Juhl.  Again, further examples are beyond the
scope of the current work, but we believe that these observations are
of interest.  In particular we believe that, in the setting of
Poincar\'e-Einstein (or suitably asymptotically Poincar\'e-Einstein)
manifolds, one may use these observations, together with the main
techniques of the present paper, to construct simpler families of
boundary operators that solve the problems \ref{GenProblem} (and
\ref{CQprob}), and indeed yield rather stronger results. Again this
will be taken up elsewhere.

\subsection{Juhl's conjectures}\label{JC}
Several of the results in the sections above solve problems closely
related to (other) directions and conjectures in Juhl's monograph
\cite{JB}. For example, the $T$-curvature pairs of
Theorem~\ref{T-thm-1}, above, satisfy the ``fundamental identity''
that expression (1.10.4) in \cite{JB} would satisfy if the residue
described there were to vanish. So the {\em critical $T$-curvature
  pair} of Theorem~\ref{T-thm-1} provides a version of the objects
sought at that point in \cite{JB}. (See also Chapter~6 of that work.)

The tractor families ``$D^{\ct}_K(M,\Sigma;g;\lambda)$'' of Definition
6.21.1 and Definition 6.21.2 in \cite{JB} are the operators denoted
``$\delta_K$'' in Theorem 5.1 of \cite{BrGonon}. In comparing with our
discussion above, the reader should note that in \cite{JB}, $n$
denotes $\dim(\Sigma)$, whereas here $n$ is $\dim(M)=\dim(\Sigma)+1$.
Conjecture 1.10.1 of \cite{JB} (which is stated in more detail as
Conjecture 6.21.3 of the same source) proposes that for case of
$\dim(\Sigma)$ even, these operators admit a certain decomposition
into an intrinsic Laplacian power part plus another part, termed an
``extrinsic'' part, and that this occurs in a manner compatible with
corresponding $Q$-curvatures. As stated, that conjecture relies on
Conjecture 6.21.1 of \cite{JB}, which asserts that for $\dim(\Sigma)$
even, $\deltaBGK$ is the zero operator on densities of
weight $w=0$.

We do not see evidence that Conjecture 6.21.1 should hold in the
curved setting if $\dim(\Sigma)$ is (even and) greater than
4. However, {\em any} of the operators $\deltaJk$, with $J+k=
\dim(\Sigma)$, or either of the operators $\delta_K$ or
$\ConFlatOp{K}$, with $K=\dim(\sm)$, provides a replacement for the
operator $D_{\dim(\Sigma)}^T (\cdot,\Sigma;g;0)$
sought in that conjecture, as $\delta_K$, $\ConFlatOp{K}$, and
$\deltaJk$ are all regular at weight 0.  One can use $\delta_K$,
$\ConFlatOp{K}$, and $\deltaJk$ to show that a variant of Juhl's
Conjecture~6.21.3 (or 1.10.1) \textit{does} in fact hold.  To see
this, begin by supposing that $n\geq 5$ and that $n$ is odd.  Let
$\iGJMSden_{n-1}$ denote the intrinsic order-$(n-1)$ GJMS operator on
$\sm$ mapping $\ceS[0]$ to $\ceS[1-n]$.  Let $\iBranQ_{n-1}(\ig)$
denote the Branson $Q$-curvature of $\sm$ associated to the metric
$\ig$ on $\sm$ induced by a representative $\g$ of the conformal
structure on $M$.  View $\iBranQ_{n-1}(\ig)$ as a density of weight
$1-n$.  For a conformal metric $\hatg=e^{2\Upsilon}\g$ on $M$,
\begin{equation}\label{BranTran}
\iBranQ_{n-1}(\,\hatig\,)
=\iBranQ_{n-1}(\ig)+(-1)^{(n-1)/2}\iGJMSden_{n-1}\Upsilon~,
\end{equation}
by (1.13) of \cite{Brsharp}.  The factor $(-1)^{(n-1)/2}$ in
\nn{BranTran} occurs because Branson's sign convention for $\Delta$
differs from ours.

With these preparations complete, we obtain at once the promised
variant of Juhl's Conjecture~6.21.3.
\begin{theorem}\label{Juhl-6.21.3}
Suppose
that $\dim(M)=n\geq 5$, and suppose that $n$ is odd.  Let
$\alpha\in\Reals$ be given, and let $K=n-1$.  Let
$\OurOp:\ce[w]\rightarrow\ice[w-K]$ denote the operator $\delta_K$ of
Lemma~\ref{basicversion}, as acting on densities, the operator
$\ConFlatOp{K}$ of Theorem~\ref{cflatkey}, or the operator $\deltaJk$
of Theorem~\ref{key}.  (In the latter case, $J+k$ must equal $K$.)
Then there exist natural conformally invariant differential operators
$\JuhlPi_{n-1}:\ceS[0]\rightarrow\ceS[1-n]$ and
$\JuhlPe_{n-1}:\ce[0]\rightarrow\ceS[1-n]$ and natural sections
$\JuhlQi_{n-1}(\ig)$ and $\JuhlQe_{n-1}(\g)$ of $\ceS[1-n]$ and
$\left.\ce[1-n]\rule{0mm}{2.12ex}\right|_{\sm}$, respectively, with
the following properties.  First, for the operator
$\OurOp:\ce[0]\rightarrow\ice[1-n]$, we have
$\alpha\OurOp=\JuhlPi_{n-1}+\JuhlPe_{n-1}$.  Also,
\[
(-1)^{(n-1)/2}Q_{\g}(\alpha\OurOp)=
\JuhlQi_{n-1}(\ig)+\JuhlQe_{n-1}(\g)~.
\]
Here $\JuhlPi_{n-1}$ has leading term $\ila^{(n-1)/2}$, and
$Q_{\g}(\alpha\OurOp)$ is as in Definition~\ref{generalQ}.  We view
$Q_{\g}(\alpha\OurOp)$ as a section of
$\left.\ce[1-n]\rule{0mm}{2.12ex}\right|_{\sm}$.  Finally, suppose
that $\hatg=e^{2\Upsilon}\g$.  Then
\begin{equation}\label{JuhlQtrani}
\JuhlQi_{n-1}(\,\hatig\,)=
\JuhlQi_{n-1}(\ig)+(-1)^{(n-1)/2}\JuhlPi_{n-1}\Upsilon~,
\end{equation}
and
\begin{equation}\label{JuhlQtrane}
\JuhlQe_{n-1}(\hatg)=\JuhlQe_{n-1}(\g)+(-1)^{(n-1)/2}\JuhlPe_{n-1}\Upsilon~.
\end{equation}
\end{theorem}
\begin{remark}\label{Juhl-6.21.3-Rem}
We do not claim that $\OurOp$, $\JuhlPi_{n-1}$, $\JuhlPe_{n-1}$,
$\JuhlQi_{n-1}(\ig)$, or $\JuhlQe_{n-1}(\g)$ is uniquely determined,
but see Remark~\ref{GJMSTracRem}.
\end{remark}
\noindent\textit{Proof of Theorem~\ref{Juhl-6.21.3}.}
Let $\iGJMSden_{n-1}$ and $\iBranQ_{n-1}(\ig)$ be as above.  Let
$\JuhlPe_{n-1}=\alpha\OurOp-\iGJMSden_{n-1}$, and let
$\JuhlPi_{n-1}=\iGJMSden_{n-1}$.  Let
$\JuhlQi_{n-1}(\ig)=\iBranQ_{n-1}(\ig)$, and let
\[
\JuhlQe_{n-1}(\g)=(-1)^{(n-1)/2}Q_{\g}(\alpha\OurOp)-\JuhlQi_{n-1}(\ig)~.
\]
The transformation rules given in \nn{JuhlQtrani} and \nn{JuhlQtrane}
then follow from \nn{BranTran} and Proposition~\ref{genQtran}.
\hspace{\fill}\QED

As stated Juhl's Conjecture~6.21.3 allows for the freedom to choose
$\alpha$ in Theorem~\ref{Juhl-6.21.3}. However it is natural to fix
this to the special value $\alpha=\alpha_0\neq 0$ so that when
specialised to the case that $(M,g_+)$ is the Poincar\'e ball metric,
with $\Sigma$ the boundary at infinity (i.e.\ the conformally flat
model for Poincar\'e-Einstein geometry), then $\JuhlPe_{n-1}$ is zero.
For the case of $\psi=\delta_K$ the existence of $\alpha_0$ follows
directly from the discussion surrounding \eqref{tan-GJMS}. More
generally it follows from the uniqueness of the intertwinors in this
case.

Finally we mention that the role of
extrinsic geometry in generalised notions of $Q$-curvature (and
related ``$T$-curvatures'' as associated transgression curvatures) is
explored in \cite{GW-volume,GW-volume-bd}.

%
%

\appendix

\section{Proof of Proposition~\ref{Embedding}}\label{EmbedApp}
All of the expressions in this proof will carry a weight of zero.  Let
$d=m'-m$, and let $M'=M\times\Reals^{d}$.  Also let
$\sm'=\sm\times\Reals^{d}$.  Let $\g_{E}$ denote the Euclidean metric
on $\Reals^{d}$, and let $\amr$ denote the product metric
$\g\oplus\g_{E}$ on $M'$.  One key to our proof is the fact that the
operator $\psi$ is given by the same universal symbolic formula
on $M$ and on $M'$.  Our strategy will be to relate expressions on
$(M',\amr)$ to their counterparts on $(M,\g)$.

Let $(y^1,\ldots,y^d)$ denote the standard coordinates on
$\Reals^{d}$.  Let $U$ be any open subset of $M$ on which a local
coordinate system $(x^1,\ldots,x^m)$ is defined, and suppose there is
a local defining function $t$ for $\sm$ on $U$.  Note that
$(z^1,\ldots,z^{m'})=(x^1,\ldots,x^m,y^1,\ldots,y^d)$ is a local
coordinate system for $M'$ on $U\times\Reals^{d}$.  Let
$\Gamma_{ab}{}^c$ and $\Gamma'_{ab}{}^{c}$ denote the Christoffel
symbols of $\g$ and $\amr$ in the $(x^1,\dots,x^m)$ and
$(z^1,\ldots,z^{m'})$ coordinate systems, respectively.

Figure~\ref{CompMatrices}
%
%
\begin{figure}
  \[
[\amr_{ab}]=
\left[
\begin{array}{c|c}
  \g_{ab}|_p&O_{m\times d}\\
  \hline
  O_{d\times m}&I_{d\times d}
\end{array}
\right]
,\hspace{4ex}
[\amr^{ab}]=
\left[
\begin{array}{c|c}
  \g^{ab}|_p&O_{m\times d}\\
  \hline
  O_{d\times m}&I_{d\times d}
\end{array}
\right]
\]
\caption{The component matrices of $\amr$ and its inverse at a point
  $(p,q)\in M\times\Reals^d$}
\label{CompMatrices}
\end{figure}
%
%
gives the components of $\amr$ and its inverse in the
$(z^1,\ldots,z^{m'})$ coordinate system at a point $(p,q)\in M'$.  In
this figure, $\g_{ab}|_p$ and $\g^{ab}|_{p}$ denote the $m\times m$
component matrices of $\g$ and its inverse, respectively, at $p$,
$I_{d\times d}$ is the $d\times d$ identity matrix, and $O_{m\times
  d}$ and $O_{d\times m}$ are the zero matrices of the indicated
dimensions.

Note that $\Gamma_{ab}'{}^{c}= \frac{1}{2}\amr^{cd}
(\partial_a\amr_{bd}+\partial_b\amr_{ad}-\partial_d\amr_{ab})$.  Thus
if $\{a,b,c\}$ is \textit{not} a subset of $\{1,\ldots,m\}$, one can
use the formulae in Figure~\ref{CompMatrices} to show that
$\Gamma'_{ab}{}^{c}=0$ on $U\times\Reals^{d}$.  On the other hand, if
$\{a,b,c\}$ \textit{is} a subset of $\{1,\ldots,m\}$, then for all
$(p,q)\in U\times\Reals^{d}$, we have
$\Gamma'_{ab}{}^{c}|_{(p,q)}=\Gamma_{ab}{}^{c}|_{p}$.  For all
$\{a,b,c\}\subseteq\{1,\ldots,m'\}$, $\Gamma'_{ab}{}^{c}$ is
independent of $z^{m+1}$, \ldots, and $z^{m'}$.

Let $(R')_{ab}{}^{c}{}_{d}$ denote the components of the Riemannian
curvature tensor of $\amr$.  Then
\[
(R')_{ab}{}^{c}{}_{d}
=
 \frac{\partial}{\partial z^a}\Gamma'_{bd}{}^{c}
-\frac{\partial}{\partial z^b}\Gamma'_{ad}{}^{c}
+\Gamma'_{bd}{}^e\Gamma'_{ae}{}^{c}
-\Gamma'_{ad}{}^e\Gamma'_{be}{}^{c}~.
\]
Now let $(p,q)\in U\times\Reals^{d}$ be given.  If
$\{a,b,c,d\}\subseteq\{1,\ldots,m\}$, then
$(R')_{ab}{}^{c}{}_{d}|_{(p,q)}=R_{ab}{}^{c}{}_{d}|_{p}$.
Otherwise, $(R')_{ab}{}^{c}{}_{d}|_{(p,q)}=0$.

We will now construct a unit conormal field for $\sm'$.  Let
$\pi:U\times\Reals^{d}\rightarrow U$ be the natural projection, and
let $t'$ denote the lift (pullback) of $t$ along $\pi$.  Consider any
given $(p,q)\in U\times\Reals^d$.  Note that
$(dt')|_{(p,q)}=\left.\frac{\partial t}{\partial
  x^i}\right|_{p}dx^i|_{(p,q)}$ and that $t'(p,q)=0$ if and only if
$p\in\sm$.  Thus $t'$ is a local defining function for $\sm'$ on
$U\times\Reals^{d}$.  The unit conormal field for $\sm'$ on
$\sm'\cap(U\times\Reals^d)$ is given by
$\Nv'_{a}=\frac{dt'}{|dt'|_{r}}$.  If we consider
Figure~\ref{CompMatrices}, we find that $|dt'|_{r}$, evaluated at
$(p,q)$, is equal to $|dt|_{\g}$, evaluated at $p$.  Now let
$a\in\{1,\ldots,m'\}$ be given, and let $\Nv'_{a}$ denote the
component of the unit conormal field for $\sm'$ corresponding to $a$.
If $m+1\leq a\leq m'$, then $\Nv'_a|_{(p,q)}=0$.  If $1\leq a\leq m$,
then $\Nv'_a|_{(p,q)}=\Nv_{a}|_{p}$; here $\Nv_a|_p$ denotes the
component of $\Nv_a$ at $p$.

Let $\mmc'$ denote the modified mean curvature associated to $M'$ and
$\sm'$, as defined in Section~\ref{RH}, above.  Also let $\nd'$ denote
the Levi-Civita connection of $\amr$ on $M'$.
%
\newc{\makenuprime}{\nu\makebox[0mm]{\hspace{0.8ex}${}'$}}%
%
Then $\mmc'=\amr^{ab}\nd'_a\Nv'_b$ on $U\times\Reals^{d}$, and for any
$(p,q)\in U\times\Reals^{d}$, we have
$\mmc'|_{(p,q)}= \amr^{ab}((\partial/\partial z^a)\Nv'_b -
\Gamma_{ab}'{}^c\Nv'_c)|_{(p,q)}$.
We may assume here that $a$, $b$, and $c$ are summed from $1$ to $m$
only.  Thus
\[
  \left.\rule{0mm}{2.5ex}\mmc'\right|_{(p,q)}
  =\g^{ab}\left.\left(\frac{\partial}{\partial
    x^a}\Nv_b
  -
  \rule{0mm}{2.5ex}
  \Gamma_{ab}{}^c\Nv_c\right)\right|_p=\mmc|_p~.
\]

Next, let $V'$ denote the lift of $V$ along $\pi$.  Also let
$\beta\in\Integers_{>0}$ and $(p,q)\in U\times\Reals^{d}$ be given,
and consider the following component expressions:
\begin{equation}\label{NabNabExp}
\left.(\nd'_{a_1}\ldots\nd'_{a_{\beta}}V')\right|_{(p,q)}
\mbox{\ \ and\ \ }
  \left.(\nd'_{a_1}\ldots
  \nd'_{a_{\beta}}(R')_{ij}{}^{k}{}_{\ell})\right|_{(p,q)}~.
\end{equation}
We may symbolically expand both of these and express them in terms of
the Christoffel symbols $\Gamma'_{ij}{}^{k}$ and partial derivatives
of $V'$, $\Gamma'_{ij}{}^{k}$ and $(R')_{ij}{}^{k}{}_{\ell}$.  (See
Section~2.5 of \cite{SyngeSchild}.)  Let $\mathcal{P}_1$ and
$\mathcal{P}_{2}$ denote the expressions that result from the
expansion of the respective expressions in \nn{NabNabExp} in this way.
Note that $V'$, $\Gamma'_{ij}{}^{k}$, and $(R')_{ij}{}^{k}{}_{\ell}$
are independent of $q$.  Also recall that
$\Gamma'_{ab}{}^{c}|_{(p,q)}$ vanishes whenever $\{a,b,c\}$ is not a
subset of $\{1,\ldots,m\}$.  Similarly, $(R')_{ab}{}^c{}_d|_{(p,q)}$
vanishes if $\{a,b,c,d\}$ is not a subset of $\{1,\ldots,m\}$.  We may
thus assume that all indices in $\mathcal{P}_1$ and $\mathcal{P}_2$
take values in $\{1,\ldots,m\}$ only.

Now let $p\in\sm$ and $q\in\Reals^{d}$ be given.  We may assume that
$p\in U$.  We claim that $(\psi V)(p)=(\psi V')(p,q)$.  To see this,
begin by expanding the symbolic formula for $(\psi V')(p,q)$ in the
way in which we expanded the two expressions in \nn{NabNabExp}, above.
In the resulting expansion, we may assume that all indices are summed
from $1$ to $m$ only.  Next, replace all occurrences of
$(R')_{ab}{}^{c}{}_{d}|_{(p,q)}$, $\Gamma_{ab}'{}^{c}|_{(p,q)}$,
$\mmc'|_{(p,q)}$, $\Nv_a'|_{(p,q)}$, $\amr_{ab}|_{(p,q)}$,
$\amr^{ab}|_{(p,q)}$, $V'(p,q)$, and $\frac{\partial}{\partial
  z^i}|_{(p,q)}$ with $R_{ab}{}^{c}{}_{d}|_{p}$,
$\Gamma_{ab}{}^{c}|_{p}$, $\mmc|_{p}$, $\Nv_a|_{p}$, $\g_{ab}|_p$,
$\g^{ab}|_p$, $V(p)$, and $\frac{\partial}{\partial x^i}|_{p}$,
respectively.  The resulting expression is $(\psi V)(p)$.  If we let
$p'=(p,q)$, then $p'$ satisfies \nn{EmbeddingRelation}.
\hspace{\fill}\QED
%

%
%

%
\end{document}